\documentclass[a4paper,12pt]{amsart}
\usepackage{amssymb,amsmath,amsfonts,amsthm,accents}
\usepackage[hmargin=1in, vmargin=1.2in]{geometry}

\newcommand*{\uwidehat}[1]{\underaccent{\widehat{\hphantom{#1}}}{#1}}

\newcommand{\ds}{\displaystyle}
\newcommand{\NN}{\mathbb N}

\newcommand{\CC}{\mathbb C}
\newcommand{\RR}{\mathbb R}
\newcommand{\ZZ}{\mathbb Z}
\newcommand{\EE}{\mathcal E}
\newcommand{\DD}{\mathcal D}
\newcommand{\SSS}{\mathcal S}

\newcommand{\ssum}{\mbox{$\sum_j\,$}}

\newcommand{\beq}{\begin{eqnarray}}
\newcommand{\eeq}{\end{eqnarray}}

\newcommand{\beqs}{\begin{eqnarray*}}
\newcommand{\eeqs}{\end{eqnarray*}}

\newcommand{\Op}{\mathrm{Op}}

\newtheorem{theorem}{Theorem}[section]
\newtheorem{proposition}[theorem]{Proposition}
\newtheorem{lemma}[theorem]{Lemma}
\newtheorem{corollary}[theorem]{Corollary}

\theoremstyle{remark}
\newtheorem{remark}[theorem]{Remark}

\theoremstyle{definition}
\newtheorem{definition}[theorem]{Definition}

\allowdisplaybreaks

\author[S. Pilipovi\' c]{Stevan Pilipovi\' c}
\address{Department of Mathematics and Informatics,
University of Novi Sad, Trg Dositeja Obradovi\'{c}a 4, 21000 Novi Sad, Serbia}
\email{stevan.pilipovic@dmi.uns.ac.rs}

\author[B. Prangoski]{Bojan Prangoski}
\address{Department of Mathematics, Faculty of Mechanical
Engineering-Skopje, Karposh 2 b.b., 1000 Skopje, Macedonia}
\email{bprangoski@yahoo.com}

\title[Complex powers of infinite order operators]{Complex powers for a class of infinite order hypoelliptic operators}



\begin{document}

\begin{abstract}
We prove the complex powers of a class of infinite order hypoelliptic pseudodifferential operators can always be represented as hypoelliptic pseudodifferential operators modulo ultrasmoothing operators. We apply this result to the study of semigroups generated by square roots of non-negative hypoelliptic infinite order operators. For this purpose, we derive precise estimates of the corresponding heat kernel.
\end{abstract}
\maketitle

\section{Introduction}

Complex powers of  pseudo-differential operators were studied by many distinguished  mathematicians; we mention here only some of their papers and books, \cite{ABP}, \cite{Seeley1}, \cite{Kumanogo}, \cite{Dustermaat},
\cite{Helffer}, \cite{DV}, \cite{Schrohe1}, \cite{Shubin}, \cite{Melrose},  \cite{Loya1},  \cite{Loya2}, \cite{CSS}. Their investigations are related to the index theory, Weyl counting function, regularity properties of solutions and to several other topics in PDE.  H\" ormander's books \cite{Hor} based on his papers and related to his theory of pseudodifferential and Fourier integral operators and the generalisations over manifolds, were the framework of the quoted investigations. Recently Buzano and Nicola \cite{BN} (see also \cite{NR}) studied complex powers of global hypoelliptic pseudodifferential operators over $\RR^d$, whose Weyl symbols belong to H\" ormander's class $S(m,g)$ associated with a tempered weight function $m$ and a slowly varying Riemann metric $g$. They provide interesting applications in the classification of Schatten classes and in the semigroup theory with infinitesimal generators being fractional powers of globally defined symbols.\\
\indent In this article we are analysing a new class of global Shubin type symbols (over $\RR^d$) for which we have had to develop a completely new calculus; the asymptotic behaviour of the symbols is sharply estimated by constants dependent on the derivatives and power functions realised by Gevrey-type sequences, i.e. they are ultradifferential with sub-exponential growth (see \cite{CPP}, \cite{CPP1}, \cite{BojanP}). Hypoellipticity in this setting allows for symbols to decay sub-exponentially at infinity. The prototypical example of such hypoelliptic symbols is realised by $e^{\pm a(x,\xi)^{1/(ms)}}$ for large enough $s>1$, where $a$ is particular globally elliptic and positive Shubin symbol of order $m\geq 1$ (see Remark \ref{kft551133} below; see also \cite{CPP1} for other non-trivial examples). Needless to say, the notion of hypoellipticity and symbols in these classes go beyond the classical Weyl-H\" ormander calculus. Although we are considering problems with already determined steps of proofs, the realisation of these ideas is rather involved. In the infinite order setting, the change of quantisation and the composition formulas always result in additional (ultra)smoothing operators and these require different techniques (often of functional analytic nature) to handle them. As a matter of fact, a part of this article is devoted to proving that the resulting (ultra)smoothing operators are ``well-behaved'' when one applies the symbolic calculus to symbols varying in bounded sets in the symbol classes. This brings a complicated calculus, especially in the estimates of the complex powers as well as in recurrence formulas for the estimates of the derivatives of the heat kernels. Our new technique involves a deep analysis of topological structures related to ultradistribution type spaces and hypoelliptic operators. Even the application in the semigroup analysis involves several new general results which are of independent interest.\\
\indent To motivate our investigations, we note that in \cite{ppv,ppv1} we studied the asymptotic behaviour of the eigenvalue counting function of a class of infinite order $\Psi$DOs. One of the main building blocks for the development of the theory are the results presented in this article; most notably the precise estimates we obtain here on the symbol of the heat parametrix. Furthermore, in \cite{ppv1} we gave interesting examples of infinite order hypoelliptic operators, namely perturbations of smaller order of power series of a classical elliptic Shubin differential operator $a^{w}=\sum_{|\alpha|+|\beta|\leq m} a_{\alpha,\beta}x^\alpha D^\beta$ on $\mathbb{R}^{d}$. More precisely, these are operators of the form
\beqs
P(a^w)+\mbox{``smaller order terms''}=\sum_{n=0}^{\infty} c_n(a^w)^n+\mbox{``smaller order terms''},
\eeqs
where $P$ is an entire function with positive Taylor coefficients $c_n$, $n\in\NN$, having suitable growth order. In \cite{ppv1} we proved that under suitable conditions on $P$, these are indeed hypoelliptic operators in our setting. Of course, the most interesting case is when $a^w=H=|x|^2-\Delta$ is the Harmonic oscillator in which case the operator is of the form $\sum_n c_n H^n$+``smaller order terms''. Also, we mention here our results concerning the hypoellipticity for certain classes of linear and non-linear problems in this context given in \cite{CPP} and \cite{CPP1}.\\
\indent We give a brief outline of our work.\\
\indent In Section \ref{section2}, we recall the definition and some basic facts concerning the symbol classes, denoted here as $\Gamma$-type classes, and the corresponding pseudodifferential operators of infinite order involved in the sequel; we refer to \cite{BojanP} and \cite{CPP} for the complete theory concerning the symbolic calculus that these operators enjoy. To these symbol classes we attach spaces of formal series (asymptotic expansions) which we denote as $FS$-type space of formal series. We introduce in Subsection \ref{sec_ch_qua} certain relations between bounded subsets of the symbol classes and their asymptotic expansions in order to prove that the ambiguity in constructing bounded sets of symbols out of bounded sets of asymptotic expansions is always given by a ``well-behaved'' set of ultrasmoothing operators. Section \ref{section3} is related to the Weyl calculus and the twisted product $\#$ (also referred to as the sharp product) of symbols and the composition of corresponding Weyl operators. For us, it will be particularly important that the spaces of asymptotic expansions have a ring structure with multiplication given by $\#$, which is in fact hypocontinuous with respect to the natural locally convex topology that the $FS$-spaces have. In Section \ref{section4} we recall from \cite{CPP} the construction of the parametrix for a hypoelliptic symbol. We also prove an important result concerning the minimal and maximal realisations of hypoelliptic symbols: they coincide. This fact is well-known in the classical Weyl-H\" ormander calculus, here we proved that it is also valid in our infinite order setting (for example, this result applies to the power series of the Harmonic oscillator discussed above).\\
\indent After these three preparatory sections, Section \ref{section5} is devoted to the complex powers of hypoelliptic operators. We follow the approach of \cite{Fractionalpowersbook,KNonnegative} and state the main result, Theorem \ref{maint}: for a hypoelliptic symbol $a$ satisfying certain conditions, the complex power $\overline{A}^z$ (where $A$ is the closure of the $L^2(\RR^d)$-realisation of the Weyl operator with symbol $a$) is always pseudodifferential with hypoelliptic symbol $a^{\uwidehat{z}}$ modulo ultrasmoothing operator $S^{\uwidehat{z}}$. In the first four parts of the theorem we give the existence and the properties of the asymptotic expansion of the hypoelliptic symbol $a^{\uwidehat{z}}$. Part $(v)$ contains uniform estimates of $a^{\uwidehat{z}}$ with respect to $z$ lying in vertical strips $\{\zeta\in\CC|\, 0<\mathrm{Re}\, \zeta\leq t\}$ and with respect to the derivatives $D^\alpha_w a^{\uwidehat{z}}(w)$. The last, sixth, part of the theorem shows that the domain of the unbounded operator $\overline{A}^z$ is in fact $\big\{v\in L^2(\RR^d)|\, (a^{\uwidehat{z}})^w v\in L^2(\RR^d)\big\}$ and proves the analyticity of the mappings $z\mapsto (a^{\uwidehat{z}})^w v$ and $z\mapsto S^{\uwidehat{z}}v$ (with values in $L^2(\RR^d)$) on vertical strips for each $v\in D(A^{[t]+1})$. The analyticity in the classical setting is obtained for free (from the general theory of complex powers of operators on Banach spaces), but here it is a difficult task because of the additional ultrasmoothing operators $S^{\uwidehat{z}}$. The major part of the article is devoted to the proof of this theorem: it is the content of this entire section.\\
\indent In Section \ref{section6} we apply the theory of complex powers to a semigroup generated by the square root of a non-negative operator given as an $L^2(\RR^d)$-realisation of a hypoelliptic symbol in our $\Gamma$-class. We prove that the semigroup is comprised of pseudodifferential operators with symbols in the $\Gamma$-class. As it turns out, the heat parametrix plays an essential role in our analysis. With our technique, we derive precise estimates on the heat kernel which are of independent interest because, as we mentioned before, we applied them in \cite{ppv,ppv1} to derive asymptotic formulae for the eigenvalue counting function and the analysis of spectral properties of hypoelliptic operators of infinite order. In this context, the result on the complex powers of pseudodifferential operators of infinite order of our class also looks promising for applications in deriving such asymptotic formulae. Furthermore, let us mention that with our semigroup related to $A$ as described in the last section, the unique bounded solution of $u_{tt}-Au=0$, $u(0)\in L^2(\RR^d)$, can be given as an action of a pseudodifferential operators plus a smooth family of ultrasmoothing operators on $u(0)$ (cf. \cite{BN}, \cite[Theorem 6.3.2, p. 165]{Fractionalpowersbook}). This can be particularly applied when $A$ is the closure of $\sum_n c_n H^n$, the operator we mentioned above (see Theorem \ref{res_sqr_rtt1} and \cite[Corollary 3.5]{ppv1}); the operator is non-negative since it is self-adjoint (cf. \cite[Proposition 4.6]{ppv}) and $(Au,u)\geq 0$, $u\in D(A)$ (see \cite[Proposition 1.3.6, p. 21]{Fractionalpowersbook}).\\
\indent In the end, we remark that the theory looks promising to be able to include systems of $\Psi$DOs, that is matrix valued infinite order pseudodifferential operators. We postpone this natural generalisation for future research; it ought to be straightforward (but technical) procedure.

\section{Preliminaries}

The sets of natural, integer, positive integer, real and complex numbers are denoted by $\NN$, $\ZZ$, $\ZZ_+$, $\RR$, $\CC$. The symbol $\RR_+$ stands for the set of positive real numbers and $\CC_+$ for the complex numbers with positive real part, i.e. $\CC_+=\{z\in\CC|\, \mathrm{Re}\,z>0\}$. For $x\in \RR^d$ and $\alpha\in\NN^d$, we use the following notation: $\langle x\rangle =(1+|x|^2)^{1/2}$ and
$D^{\alpha}= D_1^{\alpha_1}\ldots D_d^{\alpha_d}$, where $D_j^
{\alpha_j}={i^{-\alpha_j}}\partial^{\alpha_j}/{\partial x_j}^{\alpha_j}$.\\
\indent Following \cite{Komatsu1}, we denote by $M_{p}$, $p\in\NN$, a sequence of positive numbers such that $M_0=M_1=1$ and satisfies some of the following conditions:
\begin{itemize}
\item[${}$] $(M.1)$ $M_{p}^{2} \leq M_{p-1} M_{p+1}$, $p \in\ZZ_+$;
\item[${}$] $(M.2)$ $M_{p} \leq c_0H^{p} \min_{0\leq q\leq p} \{M_{p-q} M_{q}\}$, $p,q\in \NN$, for some $c_0,H\geq1$;
\item[${}$] $(M.3)$ $\sum^{\infty}_{p=q+1} M_{p-1}/M_{p}\leq c_0q M_{q}/M_{q+1}$, $q\in \ZZ_+$;
\item[${}$] $(M.3)'$ $\sum_{p=1}^{\infty}M_{p-1}/M_p<\infty$;
\item[${}$] $(M.4)$ $M_p^2/p!^2\leq (M_{p-1}/(p-1)!)\cdot (M_{p+1}/(p+1)!)$, $p\in\ZZ_+$.
\end{itemize}
Observe that $(M.4)$ implies $(M.1)$ and $(M.3)$ implies $(M.3)'$. The sequence $M_p=p!^{\sigma}$, $\sigma>1$, satisfies all conditions listed above. For a multi-index $\alpha\in\NN^d$, $M_{\alpha}$ will mean $M_{|\alpha|}$, $|\alpha|=\alpha_1+...+\alpha_d$. Recall (see \cite[Section 3]{Komatsu1}) that $m_p=M_p/M_{p-1}$, $p\in\ZZ_+$, and if $M_p$ satisfies $(M.1)$ and $M_p/C^p\rightarrow \infty$, for any $C>0$ (which obviously holds when $M_p$ satisfies $(M.3)'$), its associated function is defined by $M(\rho)=\sup  _{p\in\NN}\ln_+ \rho^{p}/M_{p}$, $\rho > 0$. It is a non-negative, continuous, monotonically increasing function, which vanishes for sufficiently small $\rho>0$ and increases more rapidly than $\ln \rho^p$ when $\rho$ tends to infinity, for any $p\in\NN$. When $M_p=p!^{\sigma}$, $\sigma>0$, $M(\rho)\asymp \rho^{1/\sigma}$.\\
\indent Let $U$ be an open subset of $\RR^d$ and $K$ a regular compact subset of $U$. For $h>0$, $\EE^{\{M_p\},h}(K)$ is the Banach space (from now on abbreviated as $(B)$-space) of all $\varphi\in \mathcal{C}^{\infty}(\mathrm{int}\, K)$ such that all of their derivatives $D^{\alpha}\varphi$, $\alpha\in\NN^d$, extend to continuous functions on $K$ and satisfy $\sup_{\alpha\in\NN^d}\sup_{x\in K}|D^{\alpha}\varphi(x)|/(h^{\alpha}M_{\alpha})<\infty$. Furthermore, $\DD^{\{M_p\},h}_K$ denotes its subspace of all smooth functions supported by $K$. Following Komatsu \cite{Komatsu1}, we define as locally convex spaces (from now on abbreviated as l.c.s.)
$$
\EE^{(M_p)}(U)=\lim_{\substack{\longleftarrow\\ K\subset\subset U}}\lim_{\substack{\longleftarrow\\ h\rightarrow 0}} \EE^{\{M_p\},h}(K),\,\,\,\,
\EE^{\{M_p\}}(U)=\lim_{\substack{\longleftarrow\\ K\subset\subset U}}
\lim_{\substack{\longrightarrow\\ h\rightarrow \infty}} \EE^{\{M_p\},h}(K),
$$
\beqs
\DD^{(M_p)}(U)=\lim_{\substack{\longrightarrow\\ K\subset\subset U}}\lim_{\substack{\longleftarrow\\ h\rightarrow 0}} \DD^{\{M_p\},h}_K,\,\,\,\,
\DD^{\{M_p\}}(U)=\lim_{\substack{\longrightarrow\\ K\subset\subset U}}\lim_{\substack{\longrightarrow\\ h\rightarrow \infty}} \DD^{\{M_p\},h}_K.
\eeqs
The spaces of ultradistributions and ultradistributions with compact support of Beurling and Roumieu type are defined as the strong duals of $\DD^{(M_p)}(U)$ and $\EE^{(M_p)}(U)$, resp. $\DD^{\{M_p\}}(U)$ and $\EE^{\{M_p\}}(U)$. For the properties of these spaces, we refer to \cite{Komatsu1,Komatsu2,Komatsu3}. The common notation for the symbols $(M_{p})$ and $\{M_{p}\} $ will be $*$.\\
\indent We denote by $\mathfrak{R}$  the set of all positive sequences which monotonically increase to infinity. There is a natural order on $\mathfrak{R}$ defined by $(r_p)\leq (k_p)$ if $r_p\leq k_p$, $\forall p\in\ZZ_+$, and with it $(\mathfrak{R},\leq)$ becomes a directed set.\\
\indent For $(r_p)\in\mathfrak{R}$, consider the sequence $N_0=1$, $N_p=M_p\prod_{j=1}^{p}r_j$, $p\in\ZZ_+$. It is easy to check that this sequence satisfies $(M.1)$ and $(M.3)'$ when $M_p$ does so and its associated function will be denoted by $N_{r_p}(\rho)$, i.e. $N_{r_{p}}(\rho)=\sup_{p\in\NN} \ln_+ \rho^{p}/(M_p\prod_{j=1}^{p}r_j)$,\footnote{here and throughout the rest of the article we use the principle of vacuous (empty) product, i.e. $\prod_{j=1}^0r_j=1$} $\rho > 0$. Note that for $(r_{p})\in\mathfrak{R}$ and $k > 0 $ there is $\rho _{0} > 0$ such that $N_{r_{p}} (\rho ) \leq M(k \rho )$, $\forall \rho > \rho _{0}$.\\
\indent A measurable function $f$ on $\RR^d$ is said to have ultrapolynomial growth of class $(M_p)$ (resp. of class $\{M_p\}$) if $\|e^{-M(h|\cdot|)}f\|_{L^{\infty}(\RR^d)}<\infty$ for some $h>0$ (resp. for every $h>0$). By the same technique as in the proof of \cite[Theorem 4.2 $(\tilde{c})$]{DPPV} we have the following lemma (in fact, the same proof works fine in our case as well).

\begin{lemma}\label{lemulgr117}
Let $B\subseteq \mathcal{C}(\RR^d)$. The following conditions are equivalent.
\begin{itemize}
\item[$(i)$] For every $h>0$ there exists $C>0$ such that $|f(x)|\leq Ce^{M(h|x|)}$, for all $x\in\RR^d$, $f\in B$.
\item[$(ii)$] There exist $(r_p)\in\mathfrak{R}$ and $C>0$ such that $|f(x)|\leq Ce^{N_{r_p}(|x|)}$, for all $x\in\RR^d$, $f\in B$.
\end{itemize}
\end{lemma}

This simple result will prove useful throughout the rest of the article and we will often tacitly apply it.\\
\indent We call an entire function $P(z) =\sum _{\alpha \in \NN^d}c_{\alpha } z^{\alpha}$, $z \in \CC^d$, an ultrapolynomial of class $(M_{p})$ (resp. of class $\{M_{p}\}$), whenever the coefficients $c_{\alpha }$ satisfy the estimate $|c_{\alpha }|  \leq C L^{|\alpha| }/M_{\alpha}$, $\alpha \in \NN^d$, for some $L > 0$ and $C>0$ (resp. for every $L > 0 $ and some $C=C(L) > 0$). The corresponding operator $P(D)=\sum_{\alpha} c_{\alpha}D^{\alpha}$ is an ultradifferential operator of the class $(M_{p})$ (resp. of class $\{M_{p}\}$) and, when $M_p$ satisfies $(M.2)$, it acts continuously on $\EE^{(M_p)}(U)$ and $\DD^{(M_p)}(U)$ (resp. on $\EE^{\{M_p\}}(U)$ and $\DD^{\{M_p\}}(U)$) and the corresponding spaces of ultradistributions.\\
\indent The Fourier transform of $f\in L^1(\RR^d)$ is given by $\mathcal{F}f (\xi)=\int_{\RR^d} e^{-i x\xi}f(x)dx$, $\xi\in\RR^d$.\\
\indent If $M_p$ satisfies $(M.1)$ and $(M.3)'$, for each $m>0$ we denote by $\SSS^{M_p,m}_{\infty}(\RR^d)$ the $(B)$-space of all $\varphi\in\mathcal{C}^{\infty}(\RR^d)$ for which the norm
\beqs
\sup_{\alpha\in \NN^d}\frac{m^{|\alpha|}\|e^{M(m|\cdot|)}D^{\alpha}\varphi\|_{L^{\infty}(\RR^d)}}{M_{\alpha}}
\eeqs
is finite. The spaces of sub-exponentially decreasing ultradifferentiable function of Beurling and Roumieu type are defined by
\beqs
\SSS^{(M_{p})}(\RR^d)=\lim_{\substack{\longleftarrow\\ m\rightarrow\infty}}\SSS^{M_{p},m}_{\infty}\left(\RR^d\right)\,\, \mbox{and}\,\, \SSS^{\{M_{p}\}}(\RR^d)=\lim_{\substack{\longrightarrow\\ m\rightarrow 0}}\SSS^{M_{p},m}_{\infty}\left(\RR^d\right),
\eeqs
respectively. Their strong duals $\SSS'^{(M_{p})}(\RR^d)$ and $\SSS'^{\{M_{p}\}}(\RR^d)$ are the spaces of tempered ultradistributions of Beurling and Roumieu type, respectively. When $M_p=p!^{\sigma}$, $\sigma>1$, $\SSS^{\{M_p\}}(\RR^d)$ is just the Gelfand-Shilov space $\SSS^{\sigma}_{\sigma}(\RR^d)$. When $M_p$ satisfies $(M.2)$, the ultradifferential operators of class $*$ act continuously on $\SSS^*(\RR^d)$ and $\SSS'^*(\RR^d)$ and the Fourier transform is a topological isomorphism on these spaces. We refer to \cite{PilipovicK,PilipovicU} for the topological properties of $\SSS^*(\RR^d)$ and $\SSS'^*(\RR^d)$. Here we recall that, when $M_p$ satisfies $(M.2)$, the space $\SSS^{\{M_p\}}(\RR^d)$ is topologically isomorphic to $\ds\lim_{\substack{\longleftarrow\\ (r_p)\in\mathfrak{R}}}\SSS^{M_p, (r_p)}_{\infty}(\RR^d)$, where the projective limit is taken with respect to the natural order on $\mathfrak{R}$ defined above and $\SSS^{M_p, (r_p)}_{\infty}(\RR^d)$ is the $(B)$-space of all $\varphi\in\mathcal{C}^{\infty}(\RR^d)$ for which the norm
\beqs
\sup_{\alpha\in \NN^d}\frac{\|e^{N_{r_p}(|\cdot|)}D^{\alpha}\varphi\|_{L^{\infty}(\RR^d)}}{M_{\alpha}\prod_{j=1}^{|\alpha|}r_j}
\eeqs
is finite.\\
\indent We end this section with a few notations from functional analysis. Given two l.c.s. $E$ and $F$, $\mathcal{L}(E,F)$ stands for the space of continuous linear mappings from $E$ to $F$. When $E=F$ we will often write $\mathcal{L}(E)$ instead of $\mathcal{L}(E,E)$. We write $\mathcal{L}_b(E,F)$ to denote the space $\mathcal{L}(E,F)$ equipped with the topology of bounded convergence and similarly $\mathcal{L}_p(E,F)$ and $\mathcal{L}_{\sigma}(E,F)$ stand for $\mathcal{L}(E,F)$ equipped with the topology of precompact convergence and simple convergence respectively. If $U$ is open in $\RR^d$, $\mathcal{C}^k(U;E)$, $0\leq k\leq \infty$, stands for the space of $k$ times continuously differentiable $E$-valued function and $\mathcal{C}^k(\overline{U};E)$ for its subspace consisting of those functions such that all of their derivatives up to order $k$ can be extended to continuous functions on $\overline{U}$. For $[a,b)\subseteq\RR$, $\mathcal{C}^{k}([a,b);E)$, $0\leq k\leq \infty$, is the vector space of all $k$ times continuously differentiable $E$-valued functions on $[a,b)$, where the derivatives at $a$ are to be understood as right derivatives; we use analogous notation when we consider functions over $(a,b]$ or $[a,b]$.

\section{Shubin type pseudodifferential operators of infinite order on $\SSS^*(\RR^d), \SSS'^*(\RR^d)$}\label{section2}

We recall in this section the definition and some basic facts concerning the symbol classes and the corresponding pseudodifferential operators of infinite order involved in the sequel; we refer to \cite{BojanP} for the complete theory concerning the symbolic calculus that these operators enjoy.\\
\indent Let $A_p$ and $M_p$ be two sequences of positive numbers such that $A_0=A_1=M_0=M_1=1$. We assume that $M_p$ satisfies $(M.1)$, $(M.2)$ and $(M.3)$ and $A_p$ satisfies $(M.1)$, $(M.2)$, $(M.3)'$ and $(M.4)$. Of course, without losing generality, we can assume that the constants $c_0$ and $H$ that appear in $(M.2)$ are the same for both sequences $M_p$ and $A_p$. Moreover, we assume that $A_p\subset M_p$, i.e. there exist $c,L>0$ such that $A_p\leq cL^pM_p$, $\forall p\in\NN$. Let $\rho_0=\inf\{\rho\in\RR_+|\,A_p\subset M_p^{\rho}\}$. Clearly $0<\rho_0\leq 1$. From now on, $\rho$ is a fixed number such that $\rho_0\leq \rho\leq1$, if the infimum can be reached, or, otherwise $\rho_0< \rho\leq1$.\\
\indent For $h,m>0$, define $\Gamma_{A_p,\rho}^{M_p,\infty}(\RR^{2d};h,m)$ to be the $(B)$-space of all $a\in \mathcal{C}^{\infty}(\RR^{2d})$ for which the following norm is finite
\beqs
\sup_{\alpha,\beta}\sup_{(x,\xi)\in\RR^{2d}}\frac{\left|D^{\alpha}_{\xi}D^{\beta}_x a(x,\xi)\right|
\langle (x,\xi)\rangle^{\rho|\alpha|+\rho|\beta|}e^{-M(m|\xi|)}e^{-M(m|x|)}} {h^{|\alpha|+|\beta|}A_{\alpha}A_{\beta}}.
\eeqs
As l.c.s., we define
\beqs
\Gamma_{A_p,\rho}^{(M_p),\infty}(\RR^{2d};m)&=& \lim_{\substack{\longleftarrow\\h\rightarrow 0}}
\Gamma_{A_p,\rho}^{M_p,\infty}(\RR^{2d};h,m),\\
\Gamma_{A_p,\rho}^{(M_p),\infty}(\RR^{2d})&=& \lim_{\substack{\longrightarrow\\m\rightarrow\infty}}
\Gamma_{A_p,\rho}^{(M_p),\infty}(\RR^{2d};m),\\
\Gamma_{A_p,\rho}^{\{M_p\},\infty}(\RR^{2d};h)&=& \lim_{\substack{\longleftarrow\\m\rightarrow 0}}
\Gamma_{A_p,\rho}^{M_p,\infty}(\RR^{2d};h,m),\\
\Gamma_{A_p,\rho}^{\{M_p\},\infty}(\RR^{2d})&=& \lim_{\substack{\longrightarrow\\h\rightarrow\infty}}
\Gamma_{A_p,\rho}^{\{M_p\},\infty}(\RR^{2d};h).
\eeqs
Then, $\Gamma_{A_p,\rho}^{(M_p),\infty}(\RR^{2d};m)$ and $\Gamma_{A_p,\rho}^{\{M_p\},\infty}(\RR^{2d};h)$ are $(F)$-spaces. The spaces $\Gamma_{A_p,\rho}^{*,\infty}(\RR^{2d})$ are barrelled and bornological. As a direct consequence of $(M.2)$ for $A_p$, for each fixed $m>0$, the norms
\beqs
\|a\|_{\Gamma,h,m}=\sup_{\alpha,\beta}\sup_{(x,\xi)\in\RR^{2d}} \frac{\left|D^{\alpha}_{\xi}D^{\beta}_x a(x,\xi)\right|
\langle (x,\xi)\rangle^{\rho|\alpha|+\rho|\beta|} e^{-M(m|\xi|)}e^{-M(m|x|)}}{h^{|\alpha|+|\beta|}A_{\alpha+\beta}},
\eeqs
when $h$ varies in $\RR_+$, generate the topology of the $(F)$-space $\Gamma_{A_p,\rho}^{(M_p),\infty}(\RR^{2d};m)$.\\
\indent Let $\tau\in\RR$ and $a\in\Gamma_{A_p,\rho}^{*,\infty}(\RR^{2d})$. The $\tau$-quantisation of the symbol $a$ is the operator $\Op_{\tau}(a):\SSS^*(\RR^d)\rightarrow \SSS'^*(\RR^d)$ defined by
\beqs
\langle \Op_{\tau}(a)u,v\rangle=\langle \mathcal{F}^{-1}_{\xi\rightarrow x-y}a((1-\tau)x+\tau y,\xi),v(x)\otimes u(y)\rangle,\,\, u,v\in\SSS^*(\RR^d).
\eeqs
We will be particularly interested in the Weyl quantisation (obtained for $\tau=1/2$) and in this case we will often write $a^w$ instead of $\Op_{1/2}(a)$. The operator $\Op_{\tau}(a)$ is continuous as an operator from $\SSS^*(\RR^d)$ into itself and $\Op_{\tau}(a)u$ is given by the following iterated integral (see \cite[Theorem 1]{BojanP}):
\beqs
\Op_{\tau}(a)u(x)=\frac{1}{(2\pi)^d}\int_{\RR^d}\int_{\RR^d}e^{i(x-y)\xi}a((1-\tau)x+\tau y,\xi) u(y)dyd\xi.
\eeqs
As a direct consequence of \cite[Theorem 2]{BojanP} and the discussion after it, for each $\tau\in\RR$, the bilinear mapping $(a,\varphi)\mapsto \Op_{\tau}(a)\varphi$, $\Gamma_{A_p,\rho}^{*,\infty}(\RR^{2d})\times \SSS^*(\RR^d)\rightarrow \SSS^*(\RR^d)$, is hypocontinuous and hence, the mapping $a\mapsto \Op_{\tau}(a)$, $\Gamma_{A_p,\rho}^{*,\infty}(\RR^{2d})\rightarrow \mathcal{L}_b(\SSS^*(\RR^d), \SSS^*(\RR^d))$, is continuous. Moreover, for each $a\in\Gamma_{A_p,\rho}^{*,\infty}(\RR^{2d})$, $\Op_{\tau}(a)$ extends to a continuous mapping from $\SSS'^*(\RR^d)$ to $\SSS'^*(\RR^d)$ and this extension is given by
\beqs
\langle \Op_{\tau}(a)T,\varphi\rangle=\langle T,\Op_{1-\tau}(a(x,-\xi))\varphi\rangle,\,\, T\in\SSS'^*(\RR^d),\, \varphi\in\SSS^*(\RR^d).
\eeqs
Since $(a,\varphi)\mapsto \Op_{\tau}(a)\varphi$ is hypocontinuous as a bilinear mapping
\beqs
\Gamma_{A_p,\rho}^{(M_p),\infty}(\RR^{2d};m)\times \SSS^{(M_p)}(\RR^d)&\rightarrow& \SSS^{(M_p)}(\RR^d)\,\,\, \mbox{and}\\
\Gamma_{A_p,\rho}^{\{M_p\},\infty}(\RR^{2d};h)\times \SSS^{\{M_p\}}(\RR^d)&\rightarrow& \SSS^{\{M_p\}}(\RR^d)
\eeqs
in the Beurling and Roumieu case respectively, we infer that for each $T\in\SSS'^*(\RR^d)$ and for each bounded subset $B$ of $\Gamma_{A_p,\rho}^{(M_p),\infty}(\RR^{2d};m)$ in the Beurling case and of $\Gamma_{A_p,\rho}^{\{M_p\},\infty}(\RR^{2d};h)$ in the Roumieu case respectively, the set $\{\Op_{\tau}(a)T|\, a\in B\}$ is bounded in $\SSS'^*(\RR^d)$. As $\Gamma_{A_p,\rho}^{(M_p),\infty}(\RR^{2d};m)$ and $\Gamma_{A_p,\rho}^{\{M_p\},\infty}(\RR^{2d};h)$ are $(F)$-spaces, $a\mapsto \Op_{\tau}(a)T$ is continuous as a mapping $\Gamma_{A_p,\rho}^{(M_p),\infty}(\RR^{2d};m)\rightarrow \SSS'^{(M_p)}(\RR^d)$ and $\Gamma_{A_p,\rho}^{\{M_p\},\infty}(\RR^{2d};h)\rightarrow \SSS'^{\{M_p\}}(\RR^d)$ in the Beurling and Roumieu case respectively. Since this holds for each $m>0$ and $h>0$ respectively, $a\mapsto \Op_{\tau}(a)T$, $\Gamma_{A_p,\rho}^{*,\infty}(\RR^{2d})\rightarrow \SSS'^*(\RR^d)$, is continuous. We deduce that $(a,T)\mapsto \Op_{\tau}(a)T$, $\Gamma_{A_p,\rho}^{*,\infty}(\RR^{2d})\times \SSS'^*(\RR^d)\rightarrow \SSS'^*(\RR^d)$, is separately continuous. As $\Gamma_{A_p,\rho}^{*,\infty}(\RR^{2d})$ and $\SSS'^*(\RR^d)$ are barrelled, this bilinear mapping is hypocontinuous. As a direct consequence, we conclude that $a\mapsto \Op_{\tau}(a)$, $\Gamma_{A_p,\rho}^{*,\infty}(\RR^{2d})\rightarrow \mathcal{L}_b(\SSS'^*(\RR^d), \SSS'^*(\RR^d))$ is continuous. Thus, we proved the following result.

\begin{proposition}\label{continuity}
For each $\tau\in\RR$, the bilinear mapping
\beqs
(a,\varphi)\mapsto \Op_{\tau}(a)\varphi,\,\,\, \Gamma_{A_p,\rho}^{*,\infty}(\RR^{2d})\times \SSS^*(\RR^d)\rightarrow \SSS^*(\RR^d),
\eeqs
is hypocontinuous and it extends to the hypocontinuous bilinear mapping
\beqs
(a,T)\mapsto \Op_{\tau}(a)T,\,\,\, \Gamma_{A_p,\rho}^{*,\infty}(\RR^{2d})\times \SSS'^*(\RR^d)\rightarrow \SSS'^*(\RR^d).
\eeqs
Consequently, the mappings
\beqs
a\mapsto \Op_{\tau}(a)&,&\,\, \Gamma_{A_p,\rho}^{*,\infty}(\RR^{2d})\rightarrow \mathcal{L}_b(\SSS^*(\RR^d), \SSS^*(\RR^d)),\,\,\,\, \mbox{and}\\
a\mapsto \Op_{\tau}(a)&,&\,\, \Gamma_{A_p,\rho}^{*,\infty}(\RR^{2d})\rightarrow \mathcal{L}_b(\SSS'^*(\RR^d), \SSS'^*(\RR^d)),
\eeqs
are continuous.
\end{proposition}

For $t\geq0$, put $Q_t=\{(x,\xi)\in\RR^{2d}|\,\langle x\rangle<t, \langle \xi\rangle<t\}$ and $Q_t^c=\RR^{2d}\backslash Q_t$. If $0\leq t\leq 1$, then $Q_t=\emptyset$ and $Q_t^c=\RR^{2d}$. Let $B\geq 0$ and $h,m>0$. We denote by $FS_{A_p,\rho}^{M_p,\infty}(\RR^{2d};B,h,m)$ the vector space of all formal series $\sum_{j=0}^{\infty}a_j$ such that $a_j\in \mathcal{C}^{\infty}(\mathrm{int\,}Q^c_{Bm_j})$, $D^{\alpha}_{\xi} D^{\beta}_x a_j(x,\xi)$ can be extended to a continuous function on $Q^c_{Bm_j}$ for all $\alpha,\beta\in\NN^d$ and
\beqs
\sup_{j\in\NN}\sup_{\alpha,\beta}\sup_{(x,\xi)\in Q_{Bm_j}^c}\frac{\left|D^{\alpha}_{\xi}D^{\beta}_x a_j(x,\xi)\right|
\langle (x,\xi)\rangle^{\rho|\alpha|+\rho|\beta|+2j\rho}e^{-M(m|\xi|)}e^{-M(m|x|)}}
{h^{|\alpha|+|\beta|+2j}A_{\alpha}A_{\beta}A_jA_j}<\infty.
\eeqs
In the above, we use the convention $m_0=0$ and hence $Q^c_{Bm_0}=\RR^{2d}$. With this norm, $FS_{A_p,\rho}^{M_p,\infty}\left(\RR^{2d};B,h,m\right)$ becomes a $(B)$-space. As l.c.s., we define
\beqs
FS_{A_p,\rho}^{(M_p),\infty}(\RR^{2d};B,m)&=&\lim_{\substack{\longleftarrow\\h\rightarrow 0}}
FS_{A_p,\rho}^{M_p,\infty}(\RR^{2d};B,h,m),\\
FS_{A_p,\rho}^{(M_p),\infty}(\RR^{2d};B)&=& \lim_{\substack{\longrightarrow\\m\rightarrow\infty}}
FS_{A_p,\rho}^{(M_p),\infty}(\RR^{2d};B,m),\\
FS_{A_p,\rho}^{\{M_p\},\infty}(\RR^{2d};B,h)&=& \lim_{\substack{\longleftarrow\\m\rightarrow 0}}
FS_{A_p,\rho}^{M_p,\infty}(\RR^{2d};B,h,m),\\
FS_{A_p,\rho}^{\{M_p\},\infty}(\RR^{2d};B)&=&\lim_{\substack{\longrightarrow\\h\rightarrow \infty}}
FS_{A_p,\rho}^{\{M_p\},\infty}(\RR^{2d};B,h).
\eeqs
Then, $FS_{A_p,\rho}^{(M_p),\infty}(\RR^{2d};B,m)$ and $FS_{A_p,\rho}^{\{M_p\},\infty}(\RR^{2d};B,h)$ are $(F)$-spaces and $FS_{A_p,\rho}^{*,\infty}(\RR^{2d};B)$ is barrelled and bornological. The inclusion $\Gamma_{A_p,\rho}^{*,\infty}(\RR^{2d})\rightarrow FS_{A_p,\rho}^{*,\infty}(\RR^{2d};B)$ defined by $a\mapsto\sum_{j\in\NN}a_j$, where $a_0=a$ and $a_j=0$, $j\geq 1$, is continuous. We call this inclusion the canonical one. Observe that for $B_1\leq B_2$, the mapping $\sum_j p_j\mapsto \sum_j p_{j|_{Q_{B_2m_j}^c}}$, $FS_{A_p,\rho}^{*,\infty}(\RR^{2d};B_1)\rightarrow FS_{A_p,\rho}^{*,\infty}(\RR^{2d};B_2)$ is continuous. We also call this mapping canonical.\\
\indent Let $\ds FS_{A_p,\rho}^{*,\infty}(\RR^{2d})=\lim_{\substack{\longrightarrow\\ B\rightarrow\infty}}FS_{A_p,\rho}^{*,\infty}(\RR^{2d};B)$ where the inductive limit is taken in an algebraic sense and the linking mappings are the canonical ones described above. Clearly, $FS_{A_p,\rho}^{*,\infty}(\RR^{2d})$ is non-trivial.\\
\indent If $\sum_j a_j\in FS_{A_p,\rho}^{*,\infty}(\RR^{2d};B)$ and $n\in\NN$, $(\sum_j a_j)_n$ will just mean the function $a_n\in\mathcal{C}^{\infty}(Q_{Bm_n}^c)$. For $N\in\ZZ_+$, $(\sum_j a_j)_{<N}$ denotes the function $\sum_{j=0}^{N-1} a_j\in\mathcal{C}^{\infty}(Q_{Bm_{N-1}}^c)$. Furthermore, $\mathbf{1}$ will denote the element $\sum_j a_j\in FS_{A_p,\rho}^{*,\infty}(\RR^{2d};B)$ given by $a_0(x,\xi)=1$ and $a_j(x,\xi)=0$, $j\in\ZZ_+$.

\begin{definition}(\cite[Definition 3]{BojanP})
Two sums, $\sum_{j\in\NN}a_j,\,\sum_{j\in\NN}b_j\in FS_{A_p,\rho}^{*,\infty}(\RR^{2d})$, are said to be equivalent, in notation $\sum_{j\in\NN}a_j\sim\sum_{j\in\NN}b_j$, if there exist $m>0$ and $B>0$ (resp. there exist $h>0$ and $B>0$), such that for every $h>0$ (resp. for every $m>0$),
\beqs
\sup_{N\in\ZZ_+}\sup_{\alpha,\beta}\sup_{(x,\xi)\in Q_{Bm_N}^c}\frac{\left|D^{\alpha}_{\xi}D^{\beta}_x \sum_{j<N}\left(a_j(x,\xi)-b_j(x,\xi)\right)\right|
\langle (x,\xi)\rangle^{\rho|\alpha|+\rho|\beta|+2N\rho}}
{h^{|\alpha|+|\beta|+2N}A_{\alpha}A_{\beta}A_NA_Ne^{M(m|\xi|)}e^{M(m|x|)}}<\infty.
\eeqs
\end{definition}

In the sequel, we will often use the notation $w=(x,\xi)\in\RR^{2d}$.

\subsection{Change of quantisation}\label{sec_ch_qua}

In \cite{BojanP} we proved that the change of quantisation and the corresponding composition of operators with symbols in $\Gamma_{A_p,\rho}^{*,\infty}(\RR^{2d})$ always results in pseudodifferential operators with symbols in the same class plus an operator which maps $\SSS'^*(\RR^d)$ into $\SSS^*(\RR^d)$, which we call $*$-regularising. However, we will need more precise results concerning the resulting operators when one performs these operations when the symbols vary in bounded subsets of $\Gamma_{A_p,\rho}^{(M_p),\infty}(\RR^{2d};m)$, (resp. of $\Gamma_{A_p,\rho}^{\{M_p\},\infty}(\RR^{2d};h)$). Some of these results are trivial generalisation to the corresponding ones in \cite{BojanP}, but some will require additional work.\\
\indent We start by introducing additional terminology. Let $\Lambda$ be an index set and $\{f_{\lambda}|\, \lambda\in\Lambda\}$ be a set of positive continuous functions on $\RR^{2d}$ each with ultrapolynomial growth of class $*$. We say that a set $U^{(\Lambda)}=\left\{\sum_j a^{(\lambda)}_j\big|\, \lambda\in\Lambda\right\}\subseteq FS_{A_p,\rho}^{*,\infty}(\RR^{2d};B')$ is subordinated to $\{f_{\lambda}|\, \lambda\in\Lambda\}$ in $FS_{A_p,\rho}^{*,\infty}(\RR^{2d})$, in notation $U^{(\Lambda)}\precsim \{f_{\lambda}|\, \lambda\in\Lambda\}$, if the following estimate holds: there exists $B\geq B'$ such that for every $h>0$ there exists $C>0$ (resp. there exist $h,C>0$) such that
\beqs
\sup_{\lambda\in\Lambda}\sup_{j\in\NN}\sup_{\alpha\in\NN^{2d}}\sup_{w\in Q_{Bm_j}^c}\frac{\left|D^{\alpha}_w a^{(\lambda)}_j(w)\right|\langle w\rangle^{\rho(|\alpha|+2j)}}{h^{|\alpha|+2j}A_{|\alpha|+2j}f_{\lambda}(w)}\leq C.
\eeqs
When we want to emphasise that this estimate holds for a particular $B\geq B'$, we write $U^{(\Lambda)}\precsim \{f_{\lambda}|\, \lambda\in\Lambda\}$ in $FS_{A_p,\rho}^{*,\infty}(\RR^{2d};B)$. If $f_{\lambda}=f$, $\forall \lambda\in\Lambda$, by abbreviating notations, we say that $U$ is subordinated to $f$, in notation $U\precsim f$. Clearly, for $U\subseteq FS_{A_p,\rho}^{*,\infty}(\RR^{2d};B_1)$ such that $U\precsim f$ there exists $B\geq B_1$ such that the image of $U$ under the canonical mapping $FS_{A_p,\rho}^{*,\infty}(\RR^{2d};B_1)\rightarrow FS_{A_p,\rho}^{*,\infty}(\RR^{2d};B)$ is a bounded subset of $FS_{A_p,\rho}^{(M_p),\infty}(\RR^{2d};B,m)$ for some $m>0$ (resp. a bounded subset of $FS_{A_p,\rho}^{\{M_p\},\infty}(\RR^{2d};B,h)$ for some $h>0$).\\
\indent Given $U\subseteq FS_{A_p,\rho}^{*,\infty}(\RR^{2d};B_1)$ with $U\precsim f$, we say that a bounded subset $V$ of $\Gamma_{A_p,\rho}^{(M_p),\infty}(\RR^{2d};m)$ for some $m>0$ (resp. a bounded subset $V$ of $\Gamma_{A_p,\rho}^{\{M_p\},\infty}(\RR^{2d};h)$ for some $h>0$) is subordinate to $U$ under $f$, in notations $V\precsim_f U$, if there exists a surjective mapping $\Sigma:U\rightarrow V$ such that the following estimate holds: there exists $B\geq B_1$ such that for every $h>0$ there exists $C>0$ (resp. there exist $h,C>0$) such that for all $\sum_j a_j\in U$ and corresponding $\Sigma(\sum_j a_j)=a\in V$
\beqs
\sup_{N\in\ZZ_+}\sup_{\alpha\in\NN^{2d}}\sup_{w\in Q_{Bm_N}^c}\frac{\left|D^{\alpha}_w\left(a(w)- \sum_{j<N}a_j(w)\right)\right|\langle w\rangle^{\rho(|\alpha|+2N)}}{h^{|\alpha|+2N}A_{|\alpha|+2N}f(w)}\leq C.
\eeqs
Again, if we want to emphasise the particular $B$ for which this holds, we write $V\precsim_f U$ in $FS_{A_p,\rho}^{*,\infty}(\RR^{2d};B)$. If $V\precsim_f U$ and if we denote by $\tilde{V}$ the image of $V$ under the canonical inclusion $\Gamma_{A_p,\rho}^{*,\infty}(\RR^{2d})\rightarrow FS_{A_p,\rho}^{*,\infty}(\RR^{2d};0)$, $a\mapsto a+\sum_{j\in\ZZ_+}0$, then particularising the above estimate for $N=1$ together with the fact that $V$ is bounded in $\Gamma_{A_p,\rho}^{(M_p),\infty}(\RR^{2d};m)$ for some $m>0$ (resp. in $\Gamma_{A_p,\rho}^{\{M_p\},\infty}(\RR^{2d};h)$ for some $h>0$) and that $f$ is continuous and positive, imply that $\tilde{V}\precsim f$ in $FS_{A_p,\rho}^{*,\infty}(\RR^{2d};0)$. In this case, by abbreviating notation, we write $V\precsim f$. This estimate also implies $\Sigma(\sum_j a_j)\sim\sum_j a_j$. To see that given such $U\subseteq FS_{A_p,\rho}^{*,\infty}(\RR^{2d};B)$ there always exists $V\precsim_f U$, we make the following observation. Let $\psi\in\DD^{(A_p)}(\RR^{d})$ in the $(M_p)$ case and $\psi\in\DD^{\{A_p\}}(\RR^{d})$ in the $\{M_p\}$ case respectively, such that $0\leq \psi\leq 1$, $\psi(\xi)=1$ when $\langle\xi\rangle\leq 2$ and $\psi(\xi)=0$ when $\langle\xi\rangle\geq 3$. Put $\chi(x,\xi)=\psi(x)\psi(\xi)$, $\chi_{n,R}(w)=\chi(w/(Rm_n))$ for $n\in\ZZ_+$ and $R>0$ and put $\chi_{0,R}(w)=0$. Given $U\subseteq FS_{A_p,\rho}^{*,\infty}(\RR^{2d};B)$ as above, for $\sum_j a_j\in U$ we denote
\beqs
R(\ssum a_j)(w)= \sum_{j=0}^{\infty} (1-\chi_{j,R}(w))a_j(w).
\eeqs
When $R> B$, this is a well defined $\mathcal{C}^{\infty}$ function on $\RR^{2d}$ since the series is locally finite.

\begin{proposition}\label{subordinate}
Let $U=\left\{\sum_j a^{(\lambda)}_j\big|\, \lambda\in\Lambda\right\}\subseteq FS_{A_p,\rho}^{*,\infty}(\RR^{2d};B')$ be subordinated to $\{f_{\lambda}|\, \lambda\in\Lambda\}$ in $FS_{A_p,\rho}^{*,\infty}(\RR^{2d})$. There exists $R_0> B'$ such that for each $R\geq R_0$, $U_R=\left\{R(\sum_j a^{(\lambda)}_j)\big|\, \lambda\in\Lambda\right\}\subseteq \Gamma_{A_p,\rho}^{*,\infty}(\RR^{2d})$ and the following estimate holds: there exists $B=B(R)\geq B'$ such that for every $h>0$ there exists $C>0$ (resp. there exist $h,C>0$) such that
\beqs
\sup_{\lambda\in\Lambda}\sup_{N\in\ZZ_+}\sup_{\alpha\in\NN^{2d}}\sup_{w\in Q_{Bm_N}^c}\frac{\left|D^{\alpha}_w\left(R(\sum_ja^{(\lambda)}_j)(w)- \sum_{j<N}a^{(\lambda)}_j(w)\right)\right|\langle w\rangle^{\rho(|\alpha|+2N)}}{h^{|\alpha|+2N}A_{|\alpha|+2N}f_{\lambda}(w)}\leq C.
\eeqs
If in addition $f_{\lambda}=f$, $\forall \lambda\in\Lambda$, then $U_R$ is bounded in $\Gamma_{A_p,\rho}^{(M_p),\infty}(\RR^{2d};m)$ for some $m>0$ (resp. bounded in $\Gamma_{A_p,\rho}^{\{M_p\},\infty}(\RR^{2d};h)$ for some $h>0$) and hence $U_R\precsim_f U$.
\end{proposition}

\begin{proof} The first part can be proved by applying the same technique as in the proof of \cite[Theorem 4]{BojanP}. If $f_{\lambda}=f$, $\forall \lambda\in\Lambda$, particularising the estimate for $N=1$ and using $U\precsim f$ one obtains that for every $h>0$ there exists $C>0$ (resp. there exist $h,C>0$) such that
\beqs
\sup_{\lambda\in\Lambda}\sup_{\alpha\in\NN^{2d}}\sup_{w\in Q_{Bm_1}^c}\frac{\left|D^{\alpha}_w R(\sum_ja^{(\lambda)}_j)(w)\right|\langle w\rangle^{\rho|\alpha|}}{h^{|\alpha|}A_{\alpha}f(w)}\leq C.
\eeqs
Now, observe that on $Q_{Bm_1}$ for all $\lambda\in\Lambda$ only a fixed number of terms of the series $\sum_j (1-\chi_{j,R}(w))a^{(\lambda)}_j$ contribute and they are uniformly estimated since $U\precsim f$. Thus, the above estimate is valid on $\RR^{2d}$ and the conclusion follows.
\end{proof}

We say that this $U_{R}$ is canonically obtained from $U$ by $\{\chi_{n,R}\}_{n\in\NN}$. Of course, in this case, the mapping $\Sigma: U\rightarrow U_R$ is just $\sum_j a_j\mapsto R(\sum_j a_j)$. Before we prove the next result we state the following proposition whose proof is the same as the proof of \cite[Theorem 3]{BojanP}.

\begin{proposition}\label{eqsse}
Let $V$ be a bounded subset of $\Gamma_{A_p,\rho}^{(M_p),\infty}(\RR^{2d};\tilde{m})$ for some $\tilde{m}>0$ (resp. of $\Gamma_{A_p,\rho}^{\{M_p\},\infty}(\RR^{2d};\tilde{h})$ for some $\tilde{h}>0$). Assume that there exist $B,m>0$ such that for every $h>0$ there exists $C>0$ (resp. there exist $B,h>0$ such that for every $m>0$ there exists $C>0$) such that
\beqs
\sup_{a\in V}\sup_{N\in\ZZ_+}\sup_{\alpha\in\NN^{2d}}\sup_{w\in Q_{Bm_N}^c}\frac{\left|D^{\alpha}_w a(w)\right|\langle w\rangle^{\rho(|\alpha|+2N)}}{h^{|\alpha|+2N}A_{|\alpha|+2N}e^{M(m|w|)}}\leq C.
\eeqs
Then, for each $\tau\in\RR$, $\{\mathrm{Op}_{\tau}(a)|\, a\in U\}$ is equicontinuous subset of $\mathcal{L}(\SSS'^*(\RR^d),\SSS^*(\RR^d))$.
\end{proposition}

In the sequel, we will often use the term ``$*$-regularising set'' for a subset of the space $\mathcal{L}(\SSS'^*(\RR^d),\SSS^*(\RR^d))$.

\begin{proposition}\label{changequa}
Let $U_1\subseteq FS_{A_p,\rho}^{*,\infty}(\RR^{2d};B)$ be such that $U_1\precsim f$, for some continuous positive function $f$ with ultrapolynomial growth of class $*$ and let $\tau,\tau_1\in\RR$. For each $\sum_j a_j\in U_1$ and $j\in\NN$, define
\beqs
p_{j,a}(x,\xi)=\sum_{k+|\beta|=j}\frac{(\tau_1-\tau)^{|\beta|}}{\beta!}\partial^{\beta}_{\xi}D^{\beta}_x a_k(x,\xi),\,\, (x,\xi)\in Q_{Bm_j}^c.
\eeqs
Then, $U=\left\{\sum_j p_{j,a}\big|\, \sum_j a_j\in U_1\right\}$ is subset of $FS_{A_p,\rho}^{*,\infty}(\RR^{2d};B)$ and $U\precsim f$. There exists $R>0$, which can be chosen arbitrarily large, such that
\beqs
\big\{\mathrm{Op}_{\tau_1}(R(\ssum a_j))-\mathrm{Op}_{\tau}(R(\ssum p_{j,a}))\big|\,\ssum a_j\in U_1\big\}
\eeqs
is an equicontinuous subset of $\mathcal{L}(\SSS'^*(\RR^d),\SSS^*(\RR^d))$. Moreover,
\beqs
\left\{R(\ssum a_j)\big|\,\ssum a_j\in U_1\right\}\precsim_f U_1\,\,\, \mbox{and}\,\,\, \left\{R(\ssum p_{j,a})|\,\ssum a_j\in U_1\right\}\precsim_f U.
\eeqs
\end{proposition}

\begin{proof} The easy proof that $U\subseteq FS_{A_p,\rho}^{*,\infty}(\RR^{2d};B)$ and $U\precsim f$ is omitted. By Proposition \ref{subordinate}, there exists $\tilde{R}_1> 2B$ such that for each $R_1\geq \tilde{R}_1$, $U_{1,R_1}=\left\{R_1(\sum_j a_j)\big|\, \sum_j a_j\in U_1\right\}$ is subordinated to $U_1$ under $f$. Fix such $R_1$. For each $a=R_1(\sum_j a_j)$ and $n\in\NN$, denote
\beqs
q_{n,a}(x,\xi)=\sum_{|\beta|=n}\frac{(\tau_1-\tau)^{|\beta|}}{\beta!}\partial^{\beta}_{\xi}D^{\beta}_x a(x,\xi).
\eeqs
Clearly, $\tilde{U}=\left\{\sum_j q_{j,a}\big|\, a\in U_{1,R_1}\right\}\subseteq FS_{A_p,\rho}^{*,\infty}(\RR^{2d};B)$ and $\tilde{U}\precsim f$. Now, by the same technique as in the proof of \cite[Theorem 5]{BojanP} (and employing Proposition \ref{subordinate}), one proves that there exists $\tilde{R}>0$, which depends on $R_1$, such that the kernels of the operators $\tilde{T}_a=\mathrm{Op}_{\tau_1}(a)-\mathrm{Op}_{\tau}\left(\tilde{R}(\sum_j q_{j,a})\right)$ for $a\in U_{1,R_1}$ form a bounded subset of $\SSS^*(\RR^{2d})$. In fact, by carefully examining the proof of the quoted result, it follows that $\tilde{R}$ can be chosen arbitrarily large, in particular, greater than the chosen $R_1$. Thus $\{\tilde{T}_a|\, a\in U_{1,R_1}\}$ is a bounded subset of $\mathcal{L}_b(\SSS'^*(\RR^d),\SSS^*(\RR^d))$, hence also equicontinuous since $\SSS'^*(\RR^d)$ is barrelled. Clearly, $W=\{\tilde{R}(\sum_j q_{j,a})|\, a\in U_{1,R_1}\}$ is a bounded subset of $\Gamma_{A_p,\rho}^{(M_p),\infty}(\RR^{2d};\tilde{m})$ for some $\tilde{m}>0$ (resp. of $\Gamma_{A_p,\rho}^{\{M_p\},\infty}(\RR^{2d};\tilde{h})$ for some $\tilde{h}>0$) and $W\precsim_f \tilde{U}$ ($\tilde{R}$ can be chosen arbitrarily large; $W$ is canonically obtained from $\tilde{U}$ by $\{\chi_{n,\tilde{R}}\}_{n\in\NN}$). For $U$ in the proposition, take $R\geq \max\{R_1,\tilde{R}\}$ as in Proposition \ref{subordinate} and let $U_R$ be the canonically obtained set from $U$ by $\{\chi_{n,R}\}_{n\in\NN}$. For $\sum_j a_j\in U_1$ denote $a=R_1(\sum_j a_j)$ and $\tilde{a}=R(\sum_j p_{j,a})$. Observe
\beq\label{eqn1}
\mathrm{Op}_{\tau_1}(a)-\mathrm{Op}_{\tau}(\tilde{a})=\tilde{T}_a+\mathrm{Op}_{\tau}\left(\tilde{R}(\ssum q_{j,a})-\tilde{a}\right).
\eeq
Note that
\beq\label{eqn2}
\tilde{R}(\ssum q_{j,a})-\tilde{a}=\tilde{R}(\ssum q_{j,a})-\sum_{j=0}^N q_{j,a}+\sum_{j=0}^N q_{j,a}-\sum_{j=0}^N p_{j,a}+\sum_{j=0}^N p_{j,a}-\tilde{a}.
\eeq
For $\beta\in\NN^d$ put $c_{\beta}=(\tau_1-\tau)^{|\beta|}/\beta!$. For $N\in\NN$, we have
\beqs
\sum_{j=0}^N p_{j,a}&=&\sum_{j=0}^N\sum_{k+s=j}\sum_{|\beta|=s}c_{\beta}\partial^{\beta}_{\xi}D^{\beta}_x a_k =\sum_{j=0}^N\sum_{s=0}^j\sum_{|\beta|=s}c_{\beta}\partial^{\beta}_{\xi}D^{\beta}_x a_{j-s}\\
&=&\sum_{s=0}^N\sum_{j=s}^N\sum_{|\beta|=s}c_{\beta}\partial^{\beta}_{\xi}D^{\beta}_x a_{j-s} =\sum_{j=0}^N\sum_{s=j}^N\sum_{|\beta|=j}c_{\beta}\partial^{\beta}_{\xi}D^{\beta}_x a_{s-j}
\eeqs
Hence, we infer
\beqs
\sum_{j=0}^N q_{j,a}-\sum_{j=0}^N p_{j,a}=\sum_{j=0}^N\sum_{|\beta|=j}c_{\beta}\partial^{\beta}_{\xi}D^{\beta}_x \left(a-\sum_{k=0}^{N-j} a_k\right).
\eeqs
Thus, (\ref{eqn2}) together with $U_{1,R_1}\precsim_f U_1$, $W\precsim_f \tilde{U}$, $U_R\precsim_f U$ and Proposition \ref{eqsse} imply that $\left\{\mathrm{Op}_{\tau}\left(\tilde{R}(\sum_j q_{j,a})-\tilde{a}\right)\big|\, \sum_j a_j\in U_1\right\}$ is equicontinuous in $\mathcal{L}(\SSS'^*(\RR^d),\SSS^*(\RR^d))$. Hence, by (\ref{eqn1}), so is the set $\{\mathrm{Op}_{\tau_1}(a)-\mathrm{Op}_{\tau}(\tilde{a})|\, \sum_j a_j\in U_1\}$. Now, observe that for $N\in\NN$,
\beqs
R(\ssum a_j)-a=R(\ssum a_j)-\sum_{j=0}^N a_j+\sum_{j=0}^N a_j-a.
\eeqs
Thus, $\{R(\sum_j a_j)-a|\, \sum_j a_j\in U_1\}$ satisfies the conditions in Proposition \ref{eqsse}. We can conclude that $\left\{\mathrm{Op}_{\tau_1}(R(\sum_j a_j))-\mathrm{Op}_{\tau}(\tilde{a})\big|\, \sum_j a_j\in U_1\right\}$ is an equicontinuous subset of $\mathcal{L}(\SSS'^*(\RR^d),\SSS^*(\RR^d))$.
\end{proof}

\section{Weyl quantisation. The sharp product and the ring structure of $FS_{A_p,\rho}^{*,\infty}(\RR^{2d};B)$}\label{section3}

\subsection{The sharp product in $FS_{A_p,\rho}^{*,\infty}(\RR^{2d};B)$}

For $\sum_j a_j,\sum_j b_j\in FS_{A_p,\rho}^{*,\infty}(\RR^{2d};B)$ we define their sharp product, in notation $\sum_j a_j \# \sum_j b_j$, by $\sum_j c_j=\sum_j a_j \# \sum_j b_j$, where
\beqs
c_j(x,\xi)=\sum_{s+k+l=j}\sum_{|\alpha+\beta|=l}\frac{(-1)^{|\beta|}}{\alpha!\beta!2^l}\partial^{\alpha}_{\xi}D^{\beta}_x a_s(x,\xi)\partial^{\beta}_{\xi} D^{\alpha}_x b_k(x,\xi),\,\, (x,\xi)\in Q^c_{Bm_j}.
\eeqs
One can easily verify that $\sum_j c_j$ is a well defined element of $FS_{A_p,\rho}^{*,\infty}(\RR^{2d};B)$. If $a\in \Gamma_{A_p,\rho}^{*,\infty}(\RR^{2d})$, then $a\#\sum_j b_j$ will denote the $\#$ product of the image of $a$ under the canonical inclusion $\Gamma_{A_p,\rho}^{*,\infty}(\RR^{2d})\rightarrow FS_{A_p,\rho}^{*,\infty}(\RR^{2d};B)$ and $\sum_j b_j$. The same goes if $b\in \Gamma_{A_p,\rho}^{*,\infty}(\RR^{2d})$ or if both $a,b\in \Gamma_{A_p,\rho}^{*,\infty}(\RR^{2d})$.

\begin{lemma}\label{subsharpproduct}
Let $\Lambda$ and $\Omega$ be two index sets and $U^{(\Lambda)}=\left\{\sum_j a^{(\lambda)}_j\big|\, \lambda\in\Lambda\right\}\precsim \{f_{\lambda}|\, \lambda\in\Lambda\}$ in $FS_{A_p,\rho}^{*,\infty}(\RR^{2d};B)$ and $U^{(\Omega)}=\left\{\sum_j b^{(\omega)}_j\big|\, \omega\in\Omega\right\}\precsim \{g_{\omega}|\, \omega\in\Omega\}$ in $FS_{A_p,\rho}^{*,\infty}(\RR^{2d};B)$. Then $U^{(\Lambda)}\# U^{(\Omega)}\precsim \{f_{\lambda}g_{\omega}|\, \lambda\in\Lambda,\, \omega\in\Omega\}$ in $FS_{A_p,\rho}^{*,\infty}(\RR^{2d};B)$.
\end{lemma}

\begin{proof} The proof is straightforward and we omit it.
\end{proof}

\subsection{Relations between Weyl symbols}

Having in mind the properties of the $\#$ product defined above and the results from Subsection \ref{sec_ch_qua}, we can prove the following result.

\begin{theorem}\label{weylq}
Let $U_1,U_2\subseteq FS_{A_p,\rho}^{*,\infty}(\RR^{2d};B)$ be such that $U_1\precsim f_1$ and $U_2\precsim f_2$ in $FS_{A_p,\rho}^{*,\infty}(\RR^{2d};B)$ for some continuous positive functions $f_1$ and $f_2$ with ultrapolynomial growth of class $*$. Then:
\begin{itemize}
\item[$i)$] $U_1\#U_2\precsim f_1f_2$ in $FS_{A_p,\rho}^{*,\infty}(\RR^{2d};B)$.
\item[$ii)$] Let $V_k\precsim_{f_k} U_k$ with $\Sigma_k: U_k\rightarrow V_k$ the surjective mapping, $k=1,2$. There exists $R>0$, which can be chosen arbitrarily large, such that
    \beqs
    \left\{\Op_{1/2}\left(\Sigma_1(\ssum a_j)\right)\Op_{1/2}\left(\Sigma_2(\ssum b_j)\right)-\Op_{1/2}\left(R(\ssum a_j\# \ssum b_j)\right)\big|\right.\\
    \left.\ssum a_j\in U_1,\, \ssum b_j\in U_2\right\}
    \eeqs
    is an equicontinuous subset of $\mathcal{L}(\SSS'^*(\RR^d),\SSS^*(\RR^d))$ and
    \beq\label{krh1791}
    \left\{R(\ssum a_j\# \ssum b_j)\big|\, \ssum a_j\in U_1,\, \ssum b_j\in U_2\right\}\precsim_{f_1f_2}U_1\#U_2.
    \eeq
\end{itemize}
\end{theorem}

\begin{proof} Observe that $i)$ is a special case of Lemma \ref{subsharpproduct}. We prove $ii)$. Denote $U=U_1\# U_2$. For $\sum_j a_j\in U_1$, $\sum_j b_j\in U_2$ denote $a=\Sigma_1(\sum_j a_j)\in V_1$, $b=\Sigma_2(\sum_j b_j)\in V_2$. For $j\in\NN$, define
\beqs
p_{j,a}(x,\xi)&=&\sum_{|\beta|=j}\frac{1}{2^j\beta!}\partial^{\beta}_{\xi}D^{\beta}_x a(x,\xi),\,\, x,\xi\in\RR^d\\
q_{j,b}(x,\xi)&=&\sum_{|\beta|=j}\frac{(-1)^j}{2^j\beta!}\partial^{\beta}_{\xi}D^{\beta}_x b(x,\xi),\,\, x,\xi\in\RR^d.
\eeqs
Since $V_k\precsim f_k$, $k=1,2$, by Proposition \ref{changequa}, there exists $R>0$ such that
\beqs
&{}&\left\{T_a=a^w-\mathrm{Op}_0\left(R(\ssum p_{j,a})\right)\big|\, \ssum a_j\in U_1\right\}\,\, \mbox{and}\\ &{}&\left\{T_b=b^w-\mathrm{Op}_1\left(R(\ssum q_{j,b})\right)\big|\, \ssum b_j\in U_2\right\}
\eeqs
are equicontinuous $*$-regularising sets (one can make this to hold for the same $R$ by taking it to be large enough; cf. Propositions \ref{subordinate} and \ref{eqsse}). For brevity in notation, for $a=\Sigma_1(\ssum a_j)\in V_1$ and $b=\Sigma_2(\ssum b_j)\in V_2$, denote $\tilde{a}=R(\sum_j p_{j,a})$ and $\tilde{b}=R(\sum_j q_{j,b})$. By Proposition \ref{changequa}, $\{\tilde{a}|\, \sum_j a_j\in U_1\}\precsim f_1$ and $\{\tilde{b}|\, \sum_j b_j\in U_2\}\precsim f_2$, and these are bounded subsets of some $\Gamma_{A_p,\rho}^{(M_p),\infty}(\RR^{2d};m)$ (resp. of some $\Gamma_{A_p,\rho}^{\{M_p\},\infty}(\RR^{2d};h)$). Thus, Proposition \ref{continuity} implies that $\{\tilde{a}(x,D)|\, \sum_j a_j\in U_1\}$ and $\{\mathrm{Op}_1(\tilde{b})|\, \sum_j b_j\in U_2\}$ are equicontinuous subsets of $\mathcal{L}(\SSS^*(\RR^d),\SSS^*(\RR^d))$ and $\mathcal{L}(\SSS'^*(\RR^d),\SSS'^*(\RR^d))$. Hence, $a^wb^w=\tilde{a}(x,D)\mathrm{Op}_1(\tilde{b})+T_{a,b}$, where $\{T_{a,b}|\, \sum_j a_j\in U_1,\, \sum_j b_j\in U_2\}$ is an equicontinuous $*$-regularising set. For $j\in\NN$, define
\beq\label{eqs5}
\tilde{c}_{j,\tilde{a},\tilde{b}}(x,\xi)=\sum_{|\gamma|=j}\frac{1}{\gamma!}\partial^{\gamma}_{\xi}\left(\tilde{a}(x,\xi) D^{\gamma}_x \tilde{b}(x,\xi)\right),\,\, x,\xi\in\RR^d.
\eeq
One easily verifies that $\sum_j \tilde{c}_{j,\tilde{a},\tilde{b}}\in FS_{A_p,\rho}^{*,\infty}(\RR^{2d};0)$ and for the set
\beqs
\tilde{W}= \left\{\ssum \tilde{c}_{j,\tilde{a},\tilde{b}}|\,\ssum a_j\in U_1,\, \ssum b_j\in U_2\right\}
\eeqs
we have that $\tilde{W}\precsim f_1f_2$ in $FS_{A_p,\rho}^{*,\infty}(\RR^{2d};0)$. By the same technique as in the proof of \cite[Theorem 7]{BojanP} (and using Proposition \ref{subordinate}) one obtains that there exists $R_1\geq R$, which can be chosen arbitrarily large, such that the kernels of the operators
\beqs
\tilde{a}(x,D)\mathrm{Op}_1(\tilde{b})-\mathrm{Op}_0\left(R_1(\ssum \tilde{c}_{j,\tilde{a},\tilde{b}})\right)\,\, \mbox{for}\,\, \ssum a_j\in U_1\,\, \mbox{and}\,\, \ssum b_j\in U_2
\eeqs
form a bounded subset of $\SSS^*(\RR^{2d})$. Denoting $\tilde{c}_{\tilde{a},\tilde{b}}=R_1(\sum_j \tilde{c}_{j,\tilde{a},\tilde{b}})$, we conclude that
\beqs
\left\{\tilde{a}(x,D)\mathrm{Op}_1(\tilde{b})-\tilde{c}_{\tilde{a},\tilde{b}}(x,D)\big|\, \ssum a_j\in U_1,\, \ssum b_j\in U_2\right\}
\eeqs
is equicontinuous in $\mathcal{L}(\SSS'^*(\RR^d),\SSS^*(\RR^d))$. Moreover, $\{\tilde{c}_{\tilde{a},\tilde{b}}|\, \sum_j a_j\in U_1,\, \sum_j b_j\in U_2\}$ is a bounded subset of $\Gamma_{A_p,\rho}^{(M_p),\infty}(\RR^{2d};m')$ for some $m'>0$ (resp. of $\Gamma_{A_p,\rho}^{\{M_p\},\infty}(\RR^{2d};h')$ for some $h'>0$) and $\{\tilde{c}_{\tilde{a},\tilde{b}}|\, \sum_j a_j\in U_1,\, \sum_j b_j\in U_2\}\precsim_{f_1f_2} \tilde{W}$ (since we can take $R_1$ arbitrarily large and $\tilde{W}\precsim f_1f_2$; cf. Proposition \ref{subordinate}). We apply Proposition \ref{changequa} with $\tilde{W}\precsim f_1f_2$ and $\tau=1/2$ and $\tau_1=0$ in order to obtain the existence of $R_2$, with $R_2\geq R_1$, such that
\beqs
\left\{\mathrm{Op}_0\left(R_2(\ssum \tilde{c}_{j,\tilde{a},\tilde{b}})\right)-\mathrm{Op}_{1/2}\left(R_2(\ssum c_{j,\tilde{a},\tilde{b}})\right)\big|\, \ssum a_j\in U_1,\, \ssum b_j\in U_2\right\}
\eeqs
is an equicontinuous subset of $\mathcal{L}(\SSS'^*(\RR^d),\SSS^*(\RR^d))$, where
\beqs
c_{j,\tilde{a},\tilde{b}}(x,\xi)=\sum_{k+|\gamma|=j}\frac{(-1)^{|\gamma|}}{\gamma!2^{|\gamma|}}\partial^{\gamma}_{\xi}D^{\gamma}_x \tilde{c}_{k,\tilde{a},\tilde{b}}(x,\xi),\,\, x,\xi\in \RR^d.
\eeqs
By (\ref{eqs5}), we infer that
\beqs
c_{j,\tilde{a},\tilde{b}}(x,\xi)=\sum_{|\alpha+\beta+\gamma|=j}\frac{(-1)^{|\gamma|}}{\alpha!\beta!\gamma!2^{|\gamma|}} \partial^{\gamma}_{\xi} D^{\gamma}_x \left(\partial^{\alpha}_{\xi}\tilde{a}(x,\xi)\partial^{\beta}_{\xi} D^{\alpha+\beta}_x \tilde{b}(x,\xi)\right),\,\, x,\xi\in \RR^d.
\eeqs
As
\beqs
\tilde{c}_{\tilde{a},\tilde{b}}-R_2\left(\ssum \tilde{c}_{j,\tilde{a},\tilde{b}}\right)= \tilde{c}_{\tilde{a},\tilde{b}}-\sum_{j=0}^{N-1}\tilde{c}_{j,\tilde{a},\tilde{b}}+ \sum_{j=0}^{N-1}\tilde{c}_{j,\tilde{a},\tilde{b}} -R_2\left(\ssum \tilde{c}_{j,\tilde{a},\tilde{b}}\right),
\eeqs
we can conclude that the set $\{\tilde{c}_{\tilde{a},\tilde{b}}-R_2(\sum_j \tilde{c}_{j,\tilde{a},\tilde{b}})|\, \sum_j a_j\in U_1,\, \sum_j b_j\in U_2\}$ satisfies the conditions of Proposition \ref{eqsse}. Hence,
\beqs
\left\{\tilde{c}_{\tilde{a},\tilde{b}}(x,D)-\mathrm{Op}_0(R_2(\ssum \tilde{c}_{j,\tilde{a},\tilde{b}}))\big|\, \ssum a_j\in U_1,\, \ssum b_j\in U_2\right\}
\eeqs
is an equicontinuous subset of $\mathcal{L}(\SSS'^*(\RR^d),\SSS^*(\RR^d))$. Thus, by denoting $c_{a,b}=R_2(\sum_j c_{j,\tilde{a},\tilde{b}})$, we conclude that $a^w b^w=c^w_{a,b}+T'_{a,b}$, where $\{T'_{a,b}|\, \sum_j a_j\in U_1,\, \sum_j b_j\in U_2\}$ is an  equicontinuous subset of $\mathcal{L}(\SSS'^*(\RR^d),\SSS^*(\RR^d))$. Next we prove
\beq\label{trs1793}
\left\{c_{a,b}\big|\, \ssum a_j\in U_1,\, \ssum b_j\in U_2\right\}\precsim_{f_1f_2} \left\{a\#b\big|\, \ssum a_j\in U_1,\, \ssum b_j\in U_2\right\},
\eeq
where the surjective mapping is given by $a\#b\mapsto c_{a,b}$. For $N\in\ZZ_+$, we have
\begin{align*}
\sum_{j=0}^{N-1}&\sum_{|\alpha+\beta+\gamma|=j}\frac{(-1)^{|\gamma|}}{\alpha!\beta!\gamma!2^{|\gamma|}} \partial^{\gamma}_{\xi} D^{\gamma}_x \left(\partial^{\alpha}_{\xi} \tilde{a}\cdot \partial^{\beta}_{\xi} D^{\alpha+\beta}_x \tilde{b}\right)\\
&=\sum_{j=0}^{N-1}\sum_{|\alpha+\beta+\gamma|=j}\frac{(-1)^{|\gamma|}}{\alpha!\beta!\gamma!2^{|\gamma|}} \partial^{\gamma}_{\xi} D^{\gamma}_x \left(\partial^{\alpha}_{\xi} \left(\tilde{a}-\sum_{s=0}^{N-j-1} p_{s,a}\right)\cdot \partial^{\beta}_{\xi} D^{\alpha+\beta}_x \tilde{b}\right)\\
&{}\,\,\,\,+ \sum_{j=0}^{N-1}\sum_{s=0}^{N-j-1}\sum_{|\alpha+\beta+\gamma|=j} \frac{(-1)^{|\gamma|}}{\alpha!\beta!\gamma!2^{|\gamma|}} \partial^{\gamma}_{\xi} D^{\gamma}_x \left(\partial^{\alpha}_{\xi} p_{s,a}\cdot \partial^{\beta}_{\xi} D^{\alpha+\beta}_x \left(\tilde{b}-\sum_{k=0}^{N-j-s-1}q_{k,b}\right)\right)\\
&{}\,\,\,\,+ \sum_{j=0}^{N-1}\sum_{s=0}^{N-j-1}\sum_{k=0}^{N-j-s-1} \sum_{|\alpha+\beta+\gamma|=j}\frac{(-1)^{|\gamma|}}{\alpha!\beta!\gamma!2^{|\gamma|}} \partial^{\gamma}_{\xi} D^{\gamma}_x \left(\partial^{\alpha}_{\xi} p_{s,a}\cdot \partial^{\beta}_{\xi} D^{\alpha+\beta}_xq_{k,b}\right)\\
&= S_{1,N-1}+S_{2,N-1}+S_{3,N-1}.
\end{align*}
For $S_{3,N-1}$ we have
\beqs
S_{3,N-1}&=&\sum_{t=0}^{N-1}\sum_{j+s+k=t} \sum_{|\alpha+\beta+\gamma|=j}\frac{(-1)^{|\gamma|}}{\alpha!\beta!\gamma!2^{|\gamma|}} \partial^{\gamma}_{\xi} D^{\gamma}_x \left(\partial^{\alpha}_{\xi} p_{s,a}\cdot \partial^{\beta}_{\xi} D^{\alpha+\beta}_xq_{k,b}\right)\\
&=& \sum_{t=0}^{N-1} \sum_{|\alpha+\beta+\gamma+\delta+\mu|=t} \frac{(-1)^{|\gamma|+|\mu|}}{\alpha!\beta!\gamma!\delta!\mu!2^{|\gamma|+|\delta|+|\mu|}} \partial^{\gamma}_{\xi} D^{\gamma}_x \left(\partial^{\alpha+\delta}_{\xi} D^{\delta}_x a\cdot \partial^{\beta+\mu}_{\xi} D^{\alpha+\beta+\mu}_x b\right).
\eeqs
Let $g(x,\eta,y,\xi)=a(x,\eta)b(y,\xi)$. Clearly $g\in\mathcal{C}^{\infty}(\RR^{4d})$. One verifies
\begin{multline*}
S_{3,N-1}(x,\xi)=\sum_{t=0}^{N-1} \sum_{|\alpha+\beta+\gamma+\delta+\mu|=t} \frac{(-1)^{|\gamma|+|\mu|}}{\alpha!\beta!\gamma!\delta!\mu!2^{|\gamma|+|\delta|+|\mu|}}\\
\cdot(\partial_{\xi}+\partial_{\eta})^{\gamma} (D_x+D_y)^{\gamma} \partial^{\alpha+\delta}_{\eta} D^{\delta}_x \partial^{\beta+\mu}_{\xi} D^{\alpha+\beta+\mu}_y g(x,\eta,y,\xi)\Big|_{\substack{y=x\\ \eta=\xi}}.
\end{multline*}
Considering only the operator on the right hand side, we have
\begin{align*}
\sum_{t=0}^{N-1}& \sum_{|\alpha+\beta+\gamma+\delta+\mu|=t} \frac{(-1)^{|\gamma|+|\mu|}}{\alpha!\beta!\gamma!\delta!\mu!2^{|\gamma|+|\delta|+|\mu|}} (\partial_{\xi}+\partial_{\eta})^{\gamma} (D_x+D_y)^{\gamma} \partial^{\alpha+\delta}_{\eta} D^{\delta}_x \partial^{\beta+\mu}_{\xi} D^{\alpha+\beta+\mu}_y\\
&=\sum_{t=0}^{N-1} \frac{1}{t!}\left(\partial_{\eta}\cdot D_y+\partial_{\xi}\cdot D_y-\frac{1}{2}(\partial_{\xi}+\partial_{\eta})\cdot(D_x+D_y)+\frac{1}{2}\partial_{\eta}\cdot D_x-\frac{1}{2}\partial_{\xi}\cdot D_y\right)^t\\
&= \sum_{t=0}^{N-1} \frac{1}{2^t t!}(\partial_{\eta}\cdot D_y-\partial_{\xi}\cdot D_x)^t.
\end{align*}
Hence,
\beqs
S_{3,N-1}(x,\xi)=\sum_{j=0}^{N-1}\sum_{|\alpha+\beta|=j}\frac{(-1)^{|\beta|}}{\alpha!\beta!2^j}\partial^{\alpha}_{\xi}D^{\beta}_x a(x,\xi)\partial^{\beta}_{\xi}D^{\alpha}_x b(x,\xi).
\eeqs
One easily verifies that for every $h>0$ there exists $C>0$ (resp. there exist $h,C>0$) such that for all $\sum_j a_j\in U_1$ and $\sum_j b_j\in U_2$
\beqs
\sup_{N\in\ZZ_+}\sup_{\alpha\in\NN^{2d}}\sup_{w\in Q_{B'm_N}^c}\frac{\left|D^{\alpha}_w S_{1,N-1}(w)\right|\langle w\rangle^{\rho(|\alpha|+2N)}}{h^{|\alpha|+2N}A_{|\alpha|+2N}f_1(w)f_2(w)}&\leq& C\\
\sup_{N\in\ZZ_+}\sup_{\alpha\in\NN^{2d}}\sup_{w\in Q_{B'm_N}^c}\frac{\left|D^{\alpha}_w S_{2,N-1}(w)\right|\langle w\rangle^{\rho(|\alpha|+2N)}}{h^{|\alpha|+2N}A_{|\alpha|+2N}f_1(w)f_2(w)} &\leq& C,
\eeqs
for some $B'>0$. We conclude that (\ref{trs1793}) holds true. We will prove that for every $h>0$ there exists $C>0$ (resp. there exist $h,C>0$) such that for all $\sum_j a_j\in U_1$ and $\sum_j b_j\in U_2$ and the corresponding $a=\Sigma_1(\sum_j a_j)$ and $b=\Sigma_2(\sum_j b_j),$
\beqs
\sup_{N\in\ZZ_+}\sup_{\alpha\in\NN^{2d}}\sup_{w\in Q_{B''m_N}^c}\frac{\left|D^{\alpha}_w \left((a\#b)_{<N}(w)-(\sum_j a_j\#\sum_j b_j)_{<N}(w)\right)\right|}{h^{|\alpha|+2N}A_{|\alpha|+2N}f_1(w)f_2(w)\langle w\rangle^{-\rho(|\alpha|+2N)}}\leq C,
\eeqs
for some $B''>0$. Note that this, together with (\ref{trs1793}), implies $\{c_{a,b}|\,\sum_j a_j\in U_1,\, \sum_j b_j\in U_2\}\precsim_{f_1f_2} U$, where the surjective mapping is given by $\sum_j a_j\#\sum_j b_j$ $\mapsto c_{a,b}$. Thus, by employing Proposition \ref{subordinate} to $U\precsim f_1f_2$ and then applying Proposition \ref{eqsse}, we conclude that the set $\{c_{a,b}^w-\Op_{1/2}(R'(\sum_j a_j\#\sum_j b_j))|$ $ \sum_j a_j\in U_1,\, \sum_j b_j\in U_2\}$ is equicontinuous $*$-regularising for arbitrarily large $R'$. To prove the desired estimate, let $N\in\ZZ_+$. By employing the notation $a=\Sigma_1(\sum_j a_j)$ and $b=\Sigma_2(\sum_j b_j)$ for $\sum_j a_j\in U_1$ and $\sum_j b_j\in U_2$, we have
\beqs
(a\#b)_{<N}&=&\sum_{j=0}^{N-1}\sum_{|\alpha+\beta|=j}\frac{(-1)^{|\beta|}}{\alpha!\beta!2^j}\partial^{\alpha}_{\xi}D^{\beta}_x \left(a-\sum_{s=0}^{N-j-1}a_s\right)\cdot \partial^{\beta}_{\xi} D^{\alpha}_x b\\
&{}&+ \sum_{j=0}^{N-1}\sum_{s=0}^{N-j-1}\sum_{|\alpha+\beta|=j} \frac{(-1)^{|\beta|}}{\alpha!\beta!2^j}\partial^{\alpha}_{\xi}D^{\beta}_x a_s\cdot \partial^{\beta}_{\xi} D^{\alpha}_x\left(b-\sum_{k=0}^{N-j-s-1} b_k\right)\\
&{}&+\sum_{j=0}^{N-1}\sum_{s=0}^{N-j-1} \sum_{k=0}^{N-j-s-1}\sum_{|\alpha+\beta|=j} \frac{(-1)^{|\beta|}}{\alpha!\beta!2^j}\partial^{\alpha}_{\xi}D^{\beta}_x a_s\cdot \partial^{\beta}_{\xi} D^{\alpha}_x b_k\\
&=&S'_{1,N-1}+S'_{2,N-1}+S'_{3,N-1}.
\eeqs
One easily verifies that for $S'_{1,N-1}$ and $S'_{2,N-1}$ the same type of estimate as for $S_{1,N-1}$ and $S_{2,N-1}$, hold. For $S'_{3,N-1}$ we have
\beqs
S'_{3,N-1}=\sum_{t=0}^{N-1}\sum_{j+s+k=t}\sum_{|\alpha+\beta|=j} \frac{(-1)^{|\beta|}}{\alpha!\beta!2^j}\partial^{\alpha}_{\xi}D^{\beta}_x a_s\cdot \partial^{\beta}_{\xi} D^{\alpha}_x b_k=(\ssum a_j\#\ssum b_j)_{<N},
\eeqs
which completes the proof.
\end{proof}

As a direct consequence of Theorem \ref{weylq} and Proposition \ref{eqsse} we have the following corollary.

\begin{corollary}\label{corweylqu}
Let $U_1,U_2\subseteq FS_{A_p,\rho}^{*,\infty}(\RR^{2d};B)$ with $U_1\precsim f_1$ and $U_2\precsim f_2$ for some continuous positive function of ultrapolynomial growth of class $*$. For $\sum_j a_j\in U_1$ and $\sum_j b_j\in U_2$ denote $\sum_j c_{j,a,b}=\sum_j a_j\#\sum_j b_j\in U_1\# U_2$. Then, there exists $R>0$, which can be chosen arbitrarily large, such that
\beqs
\left\{a^wb^w-c^w\big|\, a=R(\ssum a_j),\, b=R(\ssum b_j),\, c=R(\ssum c_{j,a,b})\right\}
\eeqs
is an equicontinuous subset of $\mathcal{L}(\SSS'^*(\RR^d),\SSS^*(\RR^d))$ and (\ref{krh1791}) holds.
\end{corollary}

\begin{remark} This corollary is applicable when $U_1$ and $U_2$ are bounded subsets of $\Gamma_{A_p,\rho}^{(M_p),\infty}(\RR^{2d};m)$ for some $m>0$ (resp. of $\Gamma_{A_p,\rho}^{\{M_p\},\infty}(\RR^{2d};h)$ for some $h>0$). Then the corollary gives the famous Moyal's formula for the asymptotic expansion of the composition of two Weyl quantisations. Furthermore, it states that the resulting $*$-regularising operators form an equicontinuous subset of $\mathcal{L}(\SSS'^*(\RR^d),\SSS^*(\RR^d))$.
\end{remark}

\subsection{The ring structure of $FS_{A_p,\rho}^{*,\infty}(\RR^{2d};B)$}

Theorem \ref{weylq} $i)$ implies that if $U_1$ and $U_2$ are two bounded subsets of $FS_{A_p,\rho}^{(M_p),\infty}(\RR^{2d};B,m_1)$ and $FS_{A_p,\rho}^{(M_p),\infty}(\RR^{2d};B,m_2)$ (resp. of $FS_{A_p,\rho}^{\{M_p\},\infty}(\RR^{2d};B,h_1)$ and $FS_{A_p,\rho}^{\{M_p\},\infty}(\RR^{2d};B,h_2)$), then $U_1\#U_2$ is a bounded subset of $FS_{A_p,\rho}^{(M_p),\infty}(\RR^{2d};B,m)$ for some $m>0$ (resp. of $FS_{A_p,\rho}^{\{M_p\},\infty}(\RR^{2d};B,h)$ for some $h>0$). As $FS_{A_p,\rho}^{(M_p),\infty}(\RR^{2d};B,m)$ and $FS_{A_p,\rho}^{\{M_p\},\infty}(\RR^{2d};B,h)$ are $(F)$-spaces it follows that the mappings
\beqs
&{}&\#: FS_{A_p,\rho}^{(M_p),\infty}(\RR^{2d};B,m_1)\times FS_{A_p,\rho}^{(M_p),\infty}(\RR^{2d};B,m_2)\rightarrow FS_{A_p,\rho}^{(M_p),\infty}(\RR^{2d};B,m)\,\, \mbox{and}\\
&{}&\#: FS_{A_p,\rho}^{\{M_p\},\infty}(\RR^{2d};B,h_1)\times FS_{A_p,\rho}^{\{M_p\},\infty}(\RR^{2d};B,h_2)\rightarrow FS_{A_p,\rho}^{\{M_p\},\infty}(\RR^{2d};B,h)
\eeqs
are continuous. Since $FS_{A_p,\rho}^{*,\infty}(\RR^{2d};B)$ is barrelled, this also proves that
\beqs
\#: FS_{A_p,\rho}^{*,\infty}(\RR^{2d};B)\times FS_{A_p,\rho}^{*,\infty}(\RR^{2d};B)\rightarrow FS_{A_p,\rho}^{*,\infty}(\RR^{2d};B)
\eeqs
is hypocontinuous.\\
\indent Given $\sum_j a_j, \sum_j b_j, \sum_j c_j\in FS_{A_p,B_p,\rho}^{*,\infty}(\RR^{2d};B)$ one easily verifies that\\
$\left(\sum_j a_j \# \sum_j b_j\right)\# \sum_j c_j =\sum_j p_j$, where
\beqs
p_j=\sum_{r=0}^j \sum_{s+k+l=r}\sum_{|\alpha+\beta+\gamma+\delta|=j-r}\frac{(-1)^{|\beta+\delta|}}{\alpha!\beta!\gamma!\delta!2^{j-r}} \partial^{\gamma}_{\xi} D^{\delta}_x \left(\partial^{\alpha}_{\xi} D^{\beta}_x a_s\cdot \partial^{\beta}_{\xi} D^{\alpha}_x b_k\right)\cdot\partial^{\delta}_{\xi}D^{\gamma}_x c_l.
\eeqs
Similarly, $\sum_j a_j \# \left(\sum_j b_j\# \sum_j c_j\right) =\sum_j \tilde{p}_j$, where
\beqs
\tilde{p}_j=\sum_{r=0}^j \sum_{s+k+l=r}\sum_{|\alpha+\beta+\gamma+\delta|=j-r}\frac{(-1)^{|\beta+\delta|}}{\alpha!\beta!\gamma!\delta!2^{j-r}} \partial^{\alpha}_{\xi} D^{\beta}_x a_s\cdot \partial^{\beta}_{\xi} D^{\alpha}_x \left(\partial^{\gamma}_{\xi} D^{\delta}_x b_k\cdot\partial^{\delta}_{\xi}D^{\gamma}_x c_l\right).
\eeqs
Let $f_{s,k,l}(z,\zeta,y,\eta,u,v)=a_s(z,\zeta)b_k(y,\eta)c_l(u,v)$. One verifies that
\begin{multline*}
p_j(x,\xi)=\sum_{r=0}^j \sum_{s+k+l=r}\frac{1}{2^{j-r}(j-r)!}\\
\cdot\left(\partial_{\zeta}\cdot D_y-\partial_{\eta}\cdot D_z+(\partial_{\zeta}+\partial_{\eta})\cdot D_u-\partial_v\cdot(D_z+D_y)\right)^{j-r}f_{s,k,l}(z,\zeta,y,\eta,u,v)\Bigg|\substack{z=x\\ y=x\\ u=x\\ \zeta=\xi\\ \eta=\xi\\ v=\xi}
\end{multline*}
and similarly,
\begin{multline*}
\tilde{p}_j(x,\xi)=\sum_{r=0}^j \sum_{s+k+l=r}\frac{1}{2^{j-r}(j-r)!}\\
\cdot\left(\partial_{\eta}\cdot D_u-\partial_v\cdot D_y+\partial_{\zeta}\cdot(D_y+D_u)-(\partial_{\eta}+\partial_v)\cdot D_z\right)^{j-r}f_{s,k,l}(z,\zeta,y,\eta,u,v)\Bigg|\substack{z=x\\ y=x\\ u=x\\ \zeta=\xi\\ \eta=\xi\\ v=\xi}.
\end{multline*}
It is easy to check that
\begin{multline*}
\partial_{\zeta}\cdot D_y-\partial_{\eta}\cdot D_z+(\partial_{\zeta}+\partial_{\eta})\cdot D_u-\partial_v\cdot(D_z+D_y)\\
=(\partial_{\eta}\cdot D_u-\partial_v\cdot D_y+\partial_{\zeta}\cdot(D_y+D_u)-(\partial_{\eta}+\partial_v)\cdot D_z.
\end{multline*}
Hence, $p_j=\tilde{p}_j$. Thus $\left(\sum_j a_j \# \sum_j b_j\right)\# \sum_j c_j=\sum_j a_j \# \left(\sum_j b_j\# \sum_j c_j\right)$, i.e. $\#$ is associative. Clearly, $\#$ is distributive over the addition and $\mathbf{1}$ is the $\#$ identity. Thus we have the following result.

\begin{proposition}\label{hyporingg}
For each $B\geq 0$, $FS_{A_p,\rho}^{*,\infty}(\RR^{2d};B)$ is a ring with the pointwise addition and multiplication given by $\#$. Moreover, the multiplication
\beqs
\#:FS_{A_p,\rho}^{*,\infty}(\RR^{2d};B)\times FS_{A_p,\rho}^{*,\infty}(\RR^{2d};B)\rightarrow FS_{A_p,\rho}^{*,\infty}(\RR^{2d};B)
\eeqs
is hypocontinuous.
\end{proposition}

\begin{remark} Given $\sum_j a_j\in FS_{A_p,\rho}^{*,\infty}(\RR^{2d};B)$, $(\sum_j a_j)^{\#1}$ will mean just $\sum_j a_j$ and, for $k\in\ZZ_+$, we define $(\sum_j a_j)^{\#(k+1)}$ as $(\sum_j a_j)^{\#k}\#\sum_j a_j$. Since $\#$ is associative, $(\sum_j a_j)^{\#k}$, $k\in\ZZ_+$, is just the $\#$ product of $\sum_j a_j$ with itself $k$ times.
\end{remark}

\section{Hypoelliptic symbols over $L^2(\RR^d)$}\label{section4}

We recall from \cite{CPP} the definition of hypoelliptic symbols in $\Gamma_{A_p,\rho}^{*,\infty}(\RR^{2d})$.

\begin{definition}(\cite[Definition 1.1]{CPP})
Let $a\in\Gamma^{*,\infty}_{A_p,\rho}(\RR^{2d})$. We say that $a$ is $\Gamma^{*,\infty}_{A_p,\rho}$-hypoelliptic if
\begin{itemize}
\item[$i$)] there exists $B>0$ such that there exist $c,m>0$ (resp. for every $m>0$ there exists $c>0$) such that
\beq\label{dd1}
|a(x,\xi)|\geq c e^{-M(m|x|)-M(m|\xi|)},\quad (x,\xi)\in Q^c_B
\eeq
\item[$ii$)] there exists $B>0$ such that for every $h>0$ there exists $C>0$ (resp. there exist $h,C>0$) such that
\beq\label{dd2}
\left|D^{\alpha}_{\xi}D^{\beta}_x a(x,\xi)\right|\leq C\frac{h^{|\alpha|+|\beta|}|a(x,\xi)|A_{\alpha}A_{\beta}}{\langle(x,\xi)\rangle^{\rho(|\alpha|+|\beta|)}},\,\, \alpha,\beta\in\NN^d,\, (x,\xi)\in Q^c_B.
\eeq
\end{itemize}
\end{definition}

\begin{proposition}\label{parametweyl}
Let $a\in\Gamma^{*,\infty}_{A_p,\rho}(\RR^{2d})$ be hypoelliptic. Define $q_0(w)=a(w)^{-1}$ on $Q^c_B$ and inductively, for $j\in\ZZ_+$,
\beqs
q_j(x,\xi)=-q_0(x,\xi)\sum_{s=1}^j\sum_{|\alpha+\beta|=s}\frac{(-1)^{|\beta|}}{\alpha!\beta!2^s}\partial^{\alpha}_{\xi} D^{\beta}_x q_{j-s}(x,\xi) \partial^{\beta}_{\xi} D^{\alpha}_x a(x,\xi),\,\, (x,\xi)\in Q^c_B.
\eeqs
Then, for every $h>0$ there exists $C>0$ (resp. there exist $h,C>0$) such that
\beq\label{estofpara}
\left|D^{\alpha}_w q_j(w)\right|\leq C\frac{h^{|\alpha|+2j}A_{|\alpha|+2j}}{|a(w)|\langle w\rangle^{\rho(|\alpha|+2j)}},\,\, w\in Q^c_B,\,\alpha\in\NN^{2d},\, j\in\NN.
\eeq
If $B\leq 1$, then $(\sum_j q_j)\# a= \mathbf{1}$ in $FS_{A_p,\rho}^{*,\infty}(\RR^{2d};0)$. If $B>1$, one can extend $q_0$ to an element of $\Gamma^{*,\infty}_{A_p,\rho}(\RR^{2d})$ by modifying it on $Q_{B'}\backslash Q^c_B$, for $B'>B$. In this case $\sum_j q_j\in FS_{A_p,\rho}^{*,\infty}(\RR^{2d};B')$, $((\sum_j q_j)\# a)_k=0$ on $Q^c_{B'}$, $\forall k\in\ZZ_+$, and $((\sum_j q_j)\#a)_0-1=q_0a-1$ belongs to $\DD^{(A_p)}(\RR^{2d})$ (resp. $\DD^{\{A_p\}}(\RR^{2d})$).\\
\indent In particular, for $q\sim\sum_j q_j$ there exists $*$-regularising operator $T$ such that $q^wa^w=\mathrm{Id}+T$.
\end{proposition}

\begin{proof} The estimate (\ref{estofpara}) for $j=0$ follows from \cite[Lemma 3.3]{CPP} and for $j\geq 1$ the proof is analogous to the proof of \cite[Lemma 3.4]{CPP}. The rest is easy and we omit it (see Corollary \ref{corweylqu} for the very last part).
\end{proof}

Given $a\in\Gamma^{*,\infty}_{A_p,\rho}(\RR^{2d})$, denote by $A$ the unbounded operator on $L^2(\RR^d)$ with domain $\SSS^*(\RR^d)$ defined as $A \varphi=a^w\varphi$, $\varphi\in\SSS^*(\RR^d)$. The restriction of $a^w$ to the subspace
\beqs
\{g\in L^2(\RR^d)|\, a^wg\in L^2(\RR^d)\}
\eeqs
defines a closed extension of $A$ which is called the maximal realisation of $A$. We denote by $\overline{A}$ the closure of $A$; it is also called the minimal realisation of $A$.

\begin{lemma}\label{operatoronl2}
Let $V\subseteq \Gamma_{A_p,\rho}^{*,\infty}(\RR^{2d})$. Assume that for every $h>0$ there exists $C>0$ (resp. there exist $h,C>0$) such that
\beq\label{estforope}
\left|D^{\alpha}_w b(w)\right|\leq Ch^{|\alpha|}A_{\alpha}\langle w\rangle^{-\rho|\alpha|},\,\, w\in\RR^{2d},\,\alpha\in \NN^{2d},\, b\in V.
\eeq
Then, for each $b\in V$, $b^w$ extends to a bounded operator on $L^2(\RR^d)$ and the set $\{b^w|\, b\in V\}$ is bounded in $\mathcal{L}_b(L^2(\RR^d),L^2(\RR^d))$.\\
\indent If $\{b_{\lambda}\}_{\lambda\in\Lambda}\subseteq V$ is a net which converges to $b_0\in V$ in the topology of $\Gamma_{A_p,\rho}^{*,\infty}(\RR^{2d})$, then $b_{\lambda}^w\rightarrow b_0^w$ in $\mathcal{L}_p(L^2(\RR^d),L^2(\RR^d))$.
\end{lemma}

\begin{proof} By Proposition \ref{changequa} (applied with $f(w)=1$) for each $b\in V$ there exists $c_b\in \Gamma_{A_p,\rho}^{*,\infty}(\RR^{2d})$, such that $\{b^w-c_b(x,D)|\, b\in V\}$ is an equicontinuous $*$-regularising set and for every $h>0$ there exists $C>0$ (resp. there exist $h,C>0$) such that
\beqs
\left|D^{\alpha}_w c_b(w)\right|\leq C h^{|\alpha|}A_{\alpha}\langle w\rangle^{-\rho|\alpha|},\,\, w\in\RR^{2d},\, \alpha\in \NN^{2d},\, b\in V.
\eeqs
Then, \cite[Theorem 1.7.14]{NR} implies that $c_b(x,D)$ extends to a continuous operator on $L^2(\RR^d)$ and the set $\{c_b(x,D)|\, b\in V\}$ is bounded in $\mathcal{L}_b(L^2(\RR^d),L^2(\RR^d))$. Thus, the same holds for $\{b^w|\, b\in V\}$. If $\{b_{\lambda}\}_{\lambda\in\Lambda}\subseteq V$ is as in the statement, Proposition \ref{continuity} implies $b_{\lambda}^w\varphi\rightarrow b_0^w\varphi$ for each $\varphi\in\SSS^*(\RR^d)$, i.e. $b_{\lambda}^w\rightarrow b_0^w$ in the topology of simple convergence on the total subset $\SSS^*(\RR^d)$ of $L^2(\RR^d)$. Since $\{b^w|\, b\in V\}$ is equicontinuous in $\mathcal{L}(L^2(\RR^d),L^2(\RR^d))$, the Banach-Steinhaus theorem \cite[Theorem 4.5, p. 85]{Sch} implies $b_{\lambda}^w\rightarrow b^w_0$ in $\mathcal{L}_p(L^2(\RR^d),L^2(\RR^d))$.
\end{proof}

Now we are ready to prove that the minimal and maximal realisation of the Weyl quantisation of a hypoelliptic symbol coincide. This is classical result for the finite order operators, in the distributional setting. The following proposition proves that this remains true even for infinite order operators. This result is instrumental for the rest of the article and we will often tacitly apply it in what follows.

\begin{proposition}\label{maximalreal}
Let $a$ be hypoelliptic and $A$ be the corresponding unbounded operator on $L^2(\RR^d)$ defined above. Then the minimal realisation $\overline{A}$ coincides with the maximal realisation. Moreover, $\overline{A}$ coincides with the restriction of $a^w$ on the domain of $\overline{A}$.
\end{proposition}

\begin{proof} Let $g\in L^2(\RR^d)$ be such that $a^w g\in L^2(\RR^d)$. We will construct a sequence $\{\varphi_j\}_{j\in\ZZ_+}\subseteq \SSS^*(\RR^d)$ such that $\varphi_j\rightarrow g$ and $a^w \varphi_j\rightarrow a^w g$ in $L^2(\RR^d)$ which will complete the proof. Fix $\psi\in \DD^{(A_p)}(\RR^{2d})$ (resp. $\psi\in\DD^{\{A_p\}}(\RR^{2d})$) such that $0\leq \psi(w)\leq 1$, $\psi(w)=1$ on $|w|\leq 1$ and $\mathrm{supp}\, \psi\subseteq \{w\in\RR^{2d}|\, |w|\leq 2\}$. For $n\in\ZZ_+$ define $\psi_n(w)=\psi(w/n)$. One easily verifies that $\psi_n\in \Gamma_{A_p,\rho}^{\{M_p\},\infty}(\RR^{2d})$, $\forall n\in\ZZ_+$ and they satisfy the following estimate: for every $h>0$ there exists $C>0$ (resp. there exist $h,C>0$) such that
\beqs
\left|D^{\alpha}_w \psi_j(w)\right|\leq C h^{|\alpha|}A_{\alpha}\langle w\rangle^{-\rho|\alpha|},\,\, w\in\RR^{2d},\,\alpha\in \NN^{2d},\, j\in\NN.
\eeqs
Also, it is easy to check that $\psi_n\rightarrow 1_w$ in $\Gamma_{A_p,\rho}^{\{M_p\},\infty}(\RR^{2d})$. The operators $\psi_n^w$ are $*$-regularising. Put $\varphi_n=\psi_n^wg\in\SSS^*(\RR^d)$. By Lemma \ref{operatoronl2}, $\varphi_n\rightarrow g$ in $L^2(\RR^d)$. It remains to prove $a^w\varphi_n\rightarrow a^wg$ in $L^2(\RR^d)$. For the given $a$, let $\sum_j q_j$ be as in Proposition \ref{parametweyl}. Fix $q\sim\sum_j q_j$ and $*$-regularising operator $T$ such that $q^wa^w=\mathrm{Id}+T$. Then $a^w\varphi_n=a^w\psi_n^wq^wa^wg-a^w\psi_n^wTg$. Since $T$ is $*$-regularising, Proposition \ref{continuity} implies $a^w\psi_n^wTg\rightarrow a^wTg$ in $\SSS^*(\RR^d)$, hence, also in $L^2(\RR^d)$. Since $\{a\#\psi_n\#q|\, n\in\ZZ_+\}\precsim 1$, by Corollary \ref{corweylqu} one can find $b_n$, $n\in\ZZ_+$, for which the uniform estimate (\ref{estforope}) holds and the set $\{T_n=a^w\psi_n^wq^w-b_n^w|\, n\in\ZZ_+\}$ is equicontinuous $*$-regularising. Lemma \ref{operatoronl2} verifies that $\{b_n^w\}_{n\in\ZZ_+}$ extends to a bounded subset of $\mathcal{L}_b(L^2(\RR^d),L^2(\RR^d))$. Hence the operators $V_n=a^w\psi_n^wq^w$, $n\in\ZZ_+$, extend to $L^2(\RR^d)$ and the set $\{V_n\}_{n\in\ZZ_+}$ is bounded in $\mathcal{L}_b(L^2(\RR^d),L^2(\RR^d))$. By Corollary \ref{corweylqu}, there exists $v$ for which (\ref{estforope}) holds and $a^wq^w-v^w$ is $*$-regularising operator. Lemma \ref{operatoronl2} yields that $v^w$ extends to a bounded operator on $L^2(\RR^d)$ and hence, so does $a^wq^w$. Since for each $\chi\in\SSS^*(\RR^d)$, $V_n\chi\rightarrow a^wq^w\chi$ (cf. Proposition \ref{continuity}), $V_n$ converges to the operator $a^wq^w\in\mathcal{L}(L^2(\RR^d),L^2(\RR^d))$ on the total subset $\SSS^*(\RR^d)$ of $L^2(\RR^d)$. Now, the Banach-Steinhaus theorem \cite[Theorem 4.5, p. 85]{Sch} implies $V_n\rightarrow a^wq^w$ in $\mathcal{L}_{\sigma}(L^2(\RR^d),L^2(\RR^d))$. Hence $V_na^wg\rightarrow a^wq^wa^wg$ in $L^2(\RR^d)$.
\end{proof}

\section{Complex powers of hypoelliptic operators}\label{section5}

Following Komatsu \cite{KNonnegative}, given a $(B)$-space $X$, a closed operator $A:D(A)\subseteq X\rightarrow X$ is said to be non-negative if $(-\infty,0)$ is contained in the resolvent set of $A$ and
\beqs
\sup_{\lambda\in\RR_+} \lambda\|(A+\lambda\mathrm{Id})^{-1}\|_{\mathcal{L}_b(X,X)}<\infty.
\eeqs
In this case, for $z\in\CC_+$ and $v\in D(A^{[\mathrm{Re}\,z]+1})$, the function $\lambda\mapsto \lambda^{z-1}\left(A(A+\lambda\mathrm{Id})^{-1}\right)^k v$, $\RR_+\rightarrow X$, is Bochner integrable for all integers $k>\mathrm{Re}\,z$ and by defining
\beqs
I^z_{A,k} v=\gamma_k(z)\int_0^{\infty} \lambda^{z-1}\left(A(A+\lambda\mathrm{Id})^{-1}\right)^kv d\lambda,\,\, v\in D(A^{[\mathrm{Re}\,z]+1}),\,\,\, k>\mathrm{Re}\, z,
\eeqs
where $\gamma_k(z)=\Gamma(k)/(\Gamma(z)\Gamma(k-z))$, we have that $I^z_{A,k+1} v=I^z_{A,k} v$, for all integers $k>\mathrm{Re}\,z$ (see \cite[Proposition 3.1.3, p. 59]{Fractionalpowersbook}). The operator
\beqs
J^z_A:D(J^z_A)=D(A^{[\mathrm{Re}\,z]+1})\subseteq X\rightarrow X,\,\, J^z_A v=I^z_{A,k} v,\,\, \mbox{for any}\,\, k>\mathrm{Re}\,z,
\eeqs
is closable (cf. \cite[Theorem 3.1.8, p. 64]{Fractionalpowersbook}). Balakrishnan defines the power of $A$ with exponent $z$ as the operator $\overline{J^z_A}$. If in addition $A$ is densely defined then $A^{z+\zeta}=A^zA^{\zeta}$, $\forall z,\zeta\in\CC_+$ (in particular, $A^k=\underbrace{A\ldots A}_{k}$) and $\sigma(A^z)=\{\zeta^z|\, \zeta\in\sigma(A)\}$; where $\zeta^z$ is defined by the principal branch of logarithm and we put $0^z=0$ (cf. \cite[Corollary 5.1.12, p. 110]{Fractionalpowersbook} and \cite[Theorem 5.3.1, p. 116]{Fractionalpowersbook}).

\subsection{Statement of the main result}\label{subse}
Let $a\in\Gamma_{A_p,\rho}^{*,\infty}(\RR^{2d})$ be hypoelliptic, where the hypoellipticity conditions (\ref{dd1}) and (\ref{dd2}) hold for some $\tilde{B}>0$. We impose the following conditions on $a$:
\begin{itemize}
\item[$(I)$] $\mathrm{Re}\,a(w)\geq -\tilde{B}|\mathrm{Im}\,a(w)|$ for $w\in Q^c_{\tilde{B}}$;
\item[$(II)$] the densely defined operator $A:\SSS^*(\RR^d)\subseteq L^2(\RR^d)\rightarrow L^2(\RR^d)$, $A=a^w|_{\SSS^*(\RR^d)}$, is such that $\overline{A}$ is non-negative.
\end{itemize}
Let $\tilde{\chi}\in\DD^{(A_p)}(\RR^{2d})$ in the $(M_p)$ case and $\tilde{\chi}\in\DD^{\{A_p\}}(\RR^{2d})$ in the $\{M_p\}$ case respectively, be such that $\tilde{\chi}\geq 0$ and $\tilde{\chi}(w)>\max\{0,-\mathrm{Re}\,a(w)\}$ when $w\in \overline{Q_{\tilde{B}}}$. Denote $a_0=a+\tilde{\chi}$. One easily verifies that, possibly for a larger $\tilde{B}$,
\beq\label{finqa}
\mathrm{Re}\,a_0(w)> -\tilde{B}|\mathrm{Im}\,a_0(w)|,\,\, \forall w\in\RR^{2d}
\eeq
and consequently
\beq\label{sinqa}
|a_0(w)|\leq \sqrt{1+\tilde{B}^2}|a_0(w)+\lambda|,\,\, \lambda \leq \sqrt{1+\tilde{B}^2}|a_0(w)+\lambda|,
\eeq
for all $w\in \RR^{2d}$ and $\lambda \geq 0$. As a consequence of (\ref{finqa}), $a_0$ never vanishes and for any $z\in \CC_+$, the function $w\mapsto (a_0(w))^z$, $\RR^{2d}\rightarrow \CC$, is a well defined $\mathcal{C}^{\infty}$ function (in $(a_0(w))^z$ we use the principal branch of the logarithm). Notice that (\ref{sinqa}) also implies the existence of $c>0$ such that
\beq\label{krt558811}
|a_0(w)|+\lambda\leq c|a_0(w)+\lambda|,\,\, w\in\RR^{2d},\, \lambda\geq 0.
\eeq
Thus, by using the identity
\beq\label{inteq}
\int_0^{\infty} \frac{\lambda^{z-1}\zeta^k}{(\zeta+\lambda)^k}d\lambda=\frac{\zeta^z}{\gamma_k(z)},
\eeq
which is valid for $z\in\CC_+$, $k\in\NN$ with $k>\mathrm{Re}\, z$ and $\zeta\in\CC\backslash\{0\}$ with $|\arg \zeta|<\pi$, we deduce that
\beqs
\left|(a_0(w))^z\right|\leq c^k|\gamma_k(z)|\int_0^{\infty}\frac{\lambda^{\mathrm{Re}\,z-1}|a_0(w)|^k}{(|a_0(w)|+\lambda)^k}d\lambda= \frac{c^k|\gamma_k(z)||a_0(w)|^{\mathrm{Re}\, z}}{\gamma_k(\mathrm{Re}\, z)},
\eeqs
for all $w\in\RR^{2d}$, $z\in\CC_+$, $k>\mathrm{Re}\, z$, $k\in\ZZ_+$. Thus, $e^{-\mathrm{Im}\, z\arg(a_0(w))}\leq c^k|\gamma_k(z)|/\gamma_k(\mathrm{Re}\, z)$. Hence,
\beq\label{thk779911}
|a_0(w)|^{\mathrm{Re}\, z}=\left|(a_0(w))^z\right|e^{\mathrm{Im}\, z\arg(a_0(w))}\leq c^k\left|(a_0(w))^z\right||\gamma_k(\bar{z})|/\gamma_k(\mathrm{Re}\, z),
\eeq
for all $w\in\RR^{2d}$, $z\in\CC_+$, $k>\mathrm{Re}\, z$, $k\in\ZZ_+$. The main result is the following theorem.

\begin{theorem}\label{maint}
Let $a\in\Gamma_{A_p,\rho}^{*,\infty}(\RR^{2d})$ be a hypoelliptic symbol that satisfies $(I)$ and $(II)$ and let $a_0$ and $A$ be defined as above. Then, for every $z\in\CC_+$ there exists $a_0^{\# z}\in FS_{A_p,\rho}^{*,\infty}(\RR^{2d};0)$ such that the following conditions hold.
\begin{itemize}
\item[$(i)$] $a_0^{\# z}\#a_0^{\#\zeta}=a_0^{\# (z+\zeta)}=a_0^{\#\zeta}\#a_0^{\# z}$, $\forall z,\zeta\in\CC_+$.
\item[$(ii)$] When $z=k\in\ZZ_+$, $a_0^{\# z}$ is just $\underbrace{a_0\#\ldots\#a_0}_{k}=a_0^{\# k}$. In particular, for $z=1$, $a_0^{\# z}$ is just $a_0$.
\item[$(iii)$] The mapping $z\mapsto a_0^{\# z}$, $\CC_+\rightarrow FS_{A_p,\rho}^{*,\infty}(\RR^{2d};0)$ is continuous.
\item[$(iv)$] $(a_0^{\# z})_0(w)=(a_0(w))^z$, $w\in\RR^{2d}$, $z\in\CC_+$.
\end{itemize}
For each fixed vertical strip $\CC_{+,t}=\{\zeta\in\CC_+|\, \mathrm{Re}\,\zeta\leq t\}$, $t>0$, and $k=[t]+1\in\ZZ_+$, the following estimate holds: for every $h>0$ there exists $C>0$ (resp. there exist $h,C>0$) such that
\beq\label{tkr554411}
\left|D^{\alpha}_w (a_0^{\# z})_j(w)\right|\leq C\frac{h^{|\alpha|+2j}A_{|\alpha|+2j}|\gamma_k(z)||a_0(w)|^{\mathrm{Re}\, z}}{\gamma_k(\mathrm{Re}\, z)\langle w\rangle^{\rho(|\alpha|+2j)}},
\eeq
for all $w\in\RR^{2d}$, $\alpha\in\NN^{2d}$, $j\in\NN$, $z\in\CC_{+,t}$. Furthermore, there exists $R_t>0$ such that $a^{\uwidehat{z}}:=R_t(a_0^{\# z})$, $z\in \CC_{+,t}$, are hypoelliptic symbols in $\Gamma_{A_p,\rho}^{*,\infty}(\RR^{2d})$ and the following conditions hold:
\begin{itemize}
\item[$(v)$] There exists $B_t>0$ such that for every $h>0$ (resp. for some $h>0$)
\begin{gather*}
\sup_{\substack{N\in\ZZ_+\\z\in\CC_{+,t}}}\sup_{\alpha\in\NN^{2d}}\sup_{w\in Q^c_{B_tm_N}}\frac{\left|D^{\alpha}_wa^{\uwidehat{z}}(w)-D^{\alpha}_w \sum_{j<N}(a_0^{\# z})_j(w)\right|\langle w\rangle^{\rho(|\alpha|+2N)}\gamma_k(\mathrm{Re}\, z)}{h^{|\alpha|+2N}A_{|\alpha|+2N}|a_0(w)|^{\mathrm{Re}\, z}|\gamma_k(z)|}<\infty,\\
\sup_{z\in\CC_{+,t}}\sup_{\alpha\in\NN^{2d}}\sup_{w\in \RR^{2d}}\frac{\left|D^{\alpha}a^{\uwidehat{z}}(w)\right|\langle w\rangle^{\rho|\alpha|}\gamma_k(\mathrm{Re}\, z)} {h^{|\alpha|}A_{\alpha}|a_0(w)|^{\mathrm{Re}\, z}|\gamma_k(z)|}<\infty.
\end{gather*}
\item[$(vi)$] $D(\overline{A}^z)=\left\{v\in L^2(\RR^d)|\, (a^{\uwidehat{z}})^w v\in L^2(\RR^d)\right\}$ and there exist $*$-regularising operators $S^{\uwidehat{z}}$, $z\in\CC_{+,t}$, such that $\{\gamma_k(\mathrm{Re}\, z)(\gamma_k(z))^{-1}S^{\uwidehat{z}}|\, z\in\CC_{+,t}\}$ is an equicontinuous subset of $\mathcal{L}(\SSS'^*(\RR^d),\SSS^*(\RR^d))$ and
    \beqs
    \overline{A}^z=\overline{(a^{\uwidehat{z}})^w|_{\SSS^*(\RR^d)}}+S^{\uwidehat{z}}.
    \eeqs
    Moreover, for each $v\in L^2(\RR^d)$, $z\mapsto S^{\uwidehat{z}}v$, $\mathrm{int}\,\CC_{+,t}\rightarrow L^2(\RR^d)$, is analytic and for each $v\in D(A^k)$, $z\mapsto (a^{\uwidehat{z}})^w v$, $\mathrm{int}\,\CC_{+,t}\rightarrow L^2(\RR^d)$, is analytic.
\end{itemize}
\end{theorem}

\begin{remark}\label{remark1144}
Notice that the second estimate in $(v)$ together with (\ref{thk779911}) implies that for every $h>0$ (resp. for some $h>0$
\beq\label{est1155776622}
\sup_{z\in\CC_{+,t}}\sup_{\alpha\in\NN^{2d}}\sup_{w\in \RR^{2d}}\frac{\left|D^{\alpha}a^{\uwidehat{z}}(w)\right|\langle w\rangle^{\rho|\alpha|}(\gamma_k(\mathrm{Re}\, z))^2} {h^{|\alpha|}A_{\alpha}|(a_0(w))^z||\gamma_k(z)|^2}<\infty.
\eeq
\end{remark}

\subsection{Proof of Theorem \ref{maint}}

This subsection is devoted to the proof of the main result. From now on $a, a_0\in \Gamma_{A_p,\rho}^{*,\infty}(\RR^{2d})$ and the non-negative operator $\overline{A}$ are as in Subsection \ref{subse}. For the simplicity in notation, for $\lambda\geq 0$ we put $a_0(w)+\lambda$ by $a_{\lambda}(w)$.\\
\indent Since the proof of Theorem \ref{maint} is rather lengthy, we will divide it into four parts.

\subsubsection{The parametrix of $a_{\lambda}^w$}

As $\tilde{\chi}$ has a compact support, (\ref{sinqa}) implies that there exist $c,C,m>0$ (resp. for every $m>0$ there exist $c,C>0$) such that
\beq\label{growthal}
c(1+\lambda)e^{-M(m|\xi|)}e^{-M(m|x|)}\leq |a_{\lambda}(x,\xi)|\leq C(1+\lambda)e^{M(m|\xi|)}e^{M(m|x|)},
\eeq
for all $(x,\xi)\in\RR^{2d}$, $\lambda\geq 0$. In the Roumieu case, this inequality together with Lemma \ref{lemulgr117} implies that there exist $(k_p)\in\mathfrak{R}$ and $c,C>0$ such that
\beq\label{growthalr}
c(1+\lambda)e^{-N_{k_p}(|\xi|)}e^{-N_{k_p}(|x|)}\leq |a_{\lambda}(x,\xi)|\leq C(1+\lambda)e^{N_{k_p}(|\xi|)}e^{N_{k_p}(|x|)},
\eeq
for all $(x,\xi)\in\RR^{2d}$, $\lambda\geq 0$. Moreover, for every $h>0$ there exists $C>0$ (resp. there exist $h,C>0$) such that
\beq\label{estfora_0}
\left|D^{\alpha}_w a_0(w)\right|\leq Ch^{|\alpha|}|a_0(w)|A_{\alpha}\langle w\rangle^{-\rho|\alpha|},\,\, \alpha\in\NN^{2d},\, w\in \RR^{2d}.
\eeq
This and (\ref{sinqa}), imply that for every $h>0$ there exists $C>0$ (resp. there exist $h,C>0$) such that
\beq\label{estforahi}
\left|D^{\alpha}_w a_{\lambda}(w)\right|\leq Ch^{|\alpha|}|a_{\lambda}(w)|A_{\alpha}\langle w\rangle^{-\rho|\alpha|},\,\, \alpha\in\NN^{2d},\, w\in \RR^{2d},\, \lambda \geq 0.
\eeq
Now, by the same technique as in \cite[Lemma 3.3]{CPP} one proves that for every $h>0$ there exists $C>0$ (resp. there exist $h,C>0$) such that
\beq\label{uniestfoa}
\left|D^{\alpha}_w \left(a_{\lambda}(w)\right)^{-1}\right|\leq Ch^{|\alpha|}A_{\alpha}|a_{\lambda}(w)|^{-1}\langle w\rangle^{-\rho|\alpha|},
\eeq
for all $\alpha\in\NN^{2d}$, $w\in \RR^{2d}$, $\lambda\geq 0$. For each $\lambda\geq 0$ define $q^{(\lambda)}_0(w)=1/a_{\lambda}(w)$, $w\in\RR^{2d}$, and inductively
\beqs
q^{(\lambda)}_j(x,\xi)=-q^{(\lambda)}_0(x,\xi)\sum_{s=1}^j\sum_{|\alpha+\beta|=s} \frac{(-1)^{|\beta|}}{\alpha!\beta!2^s}\partial^{\alpha}_{\xi} D^{\beta}_x q^{(\lambda)}_{j-s}(x,\xi) \partial^{\beta}_{\xi} D^{\alpha}_x a_{\lambda}(x,\xi),\,\, (x,\xi)\in \RR^{2d}.
\eeqs
Because of the uniform estimates in $\lambda$ given in (\ref{uniestfoa}), analogously as in the proof of Proposition \ref{parametweyl} (cf. the proof of \cite[Lemma 3.4]{CPP}), we can conclude the following estimate: for every $h>0$ there exists $C>0$ (resp. there exist $h,C>0$) such that
\beq\label{estforthepar}
\left|D^{\alpha}_w q^{(\lambda)}_j(w)\right|\leq C\frac{h^{|\alpha|+2j}A_{|\alpha|+2j}}{|a_{\lambda}(w)|\langle w\rangle^{\rho(|\alpha|+2j)}},\,\, w\in \RR^{2d},\,\alpha\in\NN^{2d},\, j\in\NN,\, \lambda\geq 0.
\eeq
This estimate together with (\ref{growthal}) in the Beurling case and (\ref{growthalr}) in the Roumieu case respectively, implies the following:
\begin{itemize}
\item[$-$] in the $(M_p)$ case, there exists $m>0$ such that for every $h>0$ there is $C>0$ such that
\beq\label{estuniforminlamb}
(1+\lambda)\left|D^{\alpha}_w q^{(\lambda)}_j(w)\right|\leq Ch^{|\alpha|+2j}A_{|\alpha|+2j}e^{M(m|\xi|)}e^{M(m|x|)}\langle w\rangle^{-\rho(|\alpha|+2j)},
\eeq
for all $w\in \RR^{2d}$, $\alpha\in\NN^{2d}$, $j\in\NN$, $\lambda\geq 0$;
\item[$-$] in the Roumieu case, there exist $(k_p)\in\mathfrak{R}$ and $h,C>0$ such that
\beq\label{estuniforminlamr}
(1+\lambda)\left|D^{\alpha}_w q^{(\lambda)}_j(w)\right|\leq Ch^{|\alpha|+2j}A_{|\alpha|+2j}e^{N_{k_p}(|\xi|)}e^{N_{k_p}(|x|)}\langle w\rangle^{-\rho(|\alpha|+2j)},
\eeq
for all $w\in \RR^{2d}$, $\alpha\in\NN^{2d}$, $j\in\NN$, $\lambda\geq 0$.
\end{itemize}
Thus, we conclude that
\beq
\Big\{\ssum (1+\lambda)q^{(\lambda)}_j\big|\, \lambda\geq 0\Big\}&\precsim& e^{M(m|\xi|)}e^{M(m|x|)}\,\, \mbox{in}\,\, FS_{A_p,\rho}^{(M_p),\infty}(\RR^{2d};0)\,\, \mbox{and}\label{ths7975}\\
\Big\{\ssum (1+\lambda)q^{(\lambda)}_j\big|\, \lambda\geq 0\Big\}&\precsim& e^{N_{k_p}(|\xi|)}e^{N_{k_p}(|x|)}\,\, \mbox{in}\,\, FS_{A_p,\rho}^{\{M_p\},\infty}(\RR^{2d};0)\label{ths7977}
\eeq
in the Beurling and the Roumieu case, respectively. Similarly, (\ref{estforahi}) and (\ref{growthal}) imply $\{a_{\lambda}/(1+\lambda)|\,\lambda\geq 0\}\precsim e^{M(m|\xi|)}e^{M(m|x|)}$ in the Beurling case and (\ref{estforahi}) and (\ref{growthalr}) imply $\{a_{\lambda}/(1+\lambda)|\,\lambda\geq 0\}\precsim e^{N_{k_p}(|\xi|)}e^{N_{k_p}(|x|)}$ in the Roumieu case. By a direct computation, one verifies $\sum_j q^{(\lambda)}_j\# a_{\lambda}=\mathbf{1}$ (cf. Proposition \ref{parametweyl}). Now, Corollary \ref{corweylqu} implies that there exist $R_1>0$ such that $\left\{\mathrm{Op}_{1/2}\left(R_1(\sum_j q^{(\lambda)}_j)\right)a_{\lambda}^w-\mathrm{Id}\big|\, \lambda\geq 0\right\}$ is an equicontinuous subset of $\mathcal{L}(\SSS'^*(\RR^d),\SSS^*(\RR^d))$ (note that $R_1(\sum_j (1+\lambda)q^{(\lambda)}_j)=(1+\lambda)R_1(\sum_j q^{(\lambda)}_j)$).\\
\indent Analogously, for each $\lambda\geq 0$ one can define $\tilde{q}^{(\lambda)}_0(w)=1/a_{\lambda}(w)(=q^{(\lambda)}_0(w))$, $w\in\RR^{2d}$, and inductively
\beqs
\tilde{q}^{(\lambda)}_j(x,\xi)=-\tilde{q}^{(\lambda)}_0(x,\xi)\sum_{s=1}^j\sum_{|\alpha+\beta|=s} \frac{(-1)^{|\beta|}}{\alpha!\beta!2^s}\partial^{\alpha}_{\xi} D^{\beta}_x a_{\lambda}(x,\xi) \partial^{\beta}_{\xi} D^{\alpha}_x \tilde{q}^{(\lambda)}_{j-s}(x,\xi),\,\, (x,\xi)\in \RR^{2d}.
\eeqs
Analogously as in Proposition \ref{parametweyl}, $\sum_j \tilde{q}_j^{(\lambda)}\in FS_{A_p,\rho}^{*,\infty}(\RR^{2d};0)$ and one easily verifies $a_{\lambda}\#\sum_j \tilde{q}^{(\lambda)}_j=\mathbf{1}$ in $FS_{A_p,\rho}^{*,\infty}(\RR^{2d};0)$. Since $\#$ is associative,
\beqs
\ssum q^{(\lambda)}_j=\ssum q^{(\lambda)}_j\#a_{\lambda}\#\ssum \tilde{q}^{(\lambda)}_j=\ssum\tilde{q}^{(\lambda)}_j.
\eeqs
By the same technique as above, one proves that there exists $R_2>0$ such that
\beqs
\left\{a_{\lambda}^w\mathrm{Op}_{1/2}\left(R_2(\ssum q^{(\lambda)}_j)\right)-\mathrm{Id}\big|\, \lambda\geq 0\right\}
\eeqs
is equicontinuous in $\mathcal{L}(\SSS'^*(\RR^d),\SSS^*(\RR^d))$. By taking $R=\max\{R_1, R_2\}$ we have the following result (taking larger $R_1$ or $R_2$ yields the same results because of Proposition \ref{eqsse}).

\begin{proposition}\label{uniformparam}
There exists $R>0$, which can be chosen arbitrarily large, such that
\beqs
\Big\{\mathrm{Op}_{1/2}\big(R(\ssum q^{(\lambda)}_j)\big)a_{\lambda}^w-\mathrm{Id}\big|\, \lambda\geq 0\Big\}\, \mbox{and}\, \Big\{a_{\lambda}^w\mathrm{Op}_{1/2}\big(R(\ssum q^{(\lambda)}_j)\big)-\mathrm{Id}\big|\, \lambda\geq 0\Big\}
\eeqs
are equicontinuous subsets of $\mathcal{L}(\SSS'^*(\RR^d),\SSS^*(\RR^d))$. Moreover, for $\{\sum_j q^{(\lambda)}_j\}_{\lambda\geq 0}$ the estimate (\ref{estforthepar}) holds.
\end{proposition}

\begin{remark} Observe that for each $\lambda\geq 0$, $a_0\#\sum_j q^{(\lambda)}_j=\mathbf{1}-\lambda\sum_j q^{(\lambda)}_j=\sum_j q^{(\lambda)}_j \# a_0$, i.e. $a_0$ and $\sum_j q^{(\lambda)}_j$ commute. We will often tacitly apply this fact throughout the rest of the article.
\end{remark}

Observe that $a^w+\lambda\mathrm{Id}-a_{\lambda}^w=(a-a_0)^w$, for all $\lambda\geq 0$, and the last operator is $*$-regularising since $a-a_0$ has compact support. Because of (\ref{ths7975}) and (\ref{ths7977}) we can conclude that the same also holds for $\{\sum_j q^{(\lambda)}_j|\, \lambda\geq 0\}$. Propositions \ref{subordinate} and \ref{continuity} prove the existence of $R'_0>0$ such that for each $R\geq R'_0$ the set $\big\{\mathrm{Op}_{1/2}\big(R(\sum_j q^{(\lambda)}_j)\big)\big|\, \lambda\geq 0\big\}$ is equicontinuous in $\mathcal{L}(\SSS^*(\RR^d),\SSS^*(\RR^d))$ and $\mathcal{L}(\SSS'^*(\RR^d),\SSS'^*(\RR^d))$. Hence, denoting by $\tilde{a}_{\lambda}$ the symbol $a+\lambda$ for $\lambda\geq 0$, as an easy consequence of Proposition \ref{uniformparam} we deduce the following corollary.

\begin{corollary}\label{corunipar}
There exists $R>0$, which can be chosen arbitrarily large, for which Proposition \ref{uniformparam} holds and furthermore
\beqs
\Big\{\mathrm{Op}_{1/2}\big(R(\ssum q^{(\lambda)}_j)\big)\tilde{a}_{\lambda}^w-\mathrm{Id}\big|\, \lambda\geq 0\Big\}\,\,\, \mbox{and}\,\,\, \Big\{\tilde{a}_{\lambda}^w\mathrm{Op}_{1/2}\big(R(\ssum q^{(\lambda)}_j)\big)-\mathrm{Id}\big|\, \lambda\geq 0\Big\}
\eeqs
are equicontinuous subsets of $\mathcal{L}(\SSS'^*(\RR^d),\SSS^*(\RR^d))$.
\end{corollary}

This corollary implies that $\tilde{a}^w_{\lambda}$ is globally regular for each $\lambda\geq 0$, i.e. for $u\in \SSS'^*(\RR^d)$ if $(a^w+\lambda\mathrm{Id})u\in \SSS^*(\RR^d)$ then $u\in\SSS^*(\RR^d)$. Since $\overline{A}+\lambda\mathrm{Id}:D(\overline{A})\subseteq L^2(\RR^d)\rightarrow L^2(\RR^d)$ is injective for $\lambda>0$, $a^w+\lambda\mathrm{Id}:\SSS^*(\RR^d)\rightarrow \SSS^*(\RR^d)$ is injective for $\lambda>0$. Clearly, it is also continuous. Since the range of $\overline{A}+\lambda\mathrm{Id}$ is $L^2(\RR^d)$, for given $\varphi\in\SSS^*(\RR^d)$ there exists $v\in L^2(\RR^d)$ such that $(\overline{A}+\lambda\mathrm{Id})v=\varphi$, i.e. $(a^w+\lambda\mathrm{Id})v=\varphi$. As $a^w+\lambda\mathrm{Id}$ is globally regular, $v\in\SSS^*(\RR^d)$. Thus $\tilde{a}_{\lambda}^w=a^w+\lambda\mathrm{Id}$ is continuous bijection on $\SSS^*(\RR^d)$. As $\SSS^{(M_p)}(\RR^d)$ is an $(F)$-space and $\SSS^{\{M_p\}}(\RR^d)$ is a $(DFS)$-space, $\SSS^*(\RR^d)$ is a Ptak space (see \cite[Sect. IV. 8, p. 162]{Sch}). The Ptak homomorphism theorem \cite[Corollary 1, p. 164]{Sch} implies that $\tilde{a}_{\lambda}^w$ is topological isomorphism on $\SSS^*(\RR^d)$, for each $\lambda>0$.

\subsubsection{The analysis of the operators $(\widetilde{a}_{\lambda}^w)^{-1}$ and $(a^w(\widetilde{a}_{\lambda}^w)^{-1})^k$, $k\in\ZZ_+$}

Fix a large enough $R>0$ such that the conclusions in Proposition \ref{uniformparam} and Corollary \ref{corunipar} hold and let $q_{\lambda}=R(\sum_j q^{(\lambda)}_j)$, for $\lambda>0$. Because of (\ref{ths7975}) and (\ref{ths7977}), $\{(1+\lambda)q_{\lambda}|\, \lambda>0\}$ is a bounded subset of $\Gamma_{A_p,\rho}^{(M_p),\infty}(\RR^{2d};m')$ for some $m'>0$ (resp. of $\Gamma_{A_p,\rho}^{\{M_p\},\infty}(\RR^{2d};h')$ for some $h'>0$) (cf. Proposition \ref{subordinate}) and consequently $\{(1+\lambda)q_{\lambda}^w|\, \lambda>0\}$ is an equicontinuous subset of $\mathcal{L}(\SSS^*(\RR^d), \SSS^*(\RR^d))$ and $\mathcal{L}(\SSS'^*(\RR^d),\SSS'^*(\RR^d))$. As $\tilde{a}_{\lambda}^wa^w=a^w\tilde{a}_{\lambda}^w$, we have
\beqs
a^w(\tilde{a}_{\lambda}^w)^{-1}=a^wq_{\lambda}^w+q_{\lambda}^wa^w(\mathrm{Id}-\tilde{a}_{\lambda}^wq_{\lambda}^w)+ (\mathrm{Id}-q_{\lambda}^w\tilde{a}_{\lambda}^w)a^w(\tilde{a}_{\lambda}^w)^{-1}(\mathrm{Id}-\tilde{a}_{\lambda}^wq_{\lambda}^w)
\eeqs
as operators on $\SSS^*(\RR^d)$. For $\lambda>0$ denote $S_{\lambda}= (\mathrm{Id}-q_{\lambda}^w\tilde{a}_{\lambda}^w)a^w(\tilde{a}_{\lambda}^w)^{-1}(\mathrm{Id}-\tilde{a}_{\lambda}^wq_{\lambda}^w)$. Corollary \ref{corunipar} implies $S_{\lambda}\in\mathcal{L}(\SSS'^*(\RR^d),\SSS^*(\RR^d))$. Moreover
\beqs
S_{\lambda}=(\mathrm{Id}-q_{\lambda}^w\tilde{a}_{\lambda}^w) \overline{A}(\overline{A}+\lambda\mathrm{Id})^{-1}(\mathrm{Id}-\tilde{a}_{\lambda}^wq_{\lambda}^w).
\eeqs
Since $\overline{A}(\overline{A}+\lambda\mathrm{Id})^{-1}=\mathrm{Id}- \lambda(\overline{A}+\lambda\mathrm{Id})^{-1}$, $\{\overline{A}(\overline{A}+\lambda\mathrm{Id})^{-1}|\, \lambda>0\}$ and $\{\lambda(\overline{A}+\lambda\mathrm{Id})^{-1}|\, \lambda>0\}$ are equicontinuous subsets of $\mathcal{L}(L^2(\RR^d),L^2(\RR^d))$ ($\overline{A}$ is non-negative, by assumption). Hence, $\{S_{\lambda}|\, \lambda>0\}$ and $\{\lambda S_{\lambda}|\, \lambda>0\}$ are equicontinuous subsets of $\mathcal{L}(\SSS'^*(\RR^d),\SSS^*(\RR^d))$ and consequently, the same also holds for $\{(1+\lambda)S_{\lambda}|\, \lambda>0\}$. As $\{(1+\lambda)q_{\lambda}^w|\, \lambda>0\}$ is an equicontinuous subset of $\mathcal{L}(\SSS^*(\RR^d),\SSS^*(\RR^d))$, Corollary \ref{corunipar} yields the equicontinuity of the set $\{(1+\lambda)q_{\lambda}^wa^w(\mathrm{Id}-\tilde{a}_{\lambda}^wq_{\lambda}^w)|\, \lambda> 0\}$ in $\mathcal{L}(\SSS'^*(\RR^d),\SSS^*(\RR^d))$. Thus, $a^w(\tilde{a}_{\lambda}^w)^{-1}=a^wq_{\lambda}^w+\tilde{S}_{\lambda}$ as operators on $\SSS^*(\RR^d)$, where $\{(1+\lambda)\tilde{S}_{\lambda}|\, \lambda>0\}$ is an  equicontinuous subset of $\mathcal{L}(\SSS'^*(\RR^d),\SSS^*(\RR^d))$. Whence $a^w(\tilde{a}_{\lambda}^w)^{-1}$ can be continuously extended to an operator on $\SSS'^*(\RR^d)$ for each $\lambda>0$. As $a-a_0$ is  compactly supported, (\ref{ths7975}) and (\ref{ths7977}) together with Corollary \ref{corweylqu} imply the existence of $R_1\geq R$ such that $b^{(1)}_{\lambda}=R_1(a_0\#\sum_j q^{(\lambda)}_j)\in\Gamma_{A_p,\rho}^{*,\infty}(\RR^{2d})$, $\lambda>0$, $a^w(\tilde{a}_{\lambda}^w)^{-1}= (b^{(1)}_{\lambda})^w+S^{(1)}_{\lambda}$ and $\{(1+\lambda)S^{(1)}_{\lambda}|\,\lambda>0\}$ is equicontinuous in $\mathcal{L}(\SSS'^*(\RR^d),\SSS^*(\RR^d))$. Moreover, there exists $B_1>0$ such that for every $h>0$ there exists $C>0$ (resp. there exist $h,C>0$) such that
\beqs
\sup_{\substack{N\in\ZZ_+\\ \alpha\in\NN^{2d}}}\sup_{w\in Q_{B_1m_N}^c}\frac{(1+\lambda)\left|D^{\alpha}_w\left(b^{(1)}_{\lambda}(w)- (a_0\#\sum_j q^{(\lambda)}_j)_{<N}(w)\right)\right|\langle w\rangle^{\rho(|\alpha|+2N)}}{h^{|\alpha|+2N}A_{|\alpha|+2N}F_1(w)}\leq C,
\eeqs
for some continuous positive function $F_1$ on $\RR^{2d}$ with ultrapolynomial growth of class $*$ and $\{(1+\lambda)b^{(1)}_{\lambda}|\, \lambda>0\}$ is bounded subset of $\Gamma_{A_p,\rho}^{(M_p),\infty}(\RR^{2d};m^{(1)})$ for some $m^{(1)}>0$ (resp. bounded subset of $\Gamma_{A_p,\rho}^{\{M_p\},\infty}(\RR^{2d};h^{(1)})$ for some $h^{(1)}>0$). This implies that $\{(1+\lambda)(b^{(1)}_{\lambda})^w|\, \lambda>0\}$ is equicontinuous in $\mathcal{L}(\SSS'^*(\RR^d),\SSS'^*(\RR^d))$ and in $\mathcal{L}(\SSS^*(\RR^d),\SSS^*(\RR^d))$. Iterating the process, by Corollary \ref{corweylqu} we can conclude that for each $k\in\ZZ_+$ there exists $R_k>0$ such that $b^{(k)}_{\lambda}=R_k\big((a_0\#\sum_j q^{(\lambda)}_j)^{\# k}\big)\in\Gamma_{A_p,\rho}^{*,\infty}(\RR^{2d})$, $\lambda>0$, $(a^w(\tilde{a}_{\lambda}^w)^{-1})^k= (b^{(k)}_{\lambda})^w+S^{(k)}_{\lambda}$ and $\{(1+\lambda)^kS^{(k)}_{\lambda}|\,\lambda>0\}$ is equicontinuous subset of $\mathcal{L}(\SSS'^*(\RR^d),\SSS^*(\RR^d))$. Moreover, there exists $B_k>0$ such that for every $h>0$ there exists $C>0$ (resp. there exist $h,C>0$) such that
\beqs
\sup_{\substack{N\in\ZZ_+\\ \alpha\in\NN^{2d}}}\sup_{w\in Q_{B_km_N}^c}\frac{(1+\lambda)^k\left|D^{\alpha}_w\left(b^{(k)}_{\lambda}(w)- \left((a_0\#\sum_j q^{(\lambda)}_j)^{\# k}\right)_{<N}(w)\right)\right|\langle w\rangle^{\rho(|\alpha|+2N)}}{h^{|\alpha|+2N}A_{|\alpha|+2N}F_k(w)}\leq C,
\eeqs
for some continuous positive function $F_k$ with the ultrapolynomial growth of class $*$ and $\{(1+\lambda)^kb^{(k)}_{\lambda}|\, \lambda>0\}$ is a bounded subset of $\Gamma_{A_p,\rho}^{(M_p),\infty}(\RR^{2d};m^{(k)})$ for some $m^{(k)}>0$ (resp. a bounded subset of $\Gamma_{A_p,\rho}^{\{M_p\},\infty}(\RR^{2d};h^{(k)})$ for some $h^{(k)}>0$). Thus we proved the first part of the following result.

\begin{proposition}\label{knoforcomp}
For each $k\in\ZZ_+$, there exists $R_k>0$ such that
$$
b^{(k)}_{\lambda}=R_k\big((a_0\#\ssum q^{(\lambda)}_j)^{\# k}\big) \in\Gamma_{A_p,\rho}^{*,\infty}(\RR^{2d}),\,\,\, \lambda>0,
$$
\beq\label{equforequikno}
\left(a^w(\tilde{a}_{\lambda}^w)^{-1}\right)^k= (b^{(k)}_{\lambda})^w+S^{(k)}_{\lambda}
\eeq
and $\{(1+\lambda)^kS^{(k)}_{\lambda}|\,\lambda>0\}$ is an equicontinuous subset of $\mathcal{L}(\SSS'^*(\RR^d),\SSS^*(\RR^d))$. Moreover for every $h>0$ there exists $C>0$ (resp. there exist $h,C>0$) such that
\beq\label{estforthepower}
\left|D^{\alpha}_w b^{(k)}_{\lambda}(w)\right|\leq C\frac{h^{|\alpha|}A_{\alpha}|a_0(w)|^k}{\langle w\rangle^{\rho|\alpha|}|a_{\lambda}(w)|^k},\,\, w\in\RR^{2d},\, \alpha\in\NN^{2d},\, \lambda>0.
\eeq
Consequently, for each $k\in\ZZ_+$, $\{(1+\lambda)^kb^{(k)}_{\lambda}|\, \lambda>0\}$ is bounded in $\Gamma^{(M_p),\infty}_{A_p,\rho}(\RR^{2d};\tilde{m}'_k)$ for some $\tilde{m}'_k>0$ (resp. in $\Gamma_{A_p,\rho}^{\{M_p\},\infty}(\RR^{2d};\tilde{h}'_k)$ for some $\tilde{h}'_k>0$).
\end{proposition}

\begin{proof} It only remains to prove the estimate (\ref{estforthepower}). Keeping in mind (\ref{estfora_0}) and (\ref{estforthepar}), Lemma \ref{subsharpproduct} implies that for each fixed $k\in\ZZ_+$,
\beq\label{growthofsol1}
\left\{a_0^{\# k}\#(\ssum q^{(\lambda)}_j)^{\# k}\big|\, \lambda>0\right\}\precsim \{|a_0/a_{\lambda}|^k|\, \lambda>0\}
\eeq
in $FS_{A_p,\rho}^{*,\infty}(\RR^{2d};0)$. We apply Proposition \ref{subordinate} to (\ref{growthofsol1}). Hence, there exists $R'_k\geq R_k$ such that $b'^{(k)}_{\lambda}=R'_k\big(a_0^{\# k}\#(\sum_j q^{(\lambda)}_j)^{\# k}\big)\in \Gamma_{A_p,\rho}^{*,\infty}(\RR^{2d})$ and for $\{b'^{(k)}_{\lambda}|\, \lambda>0\}$ the uniform estimates (\ref{estforthepower}) holds on $Q^c_{B'_km_1}$ for some $B'_k>0$. There exists $j_0\in\ZZ_+$ such that $b'^{(k)}_{\lambda}(w)=\sum_{j=0}^{j_0}(1-\chi_{j,R'_k}(w))\big(a_0^{\# k}\#(\sum_j q^{(\lambda)}_j)^{\# k}\big)_j(w)$ on $Q_{B'_km_1}$ for all $\lambda\in\RR_+$. As (\ref{growthofsol1}) holds in $FS_{A_p,\rho}^{*,\infty}(\RR^{2d};0)$, we conclude that (\ref{estforthepower}) holds for $\{b'^{(k)}_{\lambda}|\, \lambda>0\}$ on $Q_{B'_km_1}$ too. Combining the estimates that we prove for $b^{(k)}_{\lambda}$ above, the fact that (\ref{estforthepower}) holds for $b'^{(k)}_{\lambda}$ and
\begin{multline*}
b^{(k)}_{\lambda}(w)-b'^{(k)}_{\lambda}(w)\\
=b^{(k)}_{\lambda}(w)-\left(a_0^{\# k}\#(\ssum q^{(\lambda)}_j)^{\# k}\right)_{<N}(w)+\left(a_0^{\# k}\#(\ssum q^{(\lambda)}_j)^{\# k}\right)_{<N}(w)-b'^{(k)}_{\lambda}(w)
\end{multline*}
on $Q^c_{B''_km_N}$, where $B''_k=\max\{B_k,B'_k\}$, we can conclude that the set
\beqs
\{(1+\lambda)^kb^{(k)}_{\lambda}-(1+\lambda)^kb'^{(k)}_{\lambda}|\, \lambda>0\}
\eeqs
satisfies the conditions of Proposition \ref{eqsse} (observe that, by (\ref{growthal}) in the Beurling case and (\ref{growthalr}) in the Roumieu case $(1+\lambda)^k|a_0(w)/a_{\lambda}(w)|^k\leq \tilde{F}_k(w)$, $\forall w\in\RR^{2d}$, $\forall \lambda\in\RR_+$, for some continuous positive function $F_k$ of ultrapolynomial growth of class $*$). Thus $\{(1+\lambda)^k(b^{(k)}_{\lambda})^w-(1+\lambda)^k(b'^{(k)}_{\lambda})^w|\, \lambda>0\}$ is equicontinuous in $\mathcal{L}(\SSS'^*(\RR^d),\SSS^*(\RR^d))$ and the conclusions of the proposition hold for $R'_k>0$ and $\{b'^{(k)}_{\lambda}|\, \lambda>0\}$.
\end{proof}

\begin{remark}\label{krs559933}
From the construction of $b^{(k)}_{\lambda}$, we can conclude that for each fixed $k\in\ZZ_+$, the estimate in Proposition \ref{subordinate} holds for $R_k$, $\{b^{(k)}_{\lambda}|\, \lambda>0\}$, $\left\{a_0^{\# k}\#(\ssum q^{(\lambda)}_j)^{\# k}\big|\, \lambda>0\right\}$ and $\{|a_0/a_{\lambda}|^k|\, \lambda>0\}$.
\end{remark}

As operators on $\SSS^*(\RR^d)$, we have
\beqs
(\tilde{a}_{\lambda}^w)^{-1}=\lambda^{-1}(\mathrm{Id}-a^w(\tilde{a}_{\lambda}^w)^{-1})= (\lambda^{-1}-\lambda^{-1}b^{(1)}_{\lambda})^w- \lambda^{-1}S^{(1)}_{\lambda},
\eeqs
for each $\lambda>0$. Thus, $(\tilde{a}_{\lambda}^w)^{-1}$ can be extended to a continuous operator on $\SSS'^*(\RR^d)$.

\subsubsection{Construction of $a^{\# z}_0$ (parts $(i)-(iv)$) and the estimates in $(v)$ of Theorem \ref{maint}}

For $\lambda,\lambda_0\geq0$, the following equalities hold true in $FS_{A_p,\rho}^{*,\infty}(\RR^{2d};0)$:
\beqs
\ssum q^{(\lambda)}_j&=&\ssum q^{(\lambda)}_j\#\mathbf{1}=\ssum q^{(\lambda)}_j\#a_{\lambda_0}\#\ssum q^{(\lambda_0)}_j\\
&=&\ssum q^{(\lambda)}_j\#(a_{\lambda}+(\lambda_0-\lambda)\mathbf{1})\#\ssum q^{(\lambda_0)}_j\\
&=&\ssum q^{(\lambda_0)}_j+(\lambda_0-\lambda)\ssum q^{(\lambda)}_j\#\ssum q^{(\lambda_0)}_j.
\eeqs
Thus,
\beq\label{diffunctionofl}
\ssum q^{(\lambda)}_j-\ssum q^{(\lambda_0)}_j=-(\lambda-\lambda_0)\ssum q^{(\lambda)}_j\#\ssum q^{(\lambda_0)}_j,\,\, \lambda,\lambda_0\geq 0.
\eeq
Incidentally, notice that this implies that $\ssum q^{(\lambda)}_j$ and $\ssum q^{(\lambda_0)}_j$ commute for every $\lambda,\lambda_0\geq0$. Now, by induction, one easily verifies that
\begin{multline}
\big(\ssum q_j^{(\lambda)}\big)^{\# k}-\big(\ssum q_j^{(\lambda_0)}\big)^{\# k}\\
=-(\lambda-\lambda_0)\sum_{s=1}^k\big(\ssum q_j^{(\lambda)}\big)^{\# (k+1-s)}\#\big(\ssum q_j^{(\lambda_0)}\big)^{\# s},\,\, \lambda,\lambda_0\geq0.\label{diffunctionofl11}
\end{multline}

\begin{lemma}\label{derofittt}
Let $\sum_j p_j\in FS_{A_p,\rho}^{*,\infty}(\RR^{2d};0)$. For each $n\in\NN$, $k\in\ZZ_+$,
\beqs
G_n^{(k)}:(\lambda,w)\mapsto \left(\ssum p_j\#\big(\ssum q^{(\lambda)}_j\big)^{\# k}\right)_n(w),\,\, \overline{\RR_+}\times\RR^{2d}\rightarrow \CC,
\eeqs
is in $\mathcal{C}^{\infty}(\overline{\RR_+}\times\RR^{2d})$ and $(\partial/\partial_{\lambda})G_n^{(k)}(\lambda,w)=-kG_n^{(k+1)}(\lambda,w)$, $(\lambda,w)\in\overline{\RR_+}\times\RR^{2d}$.
\end{lemma}

\begin{proof} Denote $\tilde{q}^{(k)}_n(\lambda,w)=\big((\sum_j q^{(\lambda)}_j)^{\# k}\big)_n(w)$. By definition,
\beqs
\tilde{q}^{(1)}_0(\lambda,w)=1/(a_0(w)+\lambda)\in\mathcal{C}^{\infty}(\overline{\RR_+}\times\RR^{2d})
\eeqs
(cf. (\ref{finqa}) and (\ref{sinqa})) and, using induction, we deduce that $\tilde{q}^{(1)}_j(\lambda,w)\in\mathcal{C}^{\infty}(\overline{\RR_+}\times\RR^{2d})$, $\forall j\in\ZZ_+$ (by the definition of $q_j^{(\lambda)}$). Thus, $G_n^{(1)}\in\mathcal{C}^{\infty}(\overline{\RR_+}\times\RR^{2d})$, for all $n\in\NN$, too. By induction on $k$, we also conclude that $G_n^{(k)}\in\mathcal{C}^{\infty}(\overline{\RR_+}\times\RR^{2d})$ for all $k\in\ZZ_+$, $n\in\NN$. Observe that (\ref{ths7975}) and (\ref{ths7977}) together with Theorem \ref{weylq} $i)$ imply that for each $k\in\ZZ_+$ there exists a positive $f_k\in\mathcal{C}(\RR^{2d})$ of ultrapolynomial growth of class $*$ such that $\{\sum_j q^{(\lambda)}_j|\, \lambda\geq 0\}^{\# k}\precsim f_k$ in $FS_{A_p,\rho}^{*,\infty}(\RR^{2d};0)$. This together with (\ref{diffunctionofl11}) and Proposition \ref{hyporingg} implies that for fixed $\sum_j c_j\in FS_{A_p,\rho}^{*,\infty}(\RR^{2d};0)$ and $k\in\ZZ_+$, the function $\lambda\mapsto \sum_j c_j\#\big(\sum_j q^{(\lambda)}_j)^{\# k}$, $\overline{\RR_+}\rightarrow FS_{A_p,\rho}^{*,\infty}(\RR^{2d};0)$, is continuous. Now, (\ref{diffunctionofl11}) yields $(\partial/\partial_{\lambda})G_n^{(k)}(\lambda,w)=-kG_n^{(k+1)}(\lambda,w)$.
\end{proof}

For brevity in notation, for each $n\in\NN$, $k\in\ZZ_+$, we denote by $g^{(k)}_n$ the function $(\lambda,w)\mapsto \big(a_0^{\# k}\#(\sum_j q^{(\lambda)}_j)^{\# k}\big)_n(w)$, $\RR_+\times\RR^{2d}\rightarrow \CC$. Lemma \ref{derofittt} implies that for each $n\in\NN$ and $k\in\ZZ_+$, $g_n^{(k)}\in\mathcal{C}^{\infty}(\RR_+\times\RR^{2d})$. For $z\in\CC_+$ and $k>\mathrm{Re}\, z$, put
\beqs
p^{(k)}_{z,n}(w)=\gamma_k(z)\int_0^{\infty} \lambda^{z-1}g^{(k)}_n(\lambda,w)d\lambda,\,\, \forall w\in\RR^{2d}.
\eeqs
Theorem \ref{weylq} $i)$ and (\ref{ths7975}) and (\ref{ths7977}) imply
\beq\label{rhv117755}
\big\{(1+\lambda)^ka_0^{\# k}\#(\ssum q^{(\lambda)}_j)^{\# k}\big|\, \lambda\geq0\big\}\precsim f_k\,\, \mbox{in}\,\, FS_{A_p,\rho}^{*,\infty}(\RR^{2d};0),
\eeq
for some positive $f_k\in\mathcal{C}(\RR^{2d})$ with ultrapolynomial growth of class $*$. This shows that for each $z\in\CC_+$, $k>\mathrm{Re}\, z$ and $n\in\NN$, $p^{(k)}_{z,n}\in\mathcal{C}^{\infty}(\RR^{2d})$ and
\beqs
D^{\alpha}_w p^{(k)}_{z,n}(w)=\gamma_k(z)\int_0^{\infty} \lambda^{z-1}D^{\alpha}_w g^{(k)}_n(\lambda,w)d\lambda,\,\, \forall w\in\RR^{2d},\,\, \alpha\in\NN^{2d}
\eeqs
and $\sum_j p^{(k)}_{z,j}(w)\in FS_{A_p,\rho}^{*,\infty}(\RR^{2d};0)$. Integration by parts together with Lemma \ref{derofittt} yields
\beqs
\frac{p^{(k)}_{z,n}(w)}{\gamma_k(z)}=\frac{\lambda^z}{z} g^{(k)}_n(\lambda,w)\big|_{\lambda=0}^{\lambda=\infty}+\frac{k}{z}\int_0^{\infty}\lambda^z \left(a_0^{\# k}\#\big(\ssum q^{(\lambda)}_j\big)^{\# (k+1)}\right)_n(w)d\lambda.
\eeqs
The first term is obviously equal to $0$. Observe that
\beqs
a_0^{\# (k+1)}\#\big(\ssum q^{(\lambda)}_j\big)^{\# (k+1)}&=&a_0^{\# k}\#(a_{\lambda}-\lambda)\#\big(\ssum q^{(\lambda)}_j\big)^{\# (k+1)}\\
&=&a_0^{\# k}\#\big(\ssum q^{(\lambda)}_j\big)^{\#k} -\lambda a_0^{\# k}\#\big(\ssum q^{(\lambda)}_j\big)^{\# (k+1)}.
\eeqs
Thus, the second term is equal to
\beqs
\frac{k}{z}\int_0^{\infty}\lambda^{z-1}g_n^{(k)}(\lambda,w)d\lambda- \frac{k}{z}\int_0^{\infty}\lambda^{z-1}g_n^{(k+1)}(\lambda,w)d\lambda=\frac{kp^{(k)}_{z,n}(w)}{z\gamma_k(z)}- \frac{kp^{(k+1)}_{z,n}(w)}{z\gamma_{k+1}(z)}.
\eeqs
Since $\gamma_{k+1}(z)=\gamma_k(z)k/(k-z)$, we conclude $p_{z,n}^{(k+1)}=p_{z,n}^{(k)}$. Thus, for $z\in\CC_+$, $j\in\NN$, we can define $p_{z,j}=p_{z,j}^{(k)}$ for arbitrary $k>\mathrm{Re}\, z$ and $\sum_j p_{z,j}\in FS_{A_p,\rho}^{*,\infty}(\RR^{2d};0)$.\\
\indent For $\{z_n\}_{n\in\ZZ_+}\subseteq\CC_+$ which converges to $z\in\CC_+$, (\ref{rhv117755}) implies that for $k\in\ZZ_+$ such that $k>[\mathrm{Re}\, z]$ and $k>[\mathrm{Re}\, z_n]$, $n\in\ZZ_+$, the following estimate holds: for every $h>0$ there exists $C>0$ (resp. there exist $h,C>0$) such that
\beqs
\left|D^{\alpha}p_{z,j}(w)-D^{\alpha}p_{z_n,j}(w)\right|
\leq  \frac{Ch^{|\alpha|+2j}A_{|\alpha|+2j}f_k(w)}{\langle w\rangle^{\rho(|\alpha|+2j)}}\int_0^{\infty}\frac{\left|\gamma_k(z)\lambda^{z-1}- \gamma_k(z_n)\lambda^{z_n-1}\right|}{(1+\lambda)^k}d\lambda.
\eeqs
Hence, Lebesgue's theorem on dominated convergence implies that $\sum_j p_{z_n,j}\rightarrow\sum_j p_{z,j}$ in $FS_{A_p,\rho}^{*,\infty}(\RR^{2d};0)$. We proved the following lemma.

\begin{lemma}\label{thj994455}
The mapping $z\mapsto \sum_j p_{z,j}$, $\CC_+\rightarrow FS_{A_p,\rho}^{*,\infty}(\RR^{2d};0)$, is continuous.
\end{lemma}

Before we proceed, we need a technical result. For $f\in\mathcal{C}^1(\overline{\RR_+})$, denote
\beqs
\tilde{f}(\lambda,\mu)=\left\{\begin{array}{ll}
\frac{f(\lambda)-f(\mu)}{\lambda-\mu},\,\, \lambda\neq \mu,\\
f'(\lambda),\,\, \lambda=\mu,
\end{array}\right.
\eeqs
Clearly, $\tilde{f}\in\mathcal{C}(\overline{\RR_+^2})$.

\begin{lemma}\label{stk774433}
Let $z,\zeta\in\CC_+$ be such that $0<\mathrm{Re}\, z,\mathrm{Re}\, \zeta<1$. Let $f\in\mathcal{C}^1(\overline{\RR_+})$ be such that $f\in L^1(\RR_+,|\lambda^{z-1}|d\lambda)\cap L^1(\RR_+,|\lambda^{\zeta-1}|d\lambda)$, $f'\in L^1(\RR_+,|\lambda^{z+\zeta-1}|d\lambda)$ and $\tilde{f}\in L^1(\RR^2_+,|\lambda^{z-1}\mu^{\zeta-1}|d\lambda d\mu)$. Then
\beq\label{kvh114477}
\gamma_1(z)\gamma_1(\zeta)\int_{\RR^2_+}\lambda^{z-1}\mu^{\zeta-1}\tilde{f}(\lambda,\mu)d\lambda d\mu=\gamma_2(z+\zeta)\int_0^{\infty}\lambda^{z+\zeta-1}f'(\lambda)d\lambda.
\eeq
\end{lemma}

\begin{proof} For $f$ as in the statement in the lemma, denote the left hand side of (\ref{kvh114477}) by $I_1(f)$ and the right hand side by $I_2(f)$. Notice that for such $f$, the mapping $f\mapsto I_1(f)-I_2(f)$ is linear. The equality (\ref{kvh114477}) holds for $f(\lambda)=(t+\lambda)^{-1}$, $t>0$ fixed, by (\ref{inteq}). We prove (\ref{kvh114477}) for $f(\lambda)=(t+\lambda)^{-n}$, $t>0$, $n\in\NN$, $n\geq 2$. First, notice that
\beq\label{sht7791}
\prod_{j=2}^n(j-z-\zeta)=\sum_{j=1}^n {{n-1}\choose{j-1}}\prod_{l=1}^{n-j}(l-z)\cdot \prod_{k=1}^{j-1}(k-\zeta),\,\, z,\zeta\in\CC.
\eeq
One verifies this equality by direct inspection for $n=2$ and then proves the general case by induction. Since $\tilde{f}(\lambda,\mu)=-\sum_{j=1}^n (t+\lambda)^{-(n+1-j)}(t+\mu)^{-j}$, by employing (\ref{inteq}), one easily verifies that
\beqs
I_1(f)=-t^{z+\zeta-n-1}\sum_{j=1}^n\frac{\gamma_1(z)\gamma_1(\zeta)}{\gamma_{n+1-j}(z)\gamma_j(\zeta)}.
\eeqs
Since
\beq\label{gmaf157}
\Gamma(v+j)=\Gamma(v)v(v+1)\cdot\ldots\cdot(v+j-1),\,\, v\in\CC_+,\, j\in\ZZ_+,
\eeq
we infer
\beqs
I_1(f)=-t^{z+\zeta-n-1}\sum_{j=1}^n\frac{(1-z)\cdot\ldots\cdot(n-j-z)\cdot (1-\zeta)\cdot\ldots \cdot(j-1-\zeta)}{(n-j)!(j-1)!}.
\eeqs
Similarly,
\beqs
I_2(f)=-nt^{z+\zeta-n-1}\gamma_2(z+\zeta)/\gamma_{n+1}(z+\zeta)=-\frac{t^{z+\zeta-n-1}}{(n-1)!}\prod_{j=2}^n(j-z-\zeta).
\eeqs
Now, (\ref{sht7791}) proves (\ref{kvh114477}) when $f(\lambda)=(t+\lambda)^{-n}$. Hence, (\ref{kvh114477}) holds true for all elements of the algebra $\mathcal{A}_t$ generated by the function $\lambda\mapsto (t+\lambda)^{-1}$, $t>0$ fixed. Next, we prove (\ref{kvh114477}) for arbitrary compactly supported $f\in\mathcal{C}^1(\overline{\RR_+})$. Obviously, it is enough to prove it when $f$ is real valued. Fix $t>0$. Since $\lambda\mapsto (t+\lambda)^2f'(\lambda)$ is real valued and in $\mathcal{C}_0(\overline{\RR_+})$ (the $(B)$-space of functions which vanish at infinity), the Stone-Weierstrass theorem implies that for every $\varepsilon>0$ there exists $g_{\varepsilon}\in\mathcal{A}_t$ such that
\beq\label{sft9944}
-\frac{\varepsilon}{(t+\lambda)^2}\leq f'(\lambda)-\frac{g_{\varepsilon}(\lambda)}{(t+\lambda)^2}\leq \frac{\varepsilon}{(t+\lambda)^2},\,\, \forall \lambda\geq 0.
\eeq
Denote $f_{\varepsilon}(\lambda)=-\int_{\lambda}^{\infty}g_{\varepsilon}(s)(t+s)^{-2}ds$. Clearly, $f_{\varepsilon},f'_{\varepsilon}\in\mathcal{A}_t$. We conclude that
\beqs
|f'(\lambda)-f'_{\varepsilon}(\lambda)|\leq \varepsilon/(t+\lambda)^2,\,\,\, \lambda\geq 0
\eeqs
and consequently $I_2(f_{\varepsilon})\rightarrow I_2(f)$, as $\varepsilon\rightarrow 0^+$. Write
\beqs
\lambda-\mu=(t+\lambda)(t+\mu)((t+\mu)^{-1}-(t+\lambda)^{-1})
\eeqs
and apply the Cauchy mean value theorem to obtain
\beqs
\tilde{f}_{\varepsilon}(\lambda,\mu)-\tilde{f}(\lambda,\mu)=-(t+\xi)^2(f'(\xi)-f'_{\varepsilon}(\xi))(t+\lambda)^{-1}(t+\mu)^{-1}
\eeqs
for some $\xi$ between $\lambda$ and $\mu$. Thus, (\ref{sft9944}) implies
\beqs
|\tilde{f}_{\varepsilon}(\lambda,\mu)-\tilde{f}(\lambda,\mu)|\leq \varepsilon(t+\lambda)^{-1}(t+\mu)^{-1},\,\, \lambda,\mu\in\overline{\RR_+}.
\eeqs
This readily implies $\tilde{f}\in L^1(\RR^2_+,|\lambda^{z-1}\mu^{\zeta-1}|d\lambda d\mu)$ and $I_1(f_{\varepsilon})\rightarrow I_1(f)$, as $\varepsilon\rightarrow 0^+$, which completes the proof of (\ref{kvh114477}) for $f\in\mathcal{C}^1(\overline{\RR_+})$ with compact support. For general $f$ as in the assumption in the lemma take $\varphi\in\mathcal{D}(\RR)$ such that $0\leq \varphi\leq 1$, $\varphi(\lambda)=1$ when $\lambda\in[-1,1]$ and $\varphi(\lambda)=0$ when $\lambda\in\RR\backslash[-2,2]$. Denote $\varphi_n(\lambda)=\varphi(\lambda/n)$, $n\in\ZZ_+$. Then (\ref{kvh114477}) holds true for $f_n=f\varphi_n$, $n\in\ZZ_+$. Denote $C_1=\sup\{(1+\mu)|\varphi'(\mu)||\, \mu\in\RR\}$ and observe that $|f'_n(\lambda)|\leq |f'(\lambda)|+C_1|f(\lambda)|(1+\lambda)^{-1}$. The dominated convergence theorem implies $I_2(f_n)\rightarrow I_2(f)$, as $n\rightarrow\infty$. Notice that
\beqs
\tilde{f}_n(\lambda,\mu)&=&\tilde{f}(\lambda,\mu)\varphi_n(\mu)+\tilde{\varphi}_n(\lambda,\mu)f(\lambda),\,\, \lambda,\mu\in\overline{\RR_+},\\
\tilde{f}_n(\lambda,\mu)&=&\tilde{f}(\lambda,\mu)\varphi_n(\lambda)+\tilde{\varphi}_n(\lambda,\mu)f(\mu),\,\, \lambda,\mu\in\overline{\RR_+}
\eeqs
When $\mu\leq \lambda$, the first equality implies $|\tilde{f}_n(\lambda,\mu)|\leq |\tilde{f}(\lambda,\mu)|+C_1|f(\lambda)|(1+\mu)^{-1}$. When $\lambda\leq \mu$, one uses the second equality to deduce $|\tilde{f}_n(\lambda,\mu)|\leq |\tilde{f}(\lambda,\mu)|+C_1|f(\mu)|(1+\lambda)^{-1}$. Hence, the dominated convergence theorem implies $I_1(f_n)\rightarrow I_1(f)$ and the proof of the lemma is complete.
\end{proof}

\begin{proposition}
The following properties hold for $\sum_j p_{z,j}$.
\begin{itemize}
\item[$(i)$] $\sum_j p_{z,j}\#\sum_j p_{\zeta,j}=\sum_j p_{\zeta,j}\#\sum_j p_{z,j}$, $\forall z,\zeta\in\CC_+$.
\item[$(ii)$] When $z=k\in\ZZ_+$, $\sum_j p_{k,j}=a_0^{\# k}$. In particular, $\sum_j p_{1,j}$ is just $a_0$.
\item[$(iii)$] $\sum_j p_{z,j}\#\sum_j p_{\zeta,j}=\sum_j p_{z+\zeta,j}$, $\forall z,\zeta\in\CC_+$.
\end{itemize}
\end{proposition}

\begin{proof} Fix $z,\zeta\in\CC_+$ and denote $\sum_j c_j=\sum_j p_{z,j}\#\sum_j p_{\zeta,j}\in FS_{A_p,\rho}^{*,\infty}(\RR^{2d};0)$. Then, for $k>\max\{[\mathrm{Re}\,z],[\mathrm{Re}\,\zeta]\}$, $k\in\ZZ_+$, we have
\beq
c_j(w)&=&\sum_{s+r+l=j}\sum_{|\alpha+\beta|=l}\frac{(-1)^{|\beta|}}{\alpha!\beta!2^l}\partial^{\alpha}_{\xi}D^{\beta}_x p_{z,s}(w)\partial^{\beta}_{\xi} D^{\alpha}_x p_{\zeta,r}(w)\nonumber\\
&=&\gamma_k(z)\gamma_k(\zeta) \sum_{s+r+l=j}\sum_{|\alpha+\beta|=l}\frac{(-1)^{|\beta|}}{\alpha!\beta!2^l}\nonumber\\
&{}&\cdot\int_{\RR^2_+}\lambda^{z-1}\mu^{\zeta-1}\partial^{\alpha}_{\xi}D^{\beta}_x \Big(\big(a_0\#\ssum q^{(\lambda)}_j\big)^{\# k}\Big)_s(w)\partial^{\beta}_{\xi} D^{\alpha}_x\Big(\big(a_0\#\ssum q^{(\mu)}_j\big)^{\# k}\Big)_r(w)d\lambda d\mu\nonumber\\
&=&\gamma_k(z)\gamma_k(\zeta) \int_{\RR^2_+}\lambda^{z-1}\mu^{\zeta-1} \Big(a_0^{\# 2k}\#\big(\ssum q^{(\lambda)}_j\#\ssum q^{(\mu)}_j\big)^{\# k}\Big)_j(w)d\lambda d\mu\label{hts77551199}
\eeq
and the integrals over $\RR^2_+$ are absolutely convergent uniformly for $w$ in compact subsets of $\RR^{2d}$ because of (\ref{ths7975}), (\ref{ths7977}) and Theorem \ref{weylq} $i)$. This readily implies that $\sum_j p_{z,j}$ and $\sum_j p_{\zeta,j}$ commute ($\ssum q^{(\lambda)}_j$ and $\ssum q^{(\mu)}_j$ commute). If $z=k\in\ZZ_+$, then (notice that $\gamma_{k+1}(k)=k$)
\beqs
p_{k,j}(w)=k\int_0^{\infty}\lambda^{k-1}\Big(a_0^{\# (k+1)}\#\big(\ssum q_j^{(\lambda)}\big)^{\# (k+1)}\Big)_j(w)d\lambda.
\eeqs
Repeated integration by parts together with Lemma \ref{derofittt} and the fact
\beqs
\lambda^{k-r}\Big(a_0^{\# (k+1)}\#\big(\ssum q_j^{(\lambda)}\big)^{\# (k+1-r)}\Big)_j(w)\Big|_{\lambda=0}^{\lambda=\infty}=0,\,\, r=1,\ldots,k-1
\eeqs
(cf. (\ref{ths7975}) and (\ref{ths7977})), imply
\beqs
p_{k,j}(w)&=&\int_0^{\infty}\Big(a_0^{\# (k+1)}\#\big(\ssum q_j^{(\lambda)}\big)^{\# 2}\Big)_j(w)d\lambda\\
&=&-\Big(a_0^{\# (k+1)}\#\ssum q_j^{(\lambda)}\Big)_j(w)\Big|_{\lambda=0}^{\lambda=\infty}=\big(a_0^{\# k}\big)_j(w).
\eeqs
Hence, $\sum_j p_{k,j}=a_0^{\# k}$.\\
\indent Let $z=n\in\ZZ_+$ and $\zeta\in\CC_+$. Then, denoting $\sum_j c_j=\sum_j p_{n,j}\#\sum_j p_{\zeta,j}$ and taking $k>[\mathrm{Re}\, \zeta]$, $k\in\ZZ_+$, one obtains
\beqs
c_j(w)&=&\gamma_k(\zeta)\sum_{s+r+l=j}\sum_{|\alpha+\beta|=l}\frac{(-1)^{|\beta|}}{\alpha!\beta!2^l}\\
&{}&\cdot\int_0^{\infty}\lambda^{\zeta-1}\partial^{\alpha}_{\xi}D^{\beta}_x \big(a_0^{\# n}\big)_s(w)\partial^{\beta}_{\xi} D^{\alpha}_x\Big(\big(a_0\#\ssum q^{(\lambda)}_j\big)^{\# k}\Big)_r(w)d\lambda\\
&=&\gamma_k(\zeta)\int_0^{\infty}\lambda^{\zeta-1}\Big(a_0^{\#(k+n)}\#\big(\ssum q^{(\lambda)}_j\big)^{\# k}\Big)_j(w)d\lambda.
\eeqs
Repeated integration by parts and Lemma \ref{derofittt} yield
\beqs
c_j(w)=\frac{\gamma_k(\zeta)k(k+1)\cdot\ldots\cdot(k+n-1)}{\zeta(\zeta+1)\cdot\ldots\cdot(\zeta+n-1)} \int_0^{\infty}\lambda^{\zeta+n-1}\Big(a_0^{\#(k+n)}\#\big(\ssum q^{(\lambda)}_j\big)^{\# (k+n)}\Big)_j(w)d\lambda,
\eeqs
which, in view of (\ref{gmaf157}), proves $\sum_j p_{n,j}\#\sum_j p_{\zeta,j}=\sum_j p_{n+\zeta,j}$. Assume now that $z,\zeta\in\CC_+$ are such that $0<\mathrm{Re}\,z,\mathrm{Re}\,\zeta<1$. Denoting $\sum_j c_j=\sum_j p_{z,j}\#\sum_j p_{\zeta,j}$, we have (\ref{hts77551199}) with $k=1$. Fix $w\in\RR^{2d}$ and put $f_{j,w}(\lambda)=\Big(a_0^{\# 2}\#\ssum q^{(\lambda)}_j\Big)_j(w)$. In view of (\ref{ths7975}), (\ref{ths7977}), Theorem \ref{weylq} $i)$ and Lemma \ref{derofittt}, $f_{j,w}$ satisfies the assumptions of Lemma \ref{stk774433}; notice that $\tilde{f}_{j,w}(\lambda,\mu)=-\Big(a_0^{\# 2}\#\ssum q^{(\lambda)}_j\#\ssum q^{(\mu)}_j\Big)_j(w)$ (because of (\ref{diffunctionofl})). Hence, employing Lemma \ref{stk774433} we conclude $c_j=p_{z+\zeta,j}$. If $z,\zeta\in\CC_+\backslash(\ZZ_+ +i\RR)$ are such that $\mathrm{Re}\,z,\mathrm{Re}\,\zeta>1$, by denoting $z'=z-[\mathrm{Re}\,z]$, $\zeta'=\zeta-[\mathrm{Re}\, \zeta]$ and using the facts we proved above, we have
\beqs
\ssum p_{z,j}\#\ssum p_{\zeta,j}&=&\ssum p_{z',j}\#\ssum p_{[\mathrm{Re}\,z],j}\#\ssum p_{\zeta',j}\#\ssum p_{[\mathrm{Re}\, \zeta],j}\\
&=&\ssum p_{[\mathrm{Re}\,z]+[\mathrm{Re}\, \zeta],j}\#\ssum p_{z'+\zeta',j}=\ssum p_{z+\zeta,j}.
\eeqs
In a similar fashion, one proves this equality when $\mathrm{Re}\,z<1$ or $\mathrm{Re}\,\zeta<1$. The case when one or both of $z$ and $\zeta$ belong to $\ZZ_+ +i(\RR\backslash\{0\})$ follows from Proposition \ref{hyporingg} and Lemma \ref{thj994455}.
\end{proof}

Because of this proposition, from now on we write $a_0^{\# z}$ instead of $\sum_j p_{z,j}$, $z\in\CC_+$. This proposition and Lemma \ref{thj994455} prove $(i)$, $(ii)$ and $(iii)$ of Theorem \ref{maint}.\\
\indent For arbitrary $R'>0$, $R'(\sum_j g^{(k)}_j(\lambda,w))\in\mathcal{C}^{\infty}(\RR_+\times\RR^{2d}_w)$ since the sum is locally finite. Clearly,
\beq\label{sumforaz}
R'(a_0^{\# z})(w)=\gamma_k(z)\int_0^{\infty} \lambda^{z-1} R'\left(\ssum g^{(k)}_j(\lambda,w)\right)d\lambda,\,\, k>\mathrm{Re}\, z.
\eeq
Keeping in mind (\ref{estfora_0}) and (\ref{estforthepar}), Lemma \ref{subsharpproduct} implies that for each fixed $k\in\ZZ_+$,
\beq\label{growthofsol}
\left\{\ssum g^{(k)}_j(\lambda,\cdot)=a_0^{\# k}\#(\ssum q^{(\lambda)}_j)^{\# k}\big|\, \lambda\geq0\right\}\precsim \{|a_0/a_{\lambda}|^k|\, \lambda\geq0\}
\eeq
in $FS_{A_p,\rho}^{*,\infty}(\RR^{2d};0)$. Because of (\ref{krt558811}) we have
\beqs
\left|D^{\alpha}_w (a_0^{\# z})_j(w)\right|\leq C\frac{h^{|\alpha|+2j}A_{|\alpha|+2j}|\gamma_k(z)|}{\langle w\rangle^{\rho(|\alpha|+2j)}}\int_0^{\infty}\frac{\lambda^{\mathrm{Re}\,z-1}|a_0(w)|^k}{\left(|a_0(w)|+\lambda\right)^k}d\lambda
\eeqs
for each $h>0$ and a corresponding $C>0$ in the Beurling case and for some $h,C>0$ in the Roumieu case, respectively. Because of (\ref{inteq}) this yields (\ref{tkr554411}), since the constant $C$ depends on $k$ (i.e. on $\mathrm{Re}\, z$). Thus, $C$ can be chosen to be the same when $z$ varies in a fixed vertical strip $\CC_{+,\tilde{c}}=\{\zeta\in\CC_+|\, \mathrm{Re}\,\zeta\leq \tilde{c}\}$. Observe that $g^{(k)}_0(\lambda,w)=\left(a_0(w)/(a_0(w)+\lambda)\right)^k$. Thus, by (\ref{inteq}), we conclude that $(a_0^{\# z})_0(w)=(a_0(w))^z$, $w\in\RR^{2d}$, which proves $(iv)$ of Theorem \ref{maint}. Incidentally, notice that (\ref{tkr554411}) and (\ref{thk779911}) imply that $w\mapsto (a_0(w))^z$ is hypoelliptic.\\
\indent For the strip $\CC_{+,\tilde{c}}$ put $k_0=[\tilde{c}]+1$. Take $R_{k_0}$ as in Proposition \ref{knoforcomp}, i.e.
\beqs
R_{k_0}(\ssum g^{(k_0)}_j(\lambda,\cdot))=b^{(k_0)}_{\lambda}\in \Gamma_{A_p,\rho}^{*,\infty}(\RR^{2d}).
\eeqs
Specifying $R'=R_{k_0}$ in (\ref{sumforaz}) and denoting $a^{\uwidehat{z}}:=R_{k_0}(a^{\# z})$, (\ref{estforthepower}), (\ref{krt558811}) and (\ref{inteq}) yield the second estimate in $(v)$ of Theorem \ref{maint}. Furthermore, Remark \ref{krs559933}, (\ref{krt558811}) and (\ref{inteq}) prove the first estimate in $(v)$ of Theorem \ref{maint}. For fixed $z\in\CC_{+,\tilde{c}}$, particularising the first estimate in $(v)$ of Theorem \ref{maint} for $N=1$, $\alpha=0$ and employing (\ref{thk779911}) we conclude that $|a^{\uwidehat{z}}(w)-(a_0(w))^z|\leq \tilde{C}|(a_0(w))^z|/\langle w\rangle^{2\rho}$ for $w$ outside of a compact neighbourhood of the origin. Thus, the second estimate in $(v)$ of Theorem \ref{maint} together with (\ref{thk779911}) yields the hypoellipticity of $a^{\uwidehat{z}}$.

\subsubsection{Part $(vi)$ of Theorem \ref{maint}}

Since $\overline{A}(\overline{A}+\lambda\mathrm{Id})^{-1}= \mathrm{Id}-\lambda(\overline{A}+\lambda\mathrm{Id})^{-1}$ as operators on $L^2(\RR^d)$, $\{\overline{A}(\overline{A}+\lambda\mathrm{Id})^{-1}|\, \lambda>0\}$ is an equicontinuous subset of $\mathcal{L}(L^2(\RR^d),L^2(\RR^d))$ ($\overline{A}$ is non-negative). Employing the resolvent identity and using induction one easily verifies that for each $v\in L^2(\RR^d)$, $k\in\ZZ_+$, the function $\lambda\mapsto \left(\overline{A}(\overline{A}+\lambda\mathrm{Id})^{-1}\right)^k v$, $\RR_+\rightarrow L^2(\RR^d)$, is continuous.\\
\indent Throughout this subsection $\CC_{+,\tilde{c}}$ is a fixed vertical strip and $a^{\uwidehat{z}}$, $z\in\CC_{+,\tilde{c}}$, and $b^{(k_0)}_{\lambda}$, $\lambda>0$, $k_0=[\tilde{c}]+1$, are defined as above. In Lemma \ref{derofittt} we proved that for arbitrary $\sum_j p_j\in FS_{A_p,\rho}^{*,\infty}(\RR^{2d};0)$, $n\in\ZZ_+$, the function $\lambda\mapsto \sum_j p_j\# (\sum_j q^{(\lambda)}_j)^{\# n}$, $\overline{\RR_+}\rightarrow FS_{A_p,\rho}^{*,\infty}(\RR^{2d};0)$, is continuous. This implies that for each $n\in\NN$, the mapping $\lambda\mapsto\big((a_0\# \sum_j q^{(\lambda)}_j)^{\# k_0}\big)_n$, $\overline{\RR_+}\rightarrow\EE^{(A_p)}(\RR^{2d})$ (resp. $\overline{\RR_+}\rightarrow\EE^{\{A_p\}}(\RR^{2d})$) is continuous. Hence, $\lambda\mapsto b^{(k_0)}_{\lambda}$ in $\RR_+\rightarrow\EE^{(A_p)}(\RR^{2d})$ (resp. $\RR_+\rightarrow\EE^{\{A_p\}}(\RR^{2d})$) is continuous; $b^{(k_0)}_{\lambda}$ is a locally finite sum of the form $\sum_j (1-\chi_{j,R_{k_0}})\big((a_0\# \sum_j q^{(\lambda)}_j)^{\# k_0}\big)_j$. Moreover, (\ref{estforthepower}) and (\ref{sinqa}) imply that $\{b^{(k_0)}_{\lambda}|\, \lambda>0\}\precsim 1$, i.e. for every $h>0$ there exists $C>0$ (resp. there exist $h,C>0$) such that
\beq\label{eeq1}
\left|D^{\alpha}_w b^{(k_0)}_{\lambda}(w)\right|\leq Ch^{|\alpha|}A_{\alpha}\langle w\rangle^{-\rho|\alpha|},\,\, w\in\RR^{2d},\, \alpha\in\NN^{2d},\, \lambda>0.
\eeq
In the Beurling case, the estimate (\ref{eeq1}) and $b^{(k_0)}_{\lambda}\rightarrow b^{(k_0)}_{\lambda_0}$ in $\EE^{(A_p)}(\RR^{2d})$ as $\lambda\rightarrow\lambda_0$ readily yield that $b^{(k_0)}_{\lambda}\rightarrow b^{(k_0)}_{\lambda_0}$ in $\Gamma^{(M_p),\infty}_{A_p,\rho}(\RR^{2d};m')$, for some $m'>0$. Thus, the mapping $\lambda\mapsto b^{(k_0)}_{\lambda}$, $\RR_+\rightarrow \Gamma^{(M_p),\infty}_{A_p,\rho}(\RR^{2d};m')$, is continuous, hence, also continuous if we regard it as a mapping from $\RR_+$ to $\Gamma_{A_p,\rho}^{(M_p),\infty}(\RR^{2d})$. In order to prove the convergence in the Roumieu case, let $h,C>0$ be the constants for which (\ref{eeq1}) holds. Denote $h_1=2h$. We prove that $b^{(k_0)}_{\lambda}\rightarrow b^{(k_0)}_{\lambda_0}$ in $\Gamma_{A_p,\rho}^{\{M_p\},\infty}(\RR^{2d};h_1H)$ as $\lambda\rightarrow\lambda_0$. Let $m'>0$ be arbitrary but fixed. Fix $\varepsilon>0$. Since $e^{-M(m'|\xi|)}e^{-M(m'|x|)}\rightarrow 0$ as $|(x,\xi)|\rightarrow \infty$, there exists $K\subset\subset \RR^{2d}$ such that
\beqs
\frac{\left|D^{\alpha}_{(x,\xi)} b^{(k_0)}_{\lambda}(x,\xi)\right|\langle w\rangle^{\rho|\alpha|}} {h^{|\alpha|}A_{\alpha} e^{M(m'|\xi|)}e^{M(m'|x|)}}\leq \frac{\varepsilon}{2},\,\, (x,\xi)\in\RR^{2d}\backslash K,\, \alpha\in\NN^{2d},\, \lambda>0.
\eeqs
Pick $n_0\in\ZZ_+$ such that $C/2^{n_0}\leq \varepsilon/2$. Thus, if $(x,\xi)\in\RR^{2d}\backslash K$ or $|\alpha|\geq n_0$ we have
\beqs
\frac{\left|D^{\alpha}_{(x,\xi)} \left(b^{(k_0)}_{\lambda}(x,\xi)-b^{(k_0)}_{\lambda_0}(x,\xi)\right)\right|\langle w\rangle^{\rho|\alpha|}} {h_1^{|\alpha|}A_{\alpha} e^{M(m'|\xi|)}e^{M(m'|x|)}}\leq \varepsilon,\,\, \forall\lambda>0.
\eeqs
Since $b^{(k_0)}_{\lambda}\rightarrow b^{(k_0)}_{\lambda_0}$ in $\EE^{\{A_p\}}(\RR^{2d})$ the convergence also holds in $\mathcal{C}^{\infty}(\RR^{2d})$. Thus, when $(x,\xi)\in K$ and $|\alpha|\leq n_0$, there exists $\delta>0$ such that for $|\lambda-\lambda_0|\leq \delta$ we have
\beqs
\frac{\left|D^{\alpha}_{(x,\xi)} \left(b^{(k_0)}_{\lambda}(x,\xi)-b^{(k_0)}_{\lambda_0}(x,\xi)\right)\right|\langle w\rangle^{\rho|\alpha|}} {h_1^{|\alpha|}A_{\alpha} e^{M(m'|\xi|)}e^{M(m'|x|)}}\leq \varepsilon.
\eeqs
We conclude that $b^{(k_0)}_{\lambda}\rightarrow b^{(k_0)}_{\lambda_0}$ in $\Gamma_{A_p,\rho}^{\{M_p\},\infty}(\RR^{2d};h_1H)$ (the extra $H$ appears since for $\alpha',\alpha''\in\NN^d$, $A_{\alpha'+\alpha''}\leq c_0H^{|\alpha'|+|\alpha''|}A_{\alpha'}A_{\alpha''}$; see the definition of $\Gamma_{A_p,\rho}^{\{M_p\},\infty}(\RR^{2d};h)$). Thus, the mapping $\lambda\mapsto b^{(k_0)}_{\lambda}$, $\RR_+\rightarrow \Gamma_{A_p,\rho}^{\{M_p\},\infty}(\RR^{2d};h_1H)$, is continuous, hence, also continuous if we regard it as a mapping from $\RR_+$ to $\Gamma_{A_p,\rho}^{\{M_p\},\infty}(\RR^{2d})$. Now, (\ref{eeq1}) and Lemma \ref{operatoronl2} yield that $\{(b^{(k_0)}_{\lambda})^w|\,\lambda>0\}$ extends to a bounded subset of $\mathcal{L}_b(L^2(\RR^d), L^2(\RR^d))$ and $(b^{(k_0)}_{\lambda})^w\rightarrow (b^{(k_0)}_{\lambda_0})^w$ in $\mathcal{L}_p(L^2(\RR^d),L^2(\RR^d))$ as $\lambda\rightarrow \lambda_0$. Thus, for each $v\in L^2(\RR^d)$, the function $\lambda\mapsto (b^{(k_0)}_{\lambda})^wv$, $\RR_+\rightarrow L^2(\RR^d)$, is continuous. Now, (\ref{equforequikno}) proves the continuity of $\lambda\mapsto S^{(k_0)}_{\lambda}v$, $\RR_+\rightarrow L^2(\RR^d)$. Moreover, since $\{(1+\lambda)^{k_0} S^{(k_0)}_{\lambda}|\, \lambda>0\}$ is equicontinuous in $\mathcal{L}(\SSS'^*(\RR^d),\SSS^*(\RR^d))$ (cf. Proposition \ref{knoforcomp}), the $L^2(\RR^d)$-valued function $\lambda\mapsto\lambda^{z-1}S^{(k_0)}_{\lambda}v$ is Bochner integrable for each $v\in L^2(\RR^d)$. Since $\lambda\mapsto \lambda^{z-1}\left(\overline{A}(\overline{A}+\lambda\mathrm{Id})^{-1}\right)^{k_0}v$ is Bochner integrable for each $v\in D(\overline{A}^{[\mathrm{Re}\,z]+1})$, (\ref{equforequikno}) implies that for $v\in D(\overline{A}^{[\mathrm{Re}\,z]+1})$, the function $\lambda\mapsto \lambda^{z-1}(b^{(k_0)}_{\lambda})^w v$ is also Bochner integrable and
\beqs
\overline{A}^z v=\gamma_{k_0}(z)\int_0^{\infty} \lambda^{z-1}(b^{(k_0)}_{\lambda})^w v d\lambda+ \gamma_{k_0}(z)\int_0^{\infty} \lambda^{z-1}S^{(k_0)}_{\lambda} v d\lambda.
\eeqs
Thus, we proved the following proposition.

\begin{proposition}\label{rkpropo117}
The mapping $\lambda\mapsto b^{(k_0)}_{\lambda}$ is continuous as a mapping from $\RR_+$ into $\Gamma^{(M_p),\infty}_{A_p,\rho}(\RR^{2d};m')$ for some $m'>0$ (resp. into $\Gamma_{A_p,\rho}^{\{M_p\},\infty}(\RR^{2d};h_1)$ for some $h_1>0$). Moreover, $\{(b^{(k_0)}_{\lambda})^w|\,\lambda>0\}$ is a bounded subset of $\mathcal{L}_b(L^2(\RR^d), L^2(\RR^d))$.\\
\indent For each $v\in L^2(\RR^d)$, the mappings $\lambda\mapsto (b^{(k_0)}_{\lambda})^wv$ and $\lambda\mapsto S^{(k_0)}_{\lambda}v$, $\RR_+\rightarrow L^2(\RR^d)$, are continuous. Furthermore, for each $v\in D(\overline{A}^{[\mathrm{Re}\,z]+1})$, $\lambda\mapsto \lambda^{z-1}(b^{(k_0)}_{\lambda})^wv$, $\RR_+\rightarrow L^2(\RR^d)$, is Bochner integrable and for each $v\in L^2(\RR^d)$, the mapping $\lambda\mapsto \lambda^{z-1}S^{(k_0)}_{\lambda}v$, $\RR_+\rightarrow L^2(\RR^d)$, is Bochner integrable.
\end{proposition}

Fix $v\in L^2(\RR^d)$ and denote
\beq
\psi&=&\gamma_{k_0}(z)\int_0^{\infty} \lambda^{z-1}S^{(k_0)}_{\lambda} v d\lambda\in L^2(\RR^d);\nonumber\\
\tilde{C}_z&=&\left|\gamma_{k_0}(z)\right|\int_0^{\infty} \lambda^{\mathrm{Re}\,z-1}(1+\lambda)^{-k_0} d\lambda>0.\label{definofc}
\eeq
Let $\varepsilon>0$. Since $\{(1+\lambda)^{k_0} S^{(k_0)}_{\lambda}|\, \lambda>0\}$ is an equicontinuous subset of $\mathcal{L}(\SSS'^*(\RR^d),\SSS^*(\RR^d))$, there exists a neighbourhood of zero $W$ in $\SSS'^*(\RR^d)$ such that $(1+\lambda)^{k_0}|\langle v,{}^tS^{(k_0)}_{\lambda} f\rangle|\leq \varepsilon/\tilde{C}_z$ for all $f\in W$. Hence, for $u\in L^2(\RR^d)\cap W$, by the properties of the Bochner integral, we have
\beqs
|\langle \psi,u\rangle|\leq \left|\gamma_{k_0}(z)\right|\int_0^{\infty} \lambda^{\mathrm{Re}\,z-1}(1+\lambda)^{-k_0}(1+\lambda)^{k_0} |\langle S^{(k_0)}_{\lambda} v,u\rangle| d\lambda\leq \varepsilon.
\eeqs
Hence, $u\mapsto \langle \psi,u\rangle$, can be extended to become a continuous functional on $\SSS'^*(\RR^d)$, i.e. $\psi\in\SSS^*(\RR^d)$. Thus, the linear mapping
\beqs
v\mapsto S^{\uwidehat{z}}(v)=\gamma_{k_0}(z)\int_0^{\infty} \lambda^{z-1}S^{(k_0)}_{\lambda} v d\lambda,\,\, L^2(\RR^d)\rightarrow \SSS^*(\RR^d),
\eeqs
is well defined. Fix $v\in L^2(\RR^d)$, $f\in\SSS'^*(\RR^d)$. There exist $f_n\in L^2(\RR^d)$, $n\in\ZZ_+$, such that $f_n\rightarrow f$ in $\SSS'^*(\RR^d)$. Since $\lambda \mapsto S^{(k_0)}_{\lambda}v$, $\RR_+\rightarrow L^2(\RR^d)$, is continuous, so is the function $F_n:\lambda\mapsto \langle S^{(k_0)}_{\lambda}v, f_n\rangle$, $\RR_+\rightarrow \CC$, for each $n\in\ZZ_+$. Employing the equicontinuity of the set $\{(1+\lambda)^{k_0}S^{(k_0)}_{\lambda}|\, \lambda>0\}$, one can easily verify that $(1+\lambda)^{k_0}F_n(\lambda)$ tends to $\lambda\mapsto (1+\lambda)^{k_0}\langle S^{(k_0)}_{\lambda}v,f\rangle$, uniformly in $\lambda\in\RR_+$. Thus, $\lambda\mapsto \langle S^{(k_0)}_{\lambda}v,f\rangle$, $\RR_+\rightarrow \CC$, is a continuous function and $(1+\lambda)^{k_0}|\langle S^{(k_0)}_{\lambda}v,f\rangle|\leq c'$, for all $\lambda\in\RR_+$. Hence
\beq\label{equ1155}
\langle S^{\uwidehat{z}}v,f\rangle=\lim_{n\rightarrow \infty} \gamma_{k_0}(z)\int_0^{\infty}\lambda^{z-1}F_n(\lambda) d\lambda=\gamma_{k_0}(z)\int_0^{\infty}\lambda^{z-1}\langle S^{(k_0)}_{\lambda} v,f\rangle d\lambda.
\eeq
Next, we prove that $S^{\uwidehat{z}}$ extends to a continuous operator from $\SSS'^*(\RR^d)$ to $\SSS^*(\RR^d)$ (i.e. it is $*$-regularising). For this purpose, let $U$ be a convex circled neighbourhood of zero in $\SSS^*(\RR^d)$ which we can take to be the absolute polar $B^{\circ}$ of a bounded set $B$ in $\SSS'^*(\RR^d)$ ($\SSS^*(\RR^d)$ is reflexive). Again, using the equicontinuity of $\{(1+\lambda)^{k_0} S^{(k_0)}_{\lambda}|\, \lambda>0\}$, the set $K=\bigcup_{\lambda>0} (1+\lambda)^{k_0}({}^tS^{(k_0)}_{\lambda})(B)$ is bounded in $\SSS^*(\RR^d)$ and its absolute polar $K^{\circ}$ is a neighbourhood of zero in $\SSS'^*(\RR^d)$. For $v\in L^2(\RR^d)\cap \tilde{C}_z^{-1} K^{\circ}$ ($\tilde{C}_z$ is defined by (\ref{definofc})), (\ref{equ1155}) implies $S^{\uwidehat{z}}v\in B^{\circ}=U$, i.e. $S^{\uwidehat{z}}$ extends to a continuous operator from $\SSS'^*(\RR^d)$ to $\SSS^*(\RR^d)$. (In fact one can prove that for $g\in\SSS'^*(\RR^d)$, $\lambda^{z-1} S^{(k_0)}_{\lambda} g$ is Pettis integrable and $S^{\uwidehat{z}}g$ is the Pettis integral $\gamma_{k_0}(z)\int_0^{\infty}\lambda^{z-1} S^{(k_0)}_{\lambda} g d\lambda$, but we will not need this fact.) Moreover, for the neighbourhood of zero $B^{\circ}$ of $\SSS^*(\RR^d)$ if we take $K$ as above, we readily see that $\gamma_{k_0}(\mathrm{Re}\, z)(\gamma_{k_0}(z))^{-1}S^{\uwidehat{z}}v\in B^{\circ}$ for all $v\in L^2(\RR^d)\cap K^{\circ}$ and $z\in\CC_{+,\tilde{c}}$ (cf. (\ref{inteq})), i.e. $\{\gamma_{k_0}(\mathrm{Re}\, z)(\gamma_{k_0}(z))^{-1}S^{\uwidehat{z}}|\, z\in\CC_{+,\tilde{c}}\}$ is equicontinuous in $\mathcal{L}(\SSS'^*(\RR^d),\SSS^*(\RR^d))$. Hence,
\beqs
\overline{A}^zv=\gamma_{k_0}(z)\int_0^{\infty}\lambda^{z-1}(b^{(k_0)}_{\lambda})^w v d\lambda+S^{\uwidehat{z}} v,\,\,\, \forall v\in D(\overline{A}^{[\mathrm{Re}\, z]+1}).
\eeqs
Next, we prove that for $v\in D(\overline{A}^{[\mathrm{Re}\, z]+1})$, $(a^{\uwidehat{z}})^w v\in L^2(\RR^d)$ and
\beqs
(a^{\uwidehat{z}})^w v=\gamma_{k_0}(z)\int_0^{\infty}\lambda^{z-1}(b^{(k_0)}_{\lambda})^w v d\lambda.
\eeqs
Fix $v\in D(\overline{A}^{[\mathrm{Re}\, z]+1})$ and denote by $\mathbf{V}$ the function $\lambda\mapsto \lambda^{z-1}(b^{(k_0)}_{\lambda})^w v$, $\RR_+\rightarrow L^2(\RR^d)$; it is continuous and Bochner integrable (cf. Proposition \ref{rkpropo117}). Since $\lambda\mapsto \|\mathbf{V}(\lambda)\|_{L^2(\RR^d)}$, $\RR_+\rightarrow [0,\infty)$, is continuous, for each $n\in\NN$, $n\geq 2$, and $j\in\{2,...,n^2\}$, there exists $\lambda_{j,n}\in [(j-1)/n,j/n]$ such that $\|\mathbf{V}(\lambda_{j,n})\|_{L^2(\RR^d)}=\min\{\|\mathbf{V}(\lambda)\|_{L^2(\RR^d)}|\, \lambda\in[(j-1)/n,j/n]\}$. For $n\in \NN$, $n\geq 2$, we define the following simple function on $\RR_+$ with values in $L^2(\RR^d),$
\beqs
\mathbf{V}_n(\lambda)=\left\{\begin{array}{ll}
0,\,\, \mbox{when}\,\, \lambda\in(0,1/n]\cup (n,\infty),\\
\mathbf{V}(\lambda_{j,n}),\,\, \mathrm{when}\,\, \lambda\in ((j-1)/n,j/n],\,\, j\in\{2,...,n^2\}.
\end{array}\right.
\eeqs
Then, $\|\mathbf{V}_n(\lambda)\|_{L^2(\RR^d)}\leq \|\mathbf{V}(\lambda)\|_{L^2(\RR^d)}$ for all $\lambda\in\RR_+$ and $n\geq 2$. Employing the continuity of $\mathbf{V}$, one easily verifies that $\mathbf{V}_n(\lambda)$ tends to $\mathbf{V}(\lambda)$ pointwise. Since $\mathbf{V}$ is Bochner integrable,the  dominated convergence theorem implies that $\mathbf{V}_n\rightarrow \mathbf{V}$ in $L^1(\RR_+;L^2(\RR^d))$ (the Bochner $L^1$ space). For each $n\in\NN$, $n\geq 2$, observe the function
\beqs
F_n(w)=\frac{\gamma_{k_0}(z)}{n}\sum_{j=2}^{n^2} \lambda_{j,n}^{z-1} b^{(k_0)}_{\lambda_{j,n}}(w),\,\,w\in\RR^{2d}.
\eeqs
Clearly, $F_n\in \Gamma_{A_p,\rho}^{*,\infty}(\RR^{2d})$ and $F_n^w$ extends to a continuous operator on $L^2(\RR^d)$. Moreover
\beqs
F_n^w v=\gamma_{k_0}(z)\int_0^{\infty} \mathbf{V}_n(\lambda)d\lambda.
\eeqs
Hence,
\beq\label{liminl2forsol}
\lim_{n\rightarrow \infty}F_n^w v=\gamma_{k_0}(z)\int_0^{\infty} \mathbf{V}(\lambda)d\lambda,\,\, \mbox{in}\,\, L^2(\RR^d).
\eeq
Next, we prove that $F_n\rightarrow a^{\uwidehat{z}}$ in $\Gamma_{A_p,\rho}^{*,\infty}(\RR^{2d})$. By Proposition \ref{rkpropo117}, $\lambda\mapsto b^{(k_0)}_{\lambda}$, $\RR_+\rightarrow \Gamma_{A_p,\rho}^{(M_p),\infty}(\RR^{2d};m)$, is continuous for some $m>0$ in the Beurling case and $\lambda\mapsto b^{(k_0)}_{\lambda}$, $\RR_+\rightarrow \Gamma_{A_p,\rho}^{\{M_p\},\infty}(\RR^{2d};h)$, is continuous for some $h>0$ in the Roumieu case. Combining (\ref{growthal}) and (\ref{estforthepower}) we deduce that there exists $m>0$ such that for every $h>0$ there exists $C>0$ (resp. there exists $h>0$ such that for every $m>0$ there exists $C>0$) such that
\beq\label{estforbing}
\left|D^{\alpha}_w b^{(k_0)}_{\lambda}(w)\right|\leq C\frac{h^{|\alpha|}A_{\alpha}e^{M(m|\xi|)}e^{M(m|x|)}}{(1+\lambda)^{k_0}\langle w\rangle^{\rho|\alpha|}},\,\, w\in\RR^{2d},\, \alpha\in\NN^{2d},\, \lambda>0.
\eeq
We consider first the Beurling case. Obviously, we can assume that this $m$ is the same as the one for the continuity of $\lambda\mapsto b^{(k_0)}_{\lambda}$, $\RR_+\rightarrow \Gamma_{A_p,\rho}^{(M_p),\infty}(\RR^{2d};m)$. Let $h'>0$ be arbitrary but fixed and let $C'\geq 1$ be a constant for which (\ref{estforbing}) holds for this $h'$. Let $0<\varepsilon<1$. There exist $0<r_1<\varepsilon/\left(4C'(|\gamma_{k_0}(z)|+1)\right)<1<r_2$ such that
\beq\label{hks775511}
|\gamma_{k_0}(z)|\int_0^{r_1}\frac{\lambda^{\mathrm{Re}z-1}}{(1+\lambda)^{k_0}}d\lambda+ |\gamma_{k_0}(z)|\int_{r_2}^{\infty} \frac{\lambda^{\mathrm{Re}z-1}}{(1+\lambda)^{k_0}}d\lambda\leq \frac{\varepsilon}{4C'}.
\eeq
If $\mathrm{Re}\, z>1$, $\lambda\mapsto \lambda^{\mathrm{Re}\, z-1}(1+\lambda)^{-k_0}$ is decreasing on $[c_1,\infty)$ for some $c_1>1$. In this case, without losing generality, we can assume that $r_2>c_1$. Observe that $\lambda\mapsto \lambda^{\mathrm{Re}\, z-1}(1+\lambda)^{-k_0}$ is decreasing on $\RR_+$ if $0<\mathrm{Re}\, z\leq 1$. As the mapping $\lambda\mapsto \lambda^{z-1}b^{(k_0)}_{\lambda}$, $[r_1/2,2r_2]\rightarrow \Gamma_{A_p,\rho}^{(M_p),\infty}(\RR^{2d};m)$, is uniformly continuous, there exists $\delta>0$ such that for arbitrary $\lambda',\lambda''\in[r_1/2,2r_2]$ with $|\lambda'-\lambda''|\leq \delta$ the following holds
\beqs
\|\lambda'^{z-1}b^{(k_0)}_{\lambda'}-\lambda''^{z-1}b^{(k_0)}_{\lambda''}\|_{\Gamma,h',m} \leq \varepsilon/\left(4(2r_2-r_1/2)|\gamma_{k_0}(z)|\right).
\eeqs
Pick $n_0\in\ZZ_+$ such that $n_0\geq \max\{\delta^{-1},4/r_1,4r_2\}$. Let $n\in\ZZ_+$, $n\geq n_0$, be arbitrary but fixed. From the definition of $n_0$, it follows that for fixed $n$ there exist $l,l'\in\{3,...,n^2-1\}$ such that $(l-1)/n\leq r_1/2<l/n<r_1$ and $r_2<(l'-1)/n<l'/n\leq 2r_2<(l'+1)/n$. Then
\begin{align*}
\big|D^{\alpha}a^{\uwidehat{z}}(w)&-D^{\alpha}F_n(w)\big|\\
&\leq |\gamma_{k_0}(z)|\int_0^{1/n} \left|\lambda^{z-1}D^{\alpha}b^{(k_0)}_{\lambda}(w)\right|d\lambda+ |\gamma_{k_0}(z)|\int_n^{\infty} \left|\lambda^{z-1}D^{\alpha}b^{(k_0)}_{\lambda}(w)\right|d\lambda\\
&{}\,\,\,\,+ |\gamma_{k_0}(z)|\sum_{j=2}^{n^2}\int_{(j-1)/n}^{j/n}\left|\lambda^{z-1} D^{\alpha}b^{(k_0)}_{\lambda}(w) -\lambda_{j,n}^{z-1}D^{\alpha}b^{(k_0)}_{\lambda_{j,n}}(w)\right|d\lambda\\
&\leq|\gamma_{k_0}(z)|\int_0^{l/n} \left|\lambda^{z-1}D^{\alpha}b^{(k_0)}_{\lambda}(w)\right|d\lambda+ |\gamma_{k_0}(z)|\int_{l'/n}^{\infty} \left|\lambda^{z-1}D^{\alpha}b^{(k_0)}_{\lambda}(w)\right|d\lambda\\
&{}\,\,\,\,+|\gamma_{k_0}(z)|\sum_{j=2}^l\int_{(j-1)/n}^{j/n} \left|\lambda_{j,n}^{z-1}D^{\alpha}b^{(k_0)}_{\lambda_{j,n}}(w)\right|d\lambda\\
&{}\,\,\,\,+ |\gamma_{k_0}(z)|\sum_{j=l'+1}^{n^2}\int_{(j-1)/n}^{j/n} \left|\lambda_{j,n}^{z-1}D^{\alpha}b^{(k_0)}_{\lambda_{j,n}}(w)\right|d\lambda\\
&{}\,\,\,\,+ |\gamma_{k_0}(z)|\sum_{j=l+1}^{l'}\int_{(j-1)/n}^{j/n}\left|\lambda^{z-1} D^{\alpha}b^{(k_0)}_{\lambda}(w) -\lambda_{j,n}^{z-1}D^{\alpha}b^{(k_0)}_{\lambda_{j,n}}(w)\right|d\lambda\\
&=I_1+I_2+I_3+I_4+I_5.
\end{align*}
Employing (\ref{estforbing}) and (\ref{hks775511}), we conclude
\beqs
I_1+I_2\leq\frac{\varepsilon}{4}\cdot\frac{h'^{|\alpha|}A_{\alpha}e^{M(m|\xi|)}e^{M(m|x|)}} {\langle w\rangle^{\rho|\alpha|}}.
\eeqs
Next, we estimate $I_3$. The inequality (\ref{estforbing}) implies
\beqs
I_3\leq C'|\gamma_{k_0}(z)|\frac{h'^{|\alpha|}A_{\alpha}e^{M(m|\xi|)}e^{M(m|x|)}}{\langle w\rangle^{\rho|\alpha|}}\sum_{j=2}^l\int_{(j-1)/n}^{j/n} \frac{\lambda_{j,n}^{\mathrm{Re}\,z-1}}{(1+\lambda_{j,n})^{k_0}}d\lambda.
\eeqs
If $0<\mathrm{Re}\, z\leq 1$, $\lambda\mapsto \lambda^{\mathrm{Re}z-1}(1+\lambda)^{-k_0}$ is decreasing, hence,
\beqs
\sum_{j=2}^l\int_{(j-1)/n}^{j/n} \frac{\lambda_{j,n}^{\mathrm{Re}\,z-1}}{(1+\lambda_{j,n})^{k_0}}d\lambda\leq\int_{0}^{r_1} \frac{\lambda^{\mathrm{Re}\,z-1}}{(1+\lambda)^{k_0}}d\lambda
\eeqs
and (\ref{hks775511}) yields $I_3\leq (\varepsilon/4)\cdot h'^{|\alpha|}A_{\alpha}e^{M(m|\xi|)}e^{M(m|x|)}\langle w\rangle^{-\rho|\alpha|}$. If $\mathrm{Re}\,z>1$, then
\beqs
\lambda_{j,n}^{\mathrm{Re}\,z-1}(1+\lambda_{j,n})^{-k_0}\leq 1,\,\,\, \forall j\in\{2,...,l\}.
\eeqs
Thus, by the choice of $r_1$, we arrive at the same estimate for $I_3$ in the case $\mathrm{Re}\,z>1$ as well. For $I_4$, observe that $\lambda\mapsto \lambda^{\mathrm{Re}\,z-1}(1+\lambda)^{-k_0}$ is decreasing on $[r_2,\infty)$ (by the choice of $r_2$), hence
\beqs
I_4&\leq&C'|\gamma_{k_0}(z)|\frac{h'^{|\alpha|}A_{\alpha}e^{M(m|\xi|)}e^{M(m|x|)}}{\langle w\rangle^{\rho|\alpha|}}\int_{r_2}^{\infty} \frac{\lambda^{\mathrm{Re}\,z-1}}{(1+\lambda)^{k_0}}d\lambda\\
&\leq&\frac{\varepsilon}{4}\cdot\frac{h'^{|\alpha|}A_{\alpha}e^{M(m|\xi|)}e^{M(m|x|)}} {\langle w\rangle^{\rho|\alpha|}}.
\eeqs
Lastly, $I_5\leq (\varepsilon/4)\cdot h'^{|\alpha|}A_{\alpha}e^{M(m|\xi|)}e^{M(m|x|)}\langle w\rangle^{-\rho|\alpha|}$. Thus, $\|a^{\uwidehat{z}}-F_n\|_{\Gamma,h',m}\leq \varepsilon$, for all $n\in\ZZ_+$, $n\geq n_0$. The proof in the Roumieu case is completely analogous and we omit it. This yields $F_n^w v\rightarrow (a^{\uwidehat{z}})^w v$ in $\SSS'^*(\RR^d)$. This and (\ref{liminl2forsol}) imply that $(a^{\uwidehat{z}})^w v\in L^2(\RR^d)$ and
\beqs
(a^{\uwidehat{z}})^w v=\gamma_{k_0}(z)\int_0^{\infty}\lambda^{z-1}(b^{(k_0)}_{\lambda})^w v d\lambda.
\eeqs
We proved that $\overline{A}^z v=(a^{\uwidehat{z}})^w v+ S^{\uwidehat{z}} v$ and $(a^{\uwidehat{z}})^w v\in L^2(\RR^d)$, for all $v\in D(\overline{A}^{[\mathrm{Re}\, z]+1})$. Hence, $\overline{A}^z -S^{\uwidehat{z}}$ is a closed extension of $(a^{\uwidehat{z}})^w|_{\SSS^*(\RR^d)}$. As $a^{\uwidehat{z}}$ is hypoelliptic, Proposition \ref{maximalreal} implies that $\overline{A}^z -S^{\uwidehat{z}}$ is a closed extension of the maximal realisation of $(a^{\uwidehat{z}})^w$ and as $(a^{\uwidehat{z}})^w v\in L^2(\RR^d)$, for all $v\in D(\overline{A}^{[\mathrm{Re}\, z]+1})$, we conclude that
\beqs
D(\overline{A}^z)=\{v\in L^2(\RR^d)|\, (a^{\uwidehat{z}})^w v\in L^2(\RR^d)\}.
\eeqs
\indent It remains to prove the last part of Theorem \ref{maint} $(vi)$ concerning the analyticity of the stated mappings. Since $\{(1+\lambda)^{k_0}S^{(k_0)}_{\lambda}|\, \lambda>0\}$ is bounded in $\mathcal{L}(L^2(\RR^d),L^2(\RR^d))$, the function $\lambda\mapsto \lambda^{z-1}\ln\lambda\, S^{(k_0)}_{\lambda}v$, $\RR_+\rightarrow L^2(\RR^d)$, is Bochner integrable for each $v\in L^2(\RR^d)$. As an easy consequence of this, we derive that for each $v\in L^2(\RR^d)$, the function $z\mapsto S^{\uwidehat{z}}v$, $\mathrm{int}\,\CC_{+,\tilde{c}}\rightarrow L^2(\RR^d)$, is analytic. Since $z\mapsto \overline{A}^z v$, $\mathrm{int}\,\CC_{+,\tilde{c}}\rightarrow L^2(\RR^d)$, is analytic for each $v\in D(A^{k_0})$ (see \cite[Theorem 3.1.5, p. 61]{Fractionalpowersbook}), we conclude the last part of Theorem \ref{maint} $(vi)$.

\section{Application: Semigroups generated by square roots of non-negative infinite order operators}\label{section6}

If $A$ is a non-negative operator with a dense domain in $L^2(\RR^d)$, then it is known  that $-A^{1/2}$ is the infinitesimal generator of an analytic semigroup (see Subsection \ref{sec_smgr_portk} below). The goal of this section is to apply Theorem \ref{maint} to prove that if a $A$ is the $L^2(\RR^d)$-closure of the Weyl quantisation of an appropriate hypoelliptic symbol in $\Gamma_{A_p,\rho}^{*,\infty}(\RR^{2d}),$ then all the operators in the semigroup are pseudodifferential with symbols in $\Gamma_{A_p,\rho}^{*,\infty}(\RR^{2d})$ modulo $*$-regularising operators. As it turns out, the heat parametrix is a crucial tool and thus we devote a whole subsection to it. We give the construction, in a slightly more general setting than what needed because the result is interesting by itself. However, before we proceed, we need a few preliminary estimates.

\subsection{Estimates for the derivative of $(b(w))^n$ and $e^{-tb(w)}$ for hypoelliptic $b\in\Gamma_{A_p,\rho}^{*,\infty}(\RR^{2d})$}

We need a few technical results.

\begin{lemma}\label{est_sequ_bi}(\cite[Lemma 3.2]{CPP})
Let $N_p$ be a sequence of positive numbers satisfying $(M.4)$ and $N_0 = N_1 = 1$. Then for all $\alpha,\beta\in\NN^d$ such that $\beta\leq \alpha$ and $1\leq |\beta|\leq |\alpha|-1$ the inequality ${\alpha\choose\beta}N_{\alpha-\beta}N_{\beta}\leq |\alpha|N_{|\alpha|-1}$ holds.
\end{lemma}

\begin{lemma}\label{est_ontheseqa}
Let $N_p$ be a sequence of positive numbers such that $N_0=N_1=1$. Assume that $N_p$ satisfies $(M.1)$. Then for all $j,k\in\ZZ_+$ such that $j\leq k$, $N_jN_{k_1}\cdot\ldots \cdot N_{k_j}\leq N_k$ for all $k_1,\ldots,k_j\in\ZZ_+$ such that $k_1+\ldots+k_j=k$.
\end{lemma}

\begin{proof} Let $j\in\ZZ_+$ be arbitrary but fixed. The proof is by induction on $k$. For $k=j$ the claim is trivial. Assume that it holds for some $k\in\ZZ_+$. To prove it for $k+1,$ let $k_1,\ldots ,k_j$ be positive integers with sum $k+1$. There exists $s\in\{1,\ldots,j\}$ such that $k_s\geq 2$. We have
\beqs
N_jN_{k_1}\cdot\ldots \cdot N_{k_j}=N_jN_{k_1}\cdot\ldots \cdot N_{k_s-1}\cdot\ldots \cdot N_{k_j}\cdot N_{k_s}/N_{k_s-1}\leq N_kN_{k_s}/N_{k_s-1}.
\eeqs
As $(M.1)$ implies $N_{k_s}/N_{k_s-1}\leq N_{k+1}/N_k$, the proof is completed.
\end{proof}

We recall the following multidimensional variant of the Fa\`a di Bruno formula (see \cite[Corollary 2.10]{faadib}).

\begin{proposition}(\cite{faadib})
Let $|\alpha|=n\geq 1$ and $h(x_1,...,x_d)=f(g(x_1,...,x_d))$ with $g\in \mathcal{C}^{n}$ in a neighbourhood of $x^0$ and $f\in\mathcal{C}^n$ in a neighbourhood of $y^0=g(x^0)$. Then
\beqs
\partial^{\alpha}h(x^0)=\sum_{r=1}^{n}f^{(r)}(y^0)\sum_{p(\alpha,r)}\alpha!\prod_{j=1}^n \frac{\left(\partial^{\alpha^{(j)}}g(x^0)\right)^{k_j}}{k_j! \left(\alpha^{(j)}!\right)^{k_j}},
\eeqs
where
\beqs
p(\alpha,r)&=&\Big\{\left(k_1,...,k_n; \alpha^{(1)},...,\alpha^{(n)}\right)\Big|\, \mbox{for some}\, 1\leq s\leq n, k_j=0 \mbox{ and } \alpha^{(j)}=0\\
&{}&\mbox{for } 1\leq j\leq n-s;\, k_j>0 \mbox{ for } n-s+1\leq j\leq n; \mbox{ and}\\
&{}&0\prec \alpha^{(n-s+1)}\prec...\prec \alpha^{(n)} \mbox{ are such that}\\
&{}&\sum_{j=1}^n k_j=r,\, \sum_{j=1}^n k_j\alpha^{(j)}=\alpha\Big\}.
\eeqs
\end{proposition}

In the above formula, the convention $0^0=1$ is used. The relation $\prec$ used in this proposition is defined in the following way (cf. \cite{faadib}). We say that $\beta\prec\alpha$ when one of the following holds:
\begin{itemize}
\item[$(i)$] $|\beta|<|\alpha|$;
\item[$(ii)$] $|\beta|=|\alpha|$ and $\beta_1<\alpha_1$;
\item[$(iii)$] $|\beta|=|\alpha|$, $\beta_1=\alpha_1,...,\beta_k=\alpha_k$ and $\beta_{k+1}<\alpha_{k+1}$ for some $1\leq k<d$.
\end{itemize}

\begin{lemma}\label{1110001}
For $\beta\in\NN^{d}\backslash\{0\}$, the following estimate holds:
\beqs
\sum_{r=1}^{|\beta|}{|\beta|\choose r}\sum_{p(\beta,r)}\frac{r!}{k_1!\cdot ... \cdot k_{|\beta|}}\leq 2^{|\beta|(d+1)}.
\eeqs
\end{lemma}

\begin{proof} Let $f(\lambda)=\lambda^{|\beta|}$ and $g(x)=(-x_1)^{-1}\cdot...\cdot(-x_d)^{-1}$. We apply the Fa\`a di Bruno formula to $\partial^{\beta}\left(f(g(x))\right)$ at $x=(-1,...,-1)$:
\beqs
\prod_{\substack{j=1\\ \beta_j\neq 0}}^d \left(|\beta|(|\beta|+1)\cdot...\cdot(|\beta|+\beta_j-1)\right)&=&\partial^{\beta}(f\circ g)(-1,\ldots,-1)\\
&=& \sum_{r=1}^{|\beta|}\frac{|\beta|!}{(|\beta|-r)!} \sum_{p(\beta,r)}\beta!\prod_{j=1}^{|\beta|}\frac{1}{k_j!}.
\eeqs
Hence, we infer that
\beqs
\prod_{\substack{j=1\\ \beta_j\neq 0}}^d \frac{|\beta|(|\beta|+1)\cdot...\cdot(|\beta|+\beta_j-1)}{\beta_j!}= \sum_{r=1}^{|\beta|}{|\beta|\choose r} \sum_{p(\beta,r)}\frac{r!}{k_1!\cdot ... \cdot k_{|\beta|}}.
\eeqs
From this, the desired inequality follows, since
\beqs
\frac{|\beta|(|\beta|+1)\cdot...\cdot(|\beta|+\beta_j-1)}{\beta_j!}=\frac{(|\beta|+\beta_j)!} {\beta_j!(|\beta|-1)!(|\beta|+\beta_j)}\leq {|\beta|+\beta_j\choose \beta_j}\leq 2^{|\beta|+\beta_j}.
\eeqs
\end{proof}

\begin{lemma}\label{flp11}
Let $b\in\Gamma^{*,\infty}_{A_p,\rho}(\RR^{2d})$ be hypoelliptic. Then for every $h>0$ there exists $C>0$ (resp. there exist $h,C>0$) such that
\beqs
\left|D^n_tD^{\alpha}_w (e^{-tb(w)})\right|\leq C2^nh^{|\alpha|}A_{\alpha}\langle w\rangle^{-\rho|\alpha|}|b(w)|^ne^{-t\mathrm{Re}\, b(w)} \sum_{r=0}^{|\alpha|}\frac{|t|^r|b(w)|^r}{r!},
\eeqs
for all $\alpha\in\NN^{2d}$, $n\in\NN$, $t\in\RR$, $w\in Q^c_B$, where $B$ is the constant from the hypoellipticity conditions (\ref{dd1}) and (\ref{dd2}) on $b$. If these hold for all $w\in\RR^{2d}$ (i.e. if $B=0$), then the above estimate holds for all $w\in\RR^{2d}$ as well.
\end{lemma}

\begin{proof} We apply the Fa\`a di Bruno formula to $\partial^{\alpha}(f(g(w)))$, $\alpha\neq 0$, with $g(w)=b(w)$ and $f(\lambda)=e^{-t\lambda}$. By the hypoellipticity of $b$, we see that for each $h>0$ there exists $C>0$ (resp. there exist $h,C>0$) such that
\beqs
\left|D^{\alpha}_w (e^{-tb(w)})\right|&\leq& h^{|\alpha|}\langle w\rangle^{-\rho|\alpha|}\sum_{r=1}^{|\alpha|}|t|^re^{-t\mathrm{Re}\, b(w)}|b(w)|^r\sum_{p(\alpha,r)}\alpha!\prod_{j=1}^{|\alpha|} \frac{C^{k_j}A_{\alpha^{(j)}}^{k_j}}{k_j! \left(\alpha^{(j)}!\right)^{k_j}}\\
&\leq&(2dh)^{|\alpha|}\langle w\rangle^{-\rho|\alpha|}e^{-t\mathrm{Re}\, b(w)}\sum_{r=1}^{|\alpha|}|t|^r|b(w)|^r\sum_{p(\alpha,r)}\alpha!\prod_{j=1}^{|\alpha|} \frac{C^{k_j}A_{\alpha^{(j)}}^{k_j}}{k_j!(|\alpha^{(j)}|!)^{k_j}}.
\eeqs
We consider first the Beurling case. Let $h_1>0$ be arbitrary but fixed and choose $h=h_1/(2^{2d+2}d)$. For this $h$ there exists $C\geq 1$ such that the above estimate holds. As $A_p$ satisfies $(M.3)'$ there exists $C'\geq C$ such that $C^n\leq C'A_n/n!$, $\forall n\in\NN$. Again, by $(M.3)'$ for $A_p$, for this $C'$, there exists $C''\geq C'$ such that $C'^n\leq C''A_n/n!$, $\forall n\in\NN$. By the definition of $p(\alpha,r)$, $k_j\neq0$ if and only if $\alpha^{(j)}\neq 0$. As $A_p/p!$ satisfies $(M.1)$ (by $(M.4)$ for $A_p$), Lemma \ref{est_ontheseqa} implies
\beqs
C^{k_j}A_{\alpha^{(j)}}^{k_j}/(|\alpha^{(j)}|!)^{k_j}\leq C'^{\mathrm{sgn}\,k_j}A_{|\alpha^{(j)}|k_j}/(|\alpha^{(j)}|k_j)!,
\eeqs
where $\mathrm{sgn}\, n$ equals $0$ when $n=0$ and equals $1$ when $n>0$. Applying again Lemma \ref{est_ontheseqa}, we have
\beqs
\prod_{j=1}^{|\alpha|}C'^{\mathrm{sgn}\,k_j}A_{|\alpha^{(j)}|k_j}/(|\alpha^{(j)}|k_j)!\leq C''A_{\alpha}/|\alpha|!.
\eeqs
We deduce
\beqs
\left|D^{\alpha}_w (e^{-tb(w)})\right|\leq C''(2dh)^{|\alpha|}A_{\alpha}\langle w\rangle^{-\rho|\alpha|}e^{-t\mathrm{Re}\, b(w)} \sum_{r=1}^{|\alpha|}\frac{|t|^r|b(w)|^r}{r!}\sum_{p(\alpha,r)}\frac{r!}{k_1!\cdot\ldots \cdot k_{|\alpha|}!}.
\eeqs
Since Lemma \ref{1110001} implies $\sum_{p(\alpha,r)}r!/(k_1!\cdot\ldots \cdot k_{|\alpha|}!)\leq 2^{|\alpha|(2d+1)}$, we have
\beq\label{est_12}
\left|D^{\alpha}_w (e^{-tb(w)})\right|\leq C'' h_1^{|\alpha|}A_{\alpha}\langle w\rangle^{-\rho|\alpha|}e^{-t\mathrm{Re}\, b(w)} \sum_{r=1}^{|\alpha|}\frac{|t|^r|b(w)|^r}{r!}.
\eeq
With analogous technique, one proves that in the Roumieu case there exist $h_1,C''>0$ for which (\ref{est_12}) holds. By allowing the sum in (\ref{est_12}) to start from $r=0$, this estimate becomes trivially valid for $\alpha=0$.\\
\indent Next, we estimate $D^{\alpha}_w\left((b(w))^n\right)$, for $\alpha\in\NN^{2d}\backslash\{0\}$ and $n\in\ZZ_+$. Applying the Fa\`a di Bruno formula to the composition $f(g(w))$ with $g(w)=b(w)$ and $f(\lambda)=\lambda^n$ and using the hypoellipticity of $b$, we infer that for each $h>0$ there exists $C>0$ (resp. there exist $h,C>0$) such that
\beqs
\left|D^{\alpha}_w (b(w))^n\right|\leq h^{|\alpha|}\langle w\rangle^{-\rho|\alpha|}|b(w)|^n\sum_{r=1}^{\min\{|\alpha|,n\}}\frac{n!}{(n-r)!}\sum_{p(\alpha,r)}\alpha!\prod_{j=1}^{|\alpha|} \frac{C^{k_j}A_{\alpha^{(j)}}^{k_j}}{k_j! \left(\alpha^{(j)}!\right)^{k_j}}.
\eeqs
Now, the same procedure as above shows that for each $h>0$ there exists $C>0$ (resp. there exist $h,C>0$) such that
\beq\label{est_113}
\left|D^{\alpha}_w (b(w))^n\right|\leq C2^nh^{|\alpha|}A_{\alpha}\langle w\rangle^{-\rho|\alpha|}|b(w)|^n.
\eeq
This estimate trivially holds when $n=0$ or $\alpha=0$. Finally, since $\left|D^n_tD^{\alpha}_w (e^{-tb(w)})\right|=\left|D^{\alpha}_w\left((b(w))^ne^{-tb(w)}\right)\right|$, (\ref{est_12}) and (\ref{est_113}) together with the Leibniz rule yield the estimate in the lemma.
\end{proof}

\begin{remark}\label{kft551133}
This lemma (and especially its proof) can be employed to construct interesting infinite order hypoelliptic symbols.\\
\indent Let $a$ be positive smooth function on $Q_B^c$ (for some $B\geq 0$) which satisfies the estimate: for every $h>0$ there exists $C>0$ (resp. there exist $h,C>0$) such that
\beq\label{kkl551133}
|D^{\alpha}a(w)|\leq C h^{|\alpha|}A_{\alpha}a(w)\langle w\rangle^{-\rho_1|\alpha|},\,\, w\in Q_B^c,\, \alpha\in\NN^{2d},
\eeq
for some $0<\rho_1\leq 1$ (not necessarily the same $\rho$ from $\Gamma^{*,\infty}_{A_p,\rho}(\RR^{2d})$). For fixed $s>1$, we can estimate the derivative of $(a(w))^{1/s}$ by applying the Fa\`a di Bruno formula to $\partial^{\alpha}(f(a(w)))$ with $f(\lambda)=\lambda^{1/s}$, $\lambda>0$. In fact, by employing the same technique as in the proof of Lemma \ref{flp11} we deduce the following estimate: for every $h>0$ there exists $C>0$ (resp. there exist $h,C>0$) such that
\beqs
|D^{\alpha}a(w)^{1/s}|\leq C h^{|\alpha|}A_{\alpha}a(w)^{1/s}\langle w\rangle^{-\rho_1|\alpha|},\,\, w\in Q_B^c,\, \alpha\in\NN^{2d}.
\eeqs
Hence, (\ref{est_12}) is valid on $Q^c_B$ for $b(w)=a(w)^{1/s}$, $t=\pm1$ and with $\rho_1$ in place of $\rho$. This is particularly interesting if one takes for $a$ a positive elliptic Shubin symbol of finite order $m\geq 1$ (i.e $c'\langle w\rangle^m\leq a(w)\leq C'\langle w\rangle^m$ on $\RR^{2d}$) that additionally satisfies (\ref{kkl551133}) for some $\rho_1>\rho$ and $B=0$ (i.e. on the whole $\RR^{2d}$). In this case (\ref{est_12}) boils down to: for every $h>0$ there exists $C>0$ (resp. there exist $h,C>0$) such that
\beqs
\big|D^{\alpha}e^{\pm a(w)^{1/(sm)}}\big|\leq C h^{|\alpha|}A_{\alpha}e^{\pm a(w)^{1/(sm)}}\langle w\rangle^{-(\rho_1-\frac{1}{s})|\alpha|}.
\eeqs
Hence, by taking $s>1$ large enough such that $\rho_1-\frac{1}{s}\geq \rho$ and the function $e^{\langle w\rangle^{1/s}}$ to be of ultrapolynomial growth of class $*$, $e^{\pm a(w)^{1/(sm)}}$ becomes $\Gamma^{*,\infty}_{A_p,\rho}$-hypoelliptic.
\end{remark}

The estimate (\ref{est_113}) will be particularly important for the sequel. We state this result separately for future reference.

\begin{corollary}\label{flp1133}
Under the assumptions of Lemma \ref{flp11} the following estimate holds: for every $h>0$ there exists $C>0$ (there exist $h,C>0$) such that
\beqs
\left|D^{\alpha}_w (b(w))^n\right|\leq C2^nh^{|\alpha|}A_{\alpha}\langle w\rangle^{-\rho|\alpha|}|b(w)|^n,
\eeqs
for all $\alpha\in\NN^{2d}$, $n\in\NN$, $w\in Q^c_B$, where $B$ is the constant from the hypoellipticity conditions (\ref{dd1}) and (\ref{dd2}) on $b$. If these hold for all $w\in\RR^{2d}$ (i.e. if $B=0$) then the above estimate holds for all $w\in\RR^{2d}$ as well.
\end{corollary}

\subsection{The heat parametrix}

Throughout this subsection, we assume that $b$ is a hypoelliptic symbol in $\Gamma^{*,\infty}_{A_p,\rho}(\RR^{2d})$ for which the condition (\ref{dd2}) holds on the whole $\RR^{2d}$. Furthermore, we assume that there exists $c>0$ such that
\beq\label{est_semigroup1}
\mathrm{Re}\, b(w)> c|\mathrm{Im}\, b(w)|,\,\, \forall w\in\RR^{2d};
\eeq
hence $\mathrm{Re}\,b(w)>0$, $\forall w\in\RR^{2d}$. Incidentally, notice that this implies that (\ref{dd1}) also holds on the whole $\RR^{2d}$.\\
\indent We want to find $u_j\in\mathcal{C}^{\infty}(\RR\times\RR^{2d})$, $j\in\NN$, which solve the system
\beq\label{system1}
\left\{\begin{array}{lll}
\partial_t u_j+\sum_{k+l=j}\sum_{|\mu+\nu|=l}\frac{(-1)^{|\nu|}}{\mu!\nu!2^l}\partial^{\mu}_{\xi} D^{\nu}_x b\cdot \partial^{\nu}_{\xi} D^{\mu}_x u_k=0,\,\,\ j\in\NN,\\
u_0(0,x,\xi)=1,\\
u_j(0,x,\xi)=0,\,\, j\in\ZZ_+.
\end{array}\right.
\eeq
We can immediately conclude $u_0(t,x,\xi)=e^{-tb(x,\xi)}\in \mathcal{C}^{\infty}(\RR\times\RR^{2d})$. Obviously, the system has a unique solution (in view of the initial conditions), since in the $j$-th step one only solves one equation in $u_j$ as $u_0,\ldots,u_{j-1}$ are already determined. We want to estimate the derivatives of $u_j$.

\begin{lemma}\label{est_heat_ker}
For every $h>0$ there exists $C>0$ (resp. there exist $h,C>0$) such that
\beqs
|D^n_tD^{\alpha}_wu_j(t,w)|\leq Cn!h^{|\alpha|+2j}A_{|\alpha|+2j}\left(\mathrm{Re}\, b(w)\right)^n\langle w\rangle^{-\rho(|\alpha|+2j)}e^{-\frac{t}{4}\mathrm{Re}\, b(w)},
\eeqs
for all $\alpha\in\NN^{2d}$, $n\in\NN$, $(t,w)\in[0,\infty)\times\RR^{2d}$.
\end{lemma}

\begin{proof} Put $p_j=e^{tb}u_j$, $j\in\NN$. Clearly $p_0=1$. The $j$-th equation of the system becomes
\beqs
e^{-tb}\partial_t p_j+\sum_{l=1}^j\sum_{|\mu+\nu|=l}\frac{(-1)^{|\nu|}}{\mu!\nu!2^l} \partial^{\mu}_{\xi}D^{\nu}_xb\cdot\partial^{\nu}_{\xi}D^{\mu}_x(e^{-tb}p_{j-l})=0.
\eeqs
As $0=u_j(0,x,\xi)=p_j(0,x,\xi)$, $j\in\ZZ_+$, it follows that
\beqs
u_j(t,w)=-\sum_{l=1}^j\sum_{|\mu+\nu|=l}\frac{(-1)^{|\nu|}}{\mu!\nu!2^l}e^{-tb(w)}\int_0^t e^{sb(w)}\partial^{\mu}_{\xi}D^{\nu}_xb(w)\partial^{\nu}_{\xi}D^{\mu}_x u_{j-l}(s,w)ds.
\eeqs
Let $n\in\NN$, $\alpha\in\NN^{2d}$. Then
\beq
\left|D^{n+1}_tD^{\alpha}_wu_j(t,w)\right|&\leq& \sum_{l=1}^j\sum_{|\mu+\nu|=l}\sum_{\substack{n_1+n_2+n_3=n\\ \beta+\gamma+\delta=\alpha}} \frac{(n+1)!\alpha!}{n_1!n_2!n_3\beta!\gamma!\delta!\mu!\nu!2^l} \left|D^{\beta}_w\left((b(w))^{n_1+n_2}\right)\right| \nonumber \\
&{}&\,\,\,\, \cdot \left|D^{\gamma}_w D^{\mu}_{\xi}D^{\nu}_xb(w)\right|\left|D^{n_3}_tD^{\delta}_w D^{\nu}_{\xi}D^{\mu}_x u_{j-l}(t,w)\right|\nonumber \\
&{}&+\sum_{l=1}^j\sum_{|\mu+\nu|=l}\sum_{\beta+\gamma+\delta+\kappa=\alpha} \frac{\alpha!}{\beta!\gamma!\delta!\kappa!\mu!\nu!2^l} \left|D^{\beta}_w\left((b(w))^{n+1}\right)\right|\nonumber \\
&{}&\,\,\,\,\cdot\int_0^t \left|D^{\gamma}_we^{-(t-s)b(w)}\right|\left|D^{\delta}_w D^{\mu}_{\xi}D^{\nu}_xb(w)\right|\left|D^{\kappa}_wD^{\nu}_{\xi}D^{\mu}_x u_{j-l}(s,w)\right|ds\nonumber \\
&=& S_1+S_2.\label{est_der_forjj}
\eeq
We consider first the Beurling case. Let $h>0$ be arbitrary but fixed. Pick $h_1\leq h/(64^{d+1}d)$. Lemma \ref{flp11} and Corollary \ref{flp1133}, together with (\ref{est_semigroup1}), imply that for the chosen $h_1$ there exists $C_{h_1}\geq1$ such that
\beq
\left|D^{\alpha}_w (b(w))^n\right|&\leq& C_{h_1}c_1^nh_1^{|\alpha|}A_{\alpha}\langle w\rangle^{-\rho|\alpha|}\left(\mathrm{Re}\,b(w)\right)^n,\label{lkl_117} \\
|D^n_tD^{\alpha}_w(e^{-tb(w)})|&\leq& C_{h_1}n!h_1^{|\alpha|}A_{\alpha}\left(\mathrm{Re}\, b(w)\right)^n\langle w\rangle^{-\rho|\alpha|}e^{-\frac{t}{2}\mathrm{Re}\, b(w)},\label{estforsek}
\eeq
for all $\alpha\in\NN^{2d}$, $n\in\NN$, $(t,w)\in[0,\infty)\times\RR^{2d}$, where $c_1\geq2$ only depends on the constant $c$ from (\ref{est_semigroup1}). Let $C_1\geq 1$ be large enough such that $2C^3_{h_1}c_1e^{2c_1}\leq C_1$. Since $A_p$ satisfies $(M.3)'$, there exists $s'\in\ZZ_+$ such that for all $s\geq s'$, $s\in\ZZ_+$, we have $C_1\leq A_s/(sA_{s-1})$. We will prove that
\beq\label{to_pr}
|D^n_tD^{\alpha}_wu_j(t,w)|\leq \frac{C_1^{\min\{s',j\}+1}n!h^{|\alpha|+2j}A_{|\alpha|+2j}\left(\mathrm{Re}\, b(w)\right)^ne^{-\frac{t}{4}\mathrm{Re}\, b(w)}}{\langle w\rangle^{\rho(|\alpha|+2j)}},
\eeq
for all $\alpha\in\NN^{2d}$, $n\in\NN$, $(t,w)\in[0,\infty)\times\RR^{2d}$. The proof is by induction on $j$. For $j=0$ this holds because of (\ref{estforsek}). Assume that it is true for all $j-1\leq s'-1$. To prove it for $j$, we first consider the case when we have $n+1$, $n\in\NN$, derivatives with respect to $t$. Then, the inequality (\ref{est_der_forjj}) holds. To estimate $S_1$, first we note that $\alpha!/(\beta!\gamma!\delta!)\leq |\alpha|!/(|\beta|!|\gamma|!|\delta|!)$ (which holds for multiindexes satisfying $\beta+\gamma+\delta=\alpha$). We use the inductive hypothesis and (\ref{lkl_117}) to deduce
\beq
S_1&\leq& \frac{C_1^j(n+1)!\left(\mathrm{Re}\,b(w)\right)^{n+1}e^{-\frac{t}{4}\mathrm{Re}\, b(w)}}{\langle w\rangle^{\rho(|\alpha|+2j)}} \sum_{l=1}^j\sum_{|\mu+\nu|=l}\sum_{\beta+\gamma+\delta=\alpha} \frac{|\alpha|!}{|\beta|!|\gamma|!|\delta|!\mu!\nu!2^l}\nonumber \\
&{}&\cdot h_1^{|\beta|+|\gamma|+l}h^{|\delta|+2j-l}A_{\beta}A_{|\gamma|+l}A_{|\delta|+2j-l} C_{h_1}^2\sum_{n_1+n_2+n_3=n}\frac{c_1^{n_1+n_2+1}}{n_1!n_2!}.\label{pre_est_s1}
\eeq
Observe that
\beq
C_{h_1}^2\sum_{n_1+n_2+n_3=n}\frac{c_1^{n_1+n_2+1}}{n_1!n_2!}&=& C_{h_1}^2c_1\sum_{n'=0}^n\frac{c_1^{n'}}{n'!}\sum_{n_1+n_2=n'}\frac{n'!}{n_1!n_2!}\nonumber \\
&\leq& C_{h_1}^2c_1e^{2c_1}\leq C_1.\label{coe_est_by_c1}
\eeq
Moreover,
\beqs
\frac{|\alpha|!}{|\beta|!|\gamma|!|\delta|!}\leq \frac{(|\alpha|+l)!}{|\beta|!(|\gamma|+l)!|\delta|!}\leq \frac{(|\alpha|+2j)!}{|\beta|!(|\gamma|+l)!(|\delta|+2j-l)!}.
\eeqs
Hence, $(M.4)$ for $A_p$ yields
\beqs
\frac{|\alpha|!}{|\beta|!|\gamma|!|\delta|!}A_{\beta}A_{|\gamma|+l}A_{|\delta|+2j-l}\leq {{|\alpha|+2j}\choose{|\beta|+|\gamma|+l}} A_{|\beta|+|\gamma|+l}A_{|\delta|+2j-l}.
\eeqs
Since $|\beta|+|\gamma|+l\geq 1$ and $|\delta|+2j-l\geq 1$, we can apply Lemma \ref{est_sequ_bi} to obtain
\beq\label{lkk_122}
\frac{|\alpha|!}{|\beta|!|\gamma|!|\delta|!}A_{\beta}A_{|\gamma|+l}A_{|\delta|+2j-l}\leq (|\alpha|+2j)A_{|\alpha|+2j-1}.
\eeq
As $pA_{p-1}\leq A_p$, $p\in\ZZ_+$, (this trivially follows from $A_0=A_1=1$ and $(M.4)$) we have
\beq
S_1&\leq& \frac{C_1^{j+1}(n+1)!h^{|\alpha|+2j}A_{|\alpha|+2j} \left(\mathrm{Re}\,b(w)\right)^{n+1}e^{-\frac{t}{4}\mathrm{Re}\, b(w)}}{\langle w\rangle^{\rho(|\alpha|+2j)}}\nonumber \\
&{}&\cdot \sum_{l=1}^j\sum_{|\mu+\nu|=l}\sum_{\beta+\gamma+\delta=\alpha} \frac{1}{\mu!\nu!2^l}\left(\frac{h_1}{h}\right)^{|\beta|+|\gamma|+l}.\label{s1_ses_forgl}
\eeq
Now, we estimate as follows,
\beqs
\sum_{l=1}^j\sum_{|\mu+\nu|=l}\sum_{\beta+\gamma+\delta=\alpha} \frac{1}{\mu!\nu!2^l}\left(\frac{h_1}{h}\right)^{|\beta|+|\gamma|+l}&\leq& \sum_{l=1}^j\left(\frac{h_1}{2h}\right)^l\sum_{|\mu+\nu|=l}\frac{1}{\mu!\nu!} \sum_{|\beta+\gamma|=0}^{|\alpha|} \left(\frac{h_1}{h}\right)^{|\beta|+|\gamma|}\\
&=&\sum_{l=1}^j\left(\frac{h_1}{2h}\right)^l\cdot\frac{(2d)^l}{l!} \sum_{s=0}^{|\alpha|} {{s+4d-1}\choose{4d-1}}\left(\frac{h_1}{h}\right)^s\\
&\leq&\sum_{l=1}^{\infty}\left(\frac{dh_1}{h}\right)^l\frac{1}{l!}\cdot \sum_{s=0}^{\infty} \left(\frac{2^{4d}h_1}{h}\right)^s\\
&\leq& 2(e^{dh_1/h}-1)\leq 1/2,
\eeqs
where the last and second to last inequality follow from the choice of $h_1$. To estimate $S_2$, we use (\ref{lkl_117}), (\ref{estforsek}) and the inductive hypothesis in order to obtain
\beq
S_2&\leq& \frac{C_{h_1}^3c_1^{n+2}C_1^j\left(\mathrm{Re}\, b(w)\right)^{n+2}e^{-\frac{t}{2}\mathrm{Re}\, b(w)}}{\langle w\rangle^{\rho(|\alpha|+2j)}}\int_0^t e^{\frac{s}{4}\mathrm{Re}\, b(w)}ds\sum_{l=1}^j\sum_{|\mu+\nu|=l} \sum_{\beta+\gamma+\delta+\kappa=\alpha}\nonumber \\
&{}&\,\,\,\,\frac{|\alpha|!}{|\beta|!|\gamma|!|\delta|!|\kappa|!\mu!\nu!2^l} h_1^{|\beta|+|\gamma|+|\delta|+l}h^{|\kappa|+2j-l}A_{\beta}A_{\gamma} A_{|\delta|+l}A_{|\kappa|+2j-l}.\label{ll_157}
\eeq
Observe that $\int_0^t e^{\frac{s}{4}\mathrm{Re}\, b(w)}ds\leq 4e^{\frac{t}{4}\mathrm{Re}\, b(w)}/\mathrm{Re}\, b(w)$. By the choice of $C_1$, we have $C_{h_1}^3c_1^{n+2}\leq C_1(n+1)!/4$. In similar fashion as above, by employing $(M.4)$ for $A_p$, we have
\beqs
\frac{|\alpha|!}{|\beta|!|\gamma|!|\delta|!|\kappa|!}A_{\beta}A_{\gamma} A_{|\delta|+l}A_{|\kappa|+2j-l}\leq {{|\alpha|+2j}\choose{|\beta|+|\gamma|+|\delta|+l}} A_{|\beta|+|\gamma|+|\delta|+l}A_{|\kappa|+2j-l}.
\eeqs
Now, since $|\beta|+|\gamma|+|\delta|+l\geq 1$ and $|\kappa|+2j-l\geq 1$, Lemma \ref{est_sequ_bi} implies
\beq\label{e_1235}
\frac{|\alpha|!}{|\beta|!|\gamma|!|\delta|!|\kappa|!}A_{\beta}A_{\gamma} A_{|\delta|+l}A_{|\kappa|+2j-l}\leq (|\alpha|+2j)A_{|\alpha|+2j-1}.
\eeq
Hence
\beq
S_2&\leq&\frac{C_1^{j+1}(n+1)!\left(\mathrm{Re}\, b(w)\right)^{n+1}h^{|\alpha|+2j}A_{|\alpha|+2j}e^{-\frac{t}{4}\mathrm{Re}\, b(w)}}{\langle w\rangle^{\rho(|\alpha|+2j)}}\nonumber \\
&{}&\,\,\,\cdot\sum_{l=1}^j\sum_{|\mu+\nu|=l} \sum_{\beta+\gamma+\delta+\kappa=\alpha}\frac{1}{\mu!\nu!2^l} \left(\frac{h_1}{h}\right)^{|\beta|+|\gamma|+|\delta|+l}.\label{e_1299}
\eeq
By analogous technique as above, keeping in mind the way we choose $h_1$, we can estimate the last sums in this inequality by $1/2$. The case when there are no derivatives with respect to $t$ can be treated in a completely analogous way as for the estimate of $S_2$ given above. This finishes the proof for $j\leq s'$. To continue the induction, assume that (\ref{to_pr}) holds for $0,\ldots,j-1$, for some $j-1\geq s'$. To prove it for $j$, we again consider first the case when there is at least one derivative with respect to $t$. Hence, our goal is to estimate $S_1$ and $S_2$ in (\ref{est_der_forjj}). For $S_1$, in the same way as above, we infer the inequality (\ref{pre_est_s1}) but with $C_1^{s'+1}$ instead of $C_1^j$. Again, we use (\ref{coe_est_by_c1}) and (\ref{lkk_122}), but now the constant $C_1$ obtained in (\ref{coe_est_by_c1}) can be swallowed: $C_1(|\alpha|+2j)A_{|\alpha|+2j-1}\leq A_{|\alpha|+2j}$. This is true because of the way we chose $s'$ and $|\alpha|+2j\geq j\geq s'$. Thus, we can conclude (\ref{s1_ses_forgl}) but with $C_1^{s'+1}$ in place of $C_1^{j+1}$. Now, we proceed in exactly the same way to estimate $S_1$ by $1/2$ times the right hand side of (\ref{to_pr}). To estimate $S_2$, we proceed as before in order to deduce (\ref{ll_157}) but with $C_1^{s'+1}$ in place of $C_1^j$. We use the inequalities (\ref{e_1235}),
\beqs
\int_0^t e^{\frac{s}{4}\mathrm{Re}\, b(w)}ds\leq \frac{4e^{\frac{t}{4}\mathrm{Re}\, b(w)}}{\mathrm{Re}\, b(w)}
\eeqs
and $C_{h_1}^3c_1^{n+2}\leq C_1(n+1)!/4$, but now notice that the constant $C_1$ obtained in the last inequality can be swallowed: $C_1(|\alpha|+2j)A_{|\alpha|+2j-1}\leq A_{|\alpha|+2j}$ (because of the way we chose $s'$). Thus, we can conclude (\ref{e_1299}) but with $C_1^{s'+1}$ in place of $C_1^{j+1}$. As previously mentioned, we can now estimate $S_2$ by $1/2$ times the right hand side of (\ref{to_pr}). The case when there are no derivatives with respect to $t$ can be treated in an analogous fashion as for the estimation of $S_2$. This concludes the proof in the Beurling case.\\
\indent For the Roumieu case, Lemma \ref{flp11} and Corollary \ref{flp1133}, together with (\ref{est_semigroup1}), imply that there exist $h_1,C_{h_1}\geq 1$ such that (\ref{lkl_117}) and (\ref{estforsek}) hold. Pick $h\geq 64^{d+1}dh_1$ and $C_1\geq 1$ large enough such that $2C^3_{h_1}c_1e^{2c_1}\leq C_1$. Let $s'\in\ZZ_+$ be such that $C_1sA_{s-1}\leq A_s$, $\forall s\geq s'$, $s\in\ZZ_+$ (such number exists because $A_p$ satisfies $(M.3)'$). One can prove (\ref{to_pr}) by induction on $j$ in the same way as for the Beurling case.
\end{proof}

Since $b(w)$ has an ultrapolynomial growth of class $*$, the same holds for $(\mathrm{Re}\, b(w))^n$, for each $n\in\NN$. Since
$\mathrm{Re}\, b(w)>0$, $\forall w\in\RR^{2d}$, as a direct consequence of Lemma \ref{est_heat_ker} we can conclude that, for each $n\in\NN$, the mappings $t\mapsto \sum_j\partial^n_tu_j(t,\cdot)$, $[0,\infty)\rightarrow FS_{A_p,\rho}^{*,\infty}(\RR^{2d};0)$, are well defined. By Taylor's formula for $\partial^n_tD^{\alpha}_wu_j(t,w)$ at $t_0$ up to order $0$, Lemma \ref{est_heat_ker} yields that for every $h>0$ there exists $C>0$ (there exist $h,C>0$) such that
\begin{multline}
|\partial^n_tD^{\alpha}_wu_j(t,w)-\partial^n_tD^{\alpha}_wu_j(t_0,w)|\\
\leq C|t-t_0|(n+1)!h^{|\alpha|+2j}A_{|\alpha|+2j}\left(\mathrm{Re}\, b(w)\right)^{n+1}\langle w\rangle^{-\rho(|\alpha|+2j)},\label{krt31}
\end{multline}
for all $\alpha\in\NN^{2d}$, $n,j\in\NN$, $w\in\RR^{2d}$, $t,t_0\in[0,\infty)$. Expanding $\partial^n_tD^{\alpha}_wu(t,w)$ at $t_0$ up to order $1$, Lemma \ref{est_heat_ker} implies that for every $h>0$ there exists $C>0$ (there exist $h,C>0$) such that
\begin{multline*}
|\partial^n_tD^{\alpha}_wu_j(t,w)-\partial^n_tD^{\alpha}_wu_j(t_0,w)- (t-t_0)\partial^{n+1}_tD^{\alpha}_wu_j(t_0,w)|\\
\leq C|t-t_0|^2(n+2)!h^{|\alpha|+2j}A_{|\alpha|+2j}\left(\mathrm{Re}\, b(w)\right)^{n+2}\langle w\rangle^{-\rho(|\alpha|+2j)},
\end{multline*}
for all $\alpha\in\NN^{2d}$, $n,j\in\NN$, $w\in\RR^{2d}$, $t,t_0\in[0,\infty)$. Hence, we infer the following result.

\begin{lemma}
The mapping
\beqs
t\mapsto \ssum u_j(t,\cdot),\,\,\, [0,\infty)\rightarrow FS_{A_p,\rho}^{*,\infty}(\RR^{2d};0),
\eeqs
is in $\mathcal{C}^{\infty}([0,\infty); FS_{A_p,\rho}^{*,\infty}(\RR^{2d};0))$ and $\partial^n_t(\sum_j u_j(t,\cdot))=\sum_j\partial^n_t u_j(t,\cdot)$, $n\in\NN$.
\end{lemma}

Clearly, for $R>0$, the function $u(t,w)=\sum_n (1-\chi_{n,R}(w))u_n(t,w)=R(\sum_j u_j)(t,w)$ is in $\mathcal{C}^{\infty}(\RR\times\RR^{2d})$. Using Lemma \ref{est_heat_ker}, by exactly the same technique as in the proof of \cite[Theorem 4]{BojanP}, one can prove the following result.

\begin{lemma}\label{rks75}
There exists $R>1$ such that the $\mathcal{C}^{\infty}$-function
\beqs
u(t,w)=\sum_{n=0}^{\infty} (1-\chi_{n,R}(w))u_n(t,w)=R(\ssum u_j)(t,w)
\eeqs
satisfies the following condition: for every $h>0$ there exists $C>0$ (resp. there exist $h,C>0$) such that
\beq\label{est_heat_par}
|D^n_tD^{\alpha}_wu(t,w)|\leq Cn!h^{|\alpha|}A_{\alpha}\left(\mathrm{Re}\, b(w)\right)^n\langle w\rangle^{-\rho|\alpha|}e^{-\frac{t}{4}\mathrm{Re}\, b(w)},
\eeq
for all $\alpha\in\NN^{2d}$, $n\in\NN$, $(t,w)\in[0,\infty)\times\RR^{2d}$ and
\beqs
\sup_{N\in\ZZ_+}\sup_{\substack{\alpha\in\NN^{2d}\\ n\in\NN}}\sup_{\substack{w\in Q^c_{3Rm_N}\\ t\in [0,\infty)}}\frac{\left|D^n_tD^{\alpha}_w\left(u(t,w)-\sum_{j<N} u_j(t,w)\right)\right|\langle w\rangle^{\rho(|\alpha|+2N)}}{n!h^{|\alpha|+2N}A_{|\alpha|+2N}\left(\mathrm{Re}\, b(w)\right)^ne^{-\frac{t}{4}\mathrm{Re}\, b(w)}}\leq C.
\eeqs
\end{lemma}

As $(\mathrm{Re}\, b(w))^n$ has ultrapolynomial growth of class
$*$ for each $n\in\NN$, and $\mathrm{Re}\, b(w)>0$, $\forall
w\in\RR^{2d}$, as a direct consequence of this lemma we can
conclude that the mappings
\beqs
\mathbf{u}_n:[0,\infty)\rightarrow\Gamma_{A_p,\rho}^{*,\infty}(\RR^{2d}),\,\, \mathbf{u}_n(t)=
(\partial^n_tu)(t,\cdot),
\eeqs
are well defined for every $n\in\NN$. When $n=0$, we simply denote it by $\mathbf{u}$ (instead of
$\mathbf{u}_0$). Our goal is to prove
$\mathbf{u}\in\mathcal{C}^{\infty}([0,\infty);\Gamma_{A_p,\rho}^{*,\infty}(\RR^{2d}))$
and $\partial^n_t\mathbf{u}=\mathbf{u}_n$, $n\in\ZZ_+$. By Taylor
formula for $\partial^n_tD^{\alpha}_wu(t,w)$ at $t_0$ up to order
$0$ and employing (\ref{est_heat_par}) we see that for every
$h>0$ there exists $C>0$ (there exist $h,C>0$) such that
\begin{multline}
|\partial^n_tD^{\alpha}_wu(t,w)-\partial^n_tD^{\alpha}_wu(t_0,w)|\\
\leq C|t-t_0|(n+1)!h^{|\alpha|}A_{\alpha}\left(\mathrm{Re}\,
b(w)\right)^{n+1}\langle w\rangle^{-\rho|\alpha|},\label{kfh779911}
\end{multline}
for all $\alpha\in\NN^{2d}$, $n\in\NN$, $w\in\RR^{2d}$, $t,t_0\in[0,\infty)$. Hence $\mathbf{u},\mathbf{u}_n\in\mathcal{C}([0,\infty);\Gamma_{A_p,\rho}^{*,\infty}(\RR^{2d}))$, $n\in\ZZ_+$. Expanding $\partial^n_tD^{\alpha}_wu(t,w)$ at $t_0$ up to order $1$ and using (\ref{est_heat_par}), we infer that for every $h>0$ there exists $C>0$ (there exist $h,C>0$) such that
\begin{multline*}
|\partial^n_tD^{\alpha}_wu(t,w)-\partial^n_tD^{\alpha}_wu(t_0,w)- (t-t_0)\partial^{n+1}_tD^{\alpha}_wu(t_0,w)|\\
\leq C|t-t_0|^2(n+2)!h^{|\alpha|}A_{\alpha}\left(\mathrm{Re}\, b(w)\right)^{n+2}\langle w\rangle^{-\rho|\alpha|},
\end{multline*}
for all $\alpha\in\NN^{2d}$, $n\in\NN$, $w\in\RR^{2d}$, $t,t_0\in[0,\infty)$. Hence $\mathbf{u}$ and $\mathbf{u}_n$, $n\in\ZZ_+$, are differentiable $\Gamma_{A_p,\rho}^{*,\infty}(\RR^{2d})$-valued functions and $\partial_t\mathbf{u}=\mathbf{u}_1$ and $\partial_t\mathbf{u}_n=\mathbf{u}_{n+1}$, $n\in\ZZ_+$. Thus, we can conclude $\mathbf{u}\in\mathcal{C}^{\infty}([0,\infty);\Gamma_{A_p,\rho}^{*,\infty}(\RR^{2d}))$ and $\partial^n_t\mathbf{u}=\mathbf{u}_n$, $\forall n\in\ZZ_+$. By Proposition \ref{continuity}, the mappings $c\mapsto c^w$, $\Gamma_{A_p,\rho}^{*,\infty}(\RR^{2d})\rightarrow \mathcal{L}_b(\SSS^*(\RR^d),\SSS^*(\RR^d))$ and $c\mapsto c^w$, $\Gamma_{A_p,\rho}^{*,\infty}(\RR^{2d})\rightarrow \mathcal{L}_b(\SSS'^*(\RR^d),\SSS'^*(\RR^d))$ are continuous. Hence the mapping $t\mapsto (\mathbf{u}(t))^w$ is in\\ $\mathcal{C}^{\infty}([0,\infty);\mathcal{L}_b(\SSS^*(\RR^d),\SSS^*(\RR^d)))$ and in $\mathcal{C}^{\infty}([0,\infty);\mathcal{L}_b(\SSS'^*(\RR^d),\SSS'^*(\RR^d)))$. Thus, the mapping
\beqs
\mathbf{K}:t\mapsto (\partial_t+b^w)(\mathbf{u}(t))^w=(\partial_t\mathbf{u}(t))^w+b^w(\mathbf{u}(t))^w
\eeqs
is in $\mathcal{C}^{\infty}([0,\infty);\mathcal{L}_b(\SSS^*(\RR^d),\SSS^*(\RR^d)))$ and in $\mathcal{C}^{\infty}([0,\infty);\mathcal{L}_b(\SSS'^*(\RR^d),\SSS'^*(\RR^d)))$.\\
\indent For $n\in\NN$, $t\in[0,\infty)$, we denote $\tilde{\mathbf{u}}_n(t)=\sum_j\partial^n_tu_j(t,\cdot)\in FS_{A_p,\rho}^{*,\infty}(\RR^{2d};0)$. Fix $n\in\ZZ_+$. For the moment, for each $0\leq k\leq n$, we denote
\beqs
\mathbf{K}_k(t)=\partial_t^k\mathbf{K}(t)=(\partial_t^{k+1}\mathbf{u}(t))^w+ b^w(\partial^k_t\mathbf{u}(t))^w
\eeqs
(clearly $\mathbf{K}_0=\mathbf{K}$). Lemmas \ref{est_heat_ker} and \ref{rks75} imply that we can apply Theorem \ref{weylq} $ii)$ with $V_1=\{b\}$, $U_1=\{b+\sum_{j\in\ZZ_+} 0\}$, $V_2=\{\partial^k_tu(t,\cdot)|\, t\in[0,\infty),\, 0\leq k\leq n\}$, $U_2=\{\tilde{\mathbf{u}}_k(t)|\, t\in[0,\infty),\, 0\leq k\leq n\}$ and $f_1(w)=\mathrm{Re}\, b(w)$, $f_2(w)=\max_{0\leq k\leq n}(\mathrm{Re}\, b(w))^k$; where $\Sigma_2(\sum_j\partial^k_tu_j(t,\cdot))=\Sigma_2(\tilde{\mathbf{u}}_k(t)) =\partial^k_tu(t,\cdot)$. Hence, there exists $R_1>0$ such that $\{b^w(\partial^k_t\mathbf{u}(t))^w-\Op_{1/2}(R_1(b\#\tilde{\mathbf{u}}_k(t)))|\, t\in[0,\infty),\, 0\leq k\leq n\}$ is an equicontinuous $*$-regularising set and (\ref{krh1791}) holds. By (\ref{system1}), for all $N\in\ZZ_+$, we have
\begin{multline*}
\partial^{k+1}_tu(t,\cdot)+R_1(b\#\tilde{\mathbf{u}}_k(t))\\
= \partial^{k+1}_tu(t,\cdot)-\sum_{j<N}\partial^{k+1}_tu_j(t,\cdot) -(b\#\ssum\partial^k_tu_j(t,\cdot))_{<N} +R_1(b\#\ssum\partial^k_tu_j(t,\cdot)).
\end{multline*}
Hence, Lemma \ref{rks75} and Proposition \ref{eqsse} imply that
\beq\label{tkr991155}
\{(\partial^{k+1}_t\mathbf{u}(t))^w+\Op_{1/2}(R_1(b\#\tilde{\mathbf{u}}_k(t)))|\, t\in[0,\infty),\, 0\leq k\leq n\}
\eeq
is equicontinuous $*$-regularising set and thus, the same holds for $K_n=\{\mathbf{K}_k(t)|\, t\in[0,\infty),\, 0\leq k\leq n\}$ as well. Let $\mathfrak{G}$ be the family of all finite subsets of $\SSS^*(\RR^d)$. The union of $\mathfrak{G}$ is total in $\SSS'^*(\RR^d)$. The equicontinuity of $K_n$ together with the Banach-Steinhaus theorem \cite[Theorem 4.5, p. 85]{Sch} implies that the topology induced on $K_n$ by $\mathcal{L}_{\mathfrak{G}}(\SSS'^*(\RR^d),\SSS^*(\RR^d))$ is the same as the topology induced on it by $\mathcal{L}_p(\SSS'^*(\RR^d),\SSS^*(\RR^d))$. Since $\SSS'^*(\RR^d)$ is Montel, the latter space is in fact $\mathcal{L}_b(\SSS'^*(\RR^d),\SSS^*(\RR^d))$. Because $\mathbf{K}\in\mathcal{C}^{\infty}([0,\infty);\mathcal{L}_b(\SSS^*(\RR^d),\SSS^*(\RR^d)))$, we conclude
\beqs
\mathbf{K}_k\in\mathcal{C}([0,\infty);\mathcal{L}_b(\SSS'^*(\RR^d),\SSS^*(\RR^d))),\,\,\, 0\leq k\leq n.
\eeqs
Fix $t_0\in[0,\infty)$. We can apply Theorem \ref{weylq} $ii)$ with $V_1=\{b\}$, $U_1=\{b+\sum_{j\in\ZZ_+} 0\}$,
\beqs
V_2&=&\{(t-t_0)^{-1}(\partial^k_tu(t,\cdot)-\partial^k_tu(t_0,\cdot))|\, t\in[0,\infty)\backslash\{t_0\},\, 0\leq k\leq n-1\},\\
U_2&=&\{(t-t_0)^{-1}(\tilde{\mathbf{u}}_k(t)-\tilde{\mathbf{u}}_k(t_0))|\, t\in[0,\infty)\backslash\{t_0\},\, 0\leq k\leq n-1\}
\eeqs
and $f_1(w)=\mathrm{Re}\, b(w)$, $f_2(w)=\max_{0\leq k\leq n}(\mathrm{Re}\, b(w))^k$. To see this, notice that (\ref{krt31}) proves $U_2\precsim f_2$ and, by Taylor expanding $\partial^k_tD^{\alpha}_wu(t,w)-\sum_{j<N} \partial^k_tD^{\alpha}_wu_j(t,w)$ at $t_0$ up to order $1$ and employing Lemma \ref{rks75}, we conclude $V_2\precsim_{f_2} U_2$ (the boundness of $V_2$ in some $\Gamma^{(M_p),\infty}_{A_p,\rho}(\RR^{2d};m)$, resp. in some $\Gamma^{\{M_p\},\infty}_{A_p,\rho}(\RR^{2d};h)$, follows from (\ref{kfh779911})). We deduce the existence of $R_2>0$ such that the set
\begin{multline*}
\big\{(t-t_0)^{-1}\left(b^w(\partial^k_t\mathbf{u}(t))^w-b^w(\partial^k_t\mathbf{u}(t_0))^w\right)\\
-(t-t_0)^{-1}\Op_{1/2}\left(R_2\left(b\#\left(\tilde{\mathbf{u}}_k(t)- \tilde{\mathbf{u}}_k(t_0)\right)\right)\right) \big|\, t\in[0,\infty)\backslash\{t_0\},\, 0\leq k\leq n-1\big\}
\end{multline*}
is equicontinuous $*$-regularising and (\ref{krh1791}) holds. Clearly, if we change the definition of $V_2$, $U_2$ and $f_2$ such that $k$ ranges up to $n$ (instead of up to $n-1$), we still have $V_2\precsim_{f_2} U_2$. This fact together with application of (\ref{system1}) in the same manner as for the equicontinuity of the set (\ref{tkr991155}) implies that
\begin{multline*}
\big\{(t-t_0)^{-1}(\partial^{k+1}_t\mathbf{u}(t)-\partial^{k+1}_t\mathbf{u}(t_0))^w\\
+(t-t_0)^{-1}\Op_{1/2}\left(R_2\left(b\# \left(\tilde{\mathbf{u}}_k(t)-\tilde{\mathbf{u}}_k(t_0)\right)\right)\right)\big|\, t\in[0,\infty)\backslash\{t_0\},\, 0\leq k\leq n-1\big\}
\end{multline*}
is equicontinuous $*$-regularising and hence so is the set
\beqs
\{(t-t_0)^{-1}(\mathbf{K}_k(t)-\mathbf{K}_k(t_0))|\, t\in[0,\infty)\backslash\{t_0\},\, 0\leq k\leq n-1\}.
\eeqs
Since $\mathbf{K}\in\mathcal{C}^{\infty}([0,\infty);\mathcal{L}_b(\SSS^*(\RR^d),\SSS^*(\RR^d)))$, by the same technique as above (applying again the Banach-Steinhaus theorem) we conclude that $\mathbf{K}_k$ is differentiable at $t_0$ as an $\mathcal{L}_b(\SSS'^*(\RR^d),\SSS^*(\RR^d))$-valued function and its derivative is $\mathbf{K}_{k+1}(t_0)$, $0\leq k\leq n-1$. Since $t_0\in[0,\infty)$ and $n\in\ZZ_+$ were arbitrary, we conclude $\mathbf{K}\in\mathcal{C}^{\infty}([0,\infty);\mathcal{L}_b(\SSS'^*(\RR^d),\SSS^*(\RR^d)))$ and $\partial^n_t\mathbf{K}(t)=\mathbf{K}_n(t)$, $n\in\NN$. Lastly, because of the initial conditions in (\ref{system1}) and the way we constructed $u(t,w)$ we have $\mathbf{u}(0)=1$ and hence $(\mathbf{u}(0))^w=\mathrm{Id}$. We proved the first part of the following result.

\begin{theorem}\label{sol_heat_ker}
The $\mathcal{C}^{\infty}$ function $u(t,w)$ constructed in Lemma \ref{rks75} is such that the mapping $\mathbf{u}:t\mapsto u(t,\cdot)$, $[0,\infty)\rightarrow \Gamma_{A_p,\rho}^{*,\infty}(\RR^{2d})$, belongs to $\mathcal{C}^{\infty}([0,\infty); \Gamma_{A_p,\rho}^{*,\infty}(\RR^{2d}))$. The mapping $t\mapsto (\mathbf{u}(t))^w$ is in both
\beqs
\mathcal{C}^{\infty}([0,\infty);\mathcal{L}_b(\SSS^*(\RR^d),\SSS^*(\RR^d)))\,\,\, \mbox{and}\,\,\, \mathcal{C}^{\infty}([0,\infty);\mathcal{L}_b(\SSS'^*(\RR^d),\SSS'^*(\RR^d))).
\eeqs
Moreover, $(\mathbf{u}(t))^w$ satisfy
\beq\label{system2}
\left\{\begin{array}{l}
(\partial_t+b^w)(\mathbf{u}(t))^w=\mathbf{K}(t),\,t\in[0,\infty),\\
(\mathbf{u}(0))^w=\mathrm{Id},\\
\end{array}\right.
\eeq
where $\mathbf{K}\in \mathcal{C}^{\infty}([0,\infty);\mathcal{L}_b(\SSS'^*(\RR^d),\SSS^*(\RR^d)))$.\\
\indent For each $t\geq 0$, $(\mathbf{u}(t))^w\in\mathcal{L}(L^2(\RR^d))$ and there exists $C>0$ such that
\beqs
\|(\mathbf{u}(t))^w\|_{\mathcal{L}_b(L^2(\RR^d))}\leq C,\,\,\, \mbox{for all}\,\, t\geq 0.
\eeqs
The mapping $t\mapsto (\mathbf{u}(t))^w$, $(0,\infty)\rightarrow \mathcal{L}_b(L^2(\RR^d))$, is continuous and
\beqs
(\mathbf{u}(t))^w\rightarrow (\mathbf{u}(0))^w=\mathrm{Id},\,\,\, \mbox{as}\,\, t\rightarrow0^+,\,\, \mbox{in}\,\, \mathcal{L}_p(L^2(\RR^d)).
\eeqs
\indent Furthermore, for each $n\in\ZZ_+$, $(\partial^n_t\mathbf{u}(t))^w\in\mathcal{L}(L^2(\RR^d))$, for all $t>0$. The mapping $t\mapsto (\mathbf{u}(t))^w$, $(0,\infty)\rightarrow \mathcal{L}_b(L^2(\RR^d))$, is smooth and $\partial^n_t(\mathbf{u}(t))^w=(\partial^n_t\mathbf{u}(t))^w$.
\end{theorem}

\begin{proof} The fact $(\mathbf{u}(t))^w\in\mathcal{L}(L^2(\RR^d))$, $t\geq 0$, and the existence of $C>0$ such that $\|(\mathbf{u}(t))^w\|_{\mathcal{L}_b(L^2(\RR^d))}\leq C$, for all $t\geq 0$, follows from the estimate (\ref{est_heat_par}) for $n=0$ and Lemma \ref{operatoronl2}. Since $(\mathbf{u}(t))^w\varphi\rightarrow (\mathbf{u}(0))^w\varphi$ for each $\varphi\in\SSS^*(\RR^d)$ and $\{(\mathbf{u}(t))^w|\, t\geq 0\}$ is bounded in $\mathcal{L}_b(L^2(\RR^d))$, the Banach-Steinhaus theorem implies $(\mathbf{u}(t))^w\rightarrow (\mathbf{u}(0))^w=\mathrm{Id}$, as $t\rightarrow0^+$, in $\mathcal{L}_p(L^2(\RR^d))$. As a direct consequence of (\ref{est_heat_par}), we deduce that for every $h>0$ there exists $C>0$ (resp. there exist $h,C>0$) such that
\beq\label{kr_estt17}
t^n|D^n_tD^{\alpha}_wu(t,w)|\leq C4^n n!^2h^{|\alpha|}A_{\alpha}\langle w\rangle^{-\rho|\alpha|},
\eeq
for all $\alpha\in\NN^{2d}$, $n,k\in\NN$, $(t,w)\in[0,\infty)\times\RR^{2d}$. Thus, for fixed $n\in\ZZ_+$, Lemma \ref{operatoronl2} implies $(\partial^n_t\mathbf{u}(t))^w\in\mathcal{L}(L^2(\RR^d))$, $t>0$. Again, fix $n\in\NN$ and Taylor expand $\partial^n_tD^{\alpha}_wu(t,w)$ at $t_0>0$ up to order $0$. By using (\ref{kr_estt17}) with $n+1$ to estimate the reminder and employing Lemma \ref{operatoronl2}, one obtains that $t\mapsto (\partial^n_t\mathbf{u}(t))^w$, $(0,\infty)\rightarrow \mathcal{L}_b(L^2(\RR^d))$, is continuous at $t_0$. As $t_0$ is arbitrary, it is continuous on $(0,\infty)$. Fix $t_0\in(0,\infty)$ and $n\in\NN$. By Taylor expanding $\partial^n_tD^{\alpha}_wu(t,w)$ at $t_0$ up to order $1$ and employing (\ref{kr_estt17}) with $n+2$ to estimate the reminder, Lemma \ref{operatoronl2} implies that $t\mapsto (\mathbf{u}(t))^w$, $(0,\infty)\rightarrow \mathcal{L}_b(L^2(\RR^d))$, is differentiable at $t_0$ and the derivative is $(\partial^{n+1}_t\mathbf{u}(t_0))^w$. Thus $(\mathbf{u}(t))^w\in\mathcal{C}^{\infty}((0,\infty); \mathcal{L}_b(L^2(\RR^d)))$ and $\partial^n_t(\mathbf{u}(t))^w=(\partial^n_t\mathbf{u}(t))^w$.
\end{proof}

\subsection{Semigroup generated by the square root of a nonnegative hypoelliptic operator}\label{sec_smgr_portk}

Let $a\in\Gamma_{A_p,\rho}^{*,\infty}(\RR^{2d})$ be a hypoelliptic symbol which satisfies the assumptions in Theorem \ref{maint} and let $a_0\in\Gamma_{A_p,\rho}^{*,\infty}(\RR^{2d})$ be the symbol defined in the statement of the same theorem. By applying Theorem \ref{maint} with $z=1/2$, we concluded the existence of a hypoelliptic symbol $a^{\uwidehat{1/2}}\in \Gamma_{A_p,\rho}^{*,\infty}(\RR^{2d})$ and a $*$-regularising operator $S_1$ such that
\beqs
\overline{A}^{1/2}=\overline{\big(a^{\uwidehat{1/2}}\big)^w|_{\SSS^*(\RR^d)}}+S_1\,\,\, \mbox{with}\,\,\, D(\overline{A}^{1/2})=\big\{v\in L^2(\RR^d)\big|\, \big(a^{\uwidehat{1/2}}\big)^w v\in L^2(\RR^d)\big\}
\eeqs
and the estimates in part $(v)$ of Theorem \ref{maint} hold true with $k=1$. By particularising the first estimate in $(v)$ for $N=1$ and $\alpha=0$, one easily verifies $\mathrm{Re}\, a^{\uwidehat{1/2}}(w)>c' \big|\mathrm{Im}\, a^{\uwidehat{1/2}}(w)\big|$, $\forall w\in Q^c_B$, for some $B,c'>0$ (cf. (\ref{finqa})). Clearly, we can assume that $a^{\uwidehat{1/2}}$ satisfies (\ref{dd1}) and (\ref{dd2}) for this $B$. Take $\tilde{\tilde{\chi}}\in \DD^{(A_p)}(\RR^{2d})$ (resp. $\tilde{\tilde{\chi}}\in\DD^{\{A_p\}}(\RR^{2d})$) such that $0\leq \tilde{\tilde{\chi}}\leq 1$, $\tilde{\tilde{\chi}}=1$ on a small neighbourhood of $Q_B$ and $\tilde{\tilde{\chi}}=0$ on the complement of a slightly larger neighbourhood and define $b=\tilde{\tilde{\chi}}+(1-\tilde{\tilde{\chi}})a^{\uwidehat{1/2}}$. Then $b$ is a hypoelliptic symbol in $\Gamma_{A_p,\rho}^{*,\infty}(\RR^{2d})$ that satisfies (\ref{dd1}), (\ref{dd2}) and (\ref{est_semigroup1}) on the whole $\RR^{2d}$. Furthermore, for every $h>0$ there exists $C>0$ (resp. there exist $h,C>0$) such that
\beq\label{est_semigroup}
\left|D^{\alpha}_w b(w)\right|\leq Ch^{|\alpha|}A_{\alpha}|a_0(w)|^{1/2} \langle w\rangle^{-\rho|\alpha|},\,\, w\in\RR^{2d},\,\alpha\in\NN^{2d}.
\eeq
Moreover, $\overline{A}^{1/2}=\overline{b^w|_{\SSS^*(\RR^d)}}+S$ with $S$ a $*$-regularising operator and
\beqs
D(\overline{A}^{1/2})=\left\{v\in L^2(\RR^d)|\, b^w v\in L^2(\RR^d)\right\}.
\eeqs
Thus, we can apply Theorem \ref{sol_heat_ker} to $b$ to deduce (\ref{system2}). As $t\mapsto S(\mathbf{u}(t))^w$ belongs to $\mathcal{C}^{\infty}([0,\infty);\mathcal{L}_b(\SSS'^*(\RR^d),\SSS^*(\RR^d)))$ (since $\mathbf{u}\in\mathcal{C}^{\infty}([0,\infty); \mathcal{L}_b(\SSS'^*(\RR^d),\SSS'^*(\RR^d)))$) we have
\beq\label{system3}
\left\{\begin{array}{l}
(\partial_t+b^w+S)(\mathbf{u}(t))^w=\tilde{\mathbf{K}}(t),\,t\in[0,\infty),\\
(\mathbf{u}(0))^w=\mathrm{Id},\\
\end{array}\right.
\eeq
for some $\tilde{\mathbf{K}}\in\mathcal{C}^{\infty}([0,\infty); \mathcal{L}_b(\SSS'^*(\RR^d),\SSS^*(\RR^d)))$.\\
\indent On the other hand, since $\overline{A}$ is non-negative and densely defined, \cite[Theorem 5.5.2, p. 131]{Fractionalpowersbook} (cf. \cite[Theorem 5.4.1, p. 123; Theorem A.7.6, p. 329]{Fractionalpowersbook}) yields that $-\overline{A}^{1/2}$ is the infinitesimal generator of an analytic semigroup $T(t)$ of amplitude less than $\pi/2$ and
\beq\label{kks93}
T(t)=\frac{2}{\pi}\lim_{s\rightarrow \infty}\int_0^s \lambda\sin(t\lambda)(\overline{A}+\lambda^2\mathrm{Id})^{-1}d\lambda,\,\, t>0,
\eeq
where the limit exists in $\mathcal{L}_b(L^2(\RR^d))$.\\
\indent The main result is the following theorem.

\begin{theorem}\label{res_sqr_rtt1}
Let $a$ be a hypoelliptic symbol in $\Gamma_{A_p,\rho}^{*,\infty}(\RR^{2d})$ that satisfies the assumptions of Theorem \ref{maint} and let $T(t)$, $t\geq 0$, be the analytic semigroup generated by $-\overline{A}^{1/2}$. There exists $u\in\mathcal{C}^{\infty}(\RR\times\RR^{2d})$ such that the mapping $t\mapsto \mathbf{u}(t)=u(t,\cdot)$ belongs to $\mathcal{C}^{\infty}([0,\infty); \Gamma_{A_p,\rho}^{*,\infty}(\RR^{2d}))$ and $T(t)=(\mathbf{u}(t))^w+\mathbf{Q}(t)$, where the mapping $t\mapsto\mathbf{Q}(t)$ belongs to $\mathcal{C}^{\infty}([0,\infty);\mathcal{L}_b(\SSS'^*(\RR^d),\SSS^*(\RR^d)))$. Moreover, the function $u$ satisfies the following estimate: there exists $0<c_1<1$ such that for every $h>0$ there exists $C>0$ (resp. there exist $h,C>0$) such that
\beq\label{est1453}
|D^n_tD^{\alpha}_wu(t,w)|\leq Cn!h^{|\alpha|}A_{\alpha}|a_0(w)|^{n/2}\langle w\rangle^{-\rho|\alpha|}e^{-c_1t |a_0(w)|^{1/2}},
\eeq
for all $\alpha\in\NN^{2d}$, $n\in\NN$, $(t,w)\in[0,\infty)\times\RR^{2d}$, where $a_0\in\Gamma_{A_p,\rho}^{*,\infty}(\RR^{2d})$ is the symbol defined in the statement of Theorem \ref{maint}.\\
\indent Furthermore, $(\mathbf{u}(0))^w=\mathrm{Id}$, $(\mathbf{u}(t))^w\in\mathcal{L}(L^2(\RR^d))$ for every $t\geq 0$, and there exists $C>0$ such that $\|(\mathbf{u}(t))^w\|_{\mathcal{L}_b(L^2(\RR^d))}\leq C$, for all $t\geq 0$. The mapping $t\mapsto (\mathbf{u}(t))^w$, $\RR_+\rightarrow \mathcal{L}_b(L^2(\RR^d))$, is continuous and $(\mathbf{u}(t))^w\rightarrow \mathrm{Id}$, as $t\rightarrow0^+$, in $\mathcal{L}_p(L^2(\RR^d))$.\\
\indent For each $n\in\ZZ_+$, $(\partial^n_t\mathbf{u}(t))^w\in\mathcal{L}(L^2(\RR^d))$, $\forall t>0$, and the mapping $t\mapsto (\mathbf{u}(t))^w$, belongs to  $\mathcal{C}^{\infty}(\RR_+;\mathcal{L}_b(L^2(\RR^d)))$, with $\partial^n_t(\mathbf{u}(t))^w=(\partial^n_t\mathbf{u}(t))^w$, $n\in\ZZ_+$.
\end{theorem}

\subsection{Proof of Theorem \ref{res_sqr_rtt1}}

This subsection is devoted to the proof of Theorem \ref{res_sqr_rtt1}. In fact, we will show that the function $u\in\mathcal{C}^{\infty}(\RR\times\RR^{2d})$ is exactly the one obtained when we apply Theorem \ref{sol_heat_ker} to the symbol $b$ given above.\\
\indent We start by proving a useful technical result.

\begin{lemma}\label{ktn1793}
Let $F$ be a Montel space, $t_0\in\RR$ and $\varepsilon>0$.
\begin{itemize}
\item[$(i)$] Assume $\mathbf{F}\in\mathcal{C}([t_0,t_0+\varepsilon];F)$ (resp. $\mathbf{F}\in\mathcal{C}([t_0-\varepsilon,t_0];F)$, resp. $\mathbf{F}\in\mathcal{C}([t_0-\varepsilon,t_0+\varepsilon];F)$) and for each $f'\in F'$, the function $t\mapsto F_{f'}(t)=\langle f',\mathbf{F}(t)\rangle$, has a right derivative at $t_0$ (resp. left derivative at $t_0$, resp. derivative at $t_0$). Then the same also holds for $\mathbf{F}$ and its right derivative (resp. left derivative, resp. derivative) is the element in $F=(F'_b)'_b$ given by $f'\mapsto (d^+/dt)F_{f'}(t_0)$ (resp. by $f'\mapsto (d^-/dt)F_{f'}(t_0)$, resp. by $f'\mapsto (d/dt)F_{f'}(t_0)$).
\item[$(ii)$] Let $E$ be a Montel space. Assume that $\mathbf{Q}\in\mathcal{C}([t_0,t_0+\varepsilon];\mathcal{L}_b(E,F))$ (resp. $\mathbf{Q}\in\mathcal{C}([t_0-\varepsilon,t_0];\mathcal{L}_b(E,F))$, resp. $\mathbf{Q}\in\mathcal{C}([t_0-\varepsilon,t_0+\varepsilon];\mathcal{L}_b(E,F))$) and for each $e\in E$ and $f'\in F'$, the function $t\mapsto F_{e,f'}(t)=\langle f',\mathbf{Q}(t)e\rangle$, has a right derivative at $t_0$ (resp. left derivative at $t_0$, resp. derivative at $t_0$). Then the same also holds for the mapping $\mathbf{Q}$ and its right derivative (resp. left derivative, resp. derivative) is the element in $\mathcal{L}(E,F)=\mathcal{L}(E,(F'_b)'_b)$ given by $e\mapsto (f'\mapsto (d^+/dt)F_{e,f'}(t_0))$ (resp. by $e\mapsto (f'\mapsto (d^-/dt)F_{e,f'}(t_0))$, resp. by $e\mapsto (f'\mapsto (d/dt)F_{e,f'}(t_0))$).
\end{itemize}
\end{lemma}

\begin{proof} $(i)$ We prove the result only for the right derivative, the other two cases are analogous. Since, for each $f'\in F'$, the set $\{(t-t_0)^{-1}(F_{f'}(t)-F_{f'}(t_0))|\, t\in(t_0,t_0+\varepsilon]\}$ is bounded in $\CC$, the set $H=\{(t-t_0)^{-1}(\mathbf{F}(t)-\mathbf{F}(t_0))|\, t\in(t_0,t_0+\varepsilon]\}$ is weakly bounded in $F$, hence bounded in $F$. As $F$ is reflexive, $H$ is equicontinuous in $\mathcal{L}(F'_b,\CC)$. By \cite[Theorem 4.3, p. 84]{Sch}, its closure $H_1$ in $\CC^{F'_b}$ for the topology of simple convergence is in fact an equicontinuous subset of $\mathcal{L}(F'_b,\CC)$. This shows that the mapping $f'\mapsto (d^+/dt)F_{f'}(t_0)$ (which is in $H_1$) is an element of $\mathcal{L}(F'_b,\CC)$, i.e. it is a well defined element of $F$ which we denote by $(d^+/dt)\mathbf{F}(t_0)$. Moreover, by the Banach-Steinhaus theorem \cite[Theorem 4.5, p. 85]{Sch}, we have $(t-t_0)^{-1}(\mathbf{F}(t)-\mathbf{F}(t_0))\rightarrow (d^+/dt)\mathbf{F}(t_0)$ in $\mathcal{L}_p(F'_b,\CC)$. As $F$ is Montel, so is $F'_b$. Hence the convergence holds in $\mathcal{L}_b(F'_b,\CC)=F$.\\
\indent $(ii)$ Again, we give the prove only for the right derivative, the other two cases are analogous. By $(i)$, for each $e\in E$, the mapping $t\mapsto\mathbf{Q}(t)e$, $[t_0,t_0+\varepsilon]\rightarrow F$, has right derivative at $t_0$ which is the element of $\mathcal{L}_b(F'_b,\CC)=F$ given by $f'\mapsto (d^+/dt)F_{e,f'}(t_0)$. This implies that the set $H=\{e\mapsto (t-t_0)^{-1}(\mathbf{Q}(t)e-\mathbf{Q}(t_0)e)|\, t\in(t_0,t_0+\varepsilon]\}$ is weakly bounded in $\mathcal{L}(E,F)$ (for each $e\in E$, the set $\{\mathbf{Q}(t)e|\, t\in[t_0,t_0+\varepsilon]\}$ is bounded in $F$ by the continuity of $\mathbf{Q}$). Thus, $H$ is equicontinuous and if we take its closure $H_1$ in $F^E$ for the topology of simple convergence, \cite[Theorem 4.3, p. 84]{Sch} implies that $H_1$ is an equicontinuous subset of $\mathcal{L}(E,F)$. Hence, the mapping $e\mapsto (f'\mapsto (d^+/dt)F_{e,f'}(t_0))$ is well defined element of $\mathcal{L}(E,(F'_b)'_b)=\mathcal{L}(E,F)$ which we denote by $(d^+/dt)\mathbf{Q}(t_0)$. Moreover, the Banach-Steinhaus theorem \cite[Theorem 4.5, p. 85]{Sch} yields that
\beqs
(t-t_0)^{-1}(\mathbf{Q}(t)-\mathbf{Q}(t_0))\rightarrow (d^+/dt)\mathbf{Q}(t_0)\,\,\, \mbox{in}\,\, \mathcal{L}_p(E,F).
\eeqs
Since $E$ is Montel, the convergence also holds in $\mathcal{L}_b(E,F)$.
\end{proof}

Our immediate goal is to prove the smoothness of $t\mapsto T(t)$ as an $\mathcal{L}_b(\SSS^*(\RR^d),\SSS^*(\RR^d))$-valued mapping.

\begin{lemma}\label{lmmsk11}
For each $t\geq0$, $T(t)\in\mathcal{L}(\SSS^*(\RR^d),\SSS^*(\RR^d))$. Moreover, the mapping $t\mapsto T(t)$ belongs to $\mathcal{C}^{\infty}([0,\infty);\mathcal{L}_b(\SSS^*(\RR^d),\SSS^*(\RR^d)))$.
\end{lemma}

\begin{proof} Observe that
\beqs
\int_0^s\lambda\sin(t\lambda)(\overline{A}+\lambda^2\mathrm{Id})^{-1}d\lambda =\mathrm{Id}\int_0^s \frac{\sin(t\lambda)}{\lambda}d\lambda- \int_0^s\frac{\sin(t\lambda)}{\lambda}\overline{A} (\overline{A}+\lambda^2\mathrm{Id})^{-1}d\lambda,
\eeqs
where the last integral exists because $\overline{A}$ is non-negative. Taking into account that the limit in (\ref{kks93}) exists in $\mathcal{L}_b(L^2(\RR^d))$, we let $s\rightarrow\infty$ in the above equality and infer
\beqs
T(t)=\mathrm{Id}-\frac{2}{\pi}\lim_{s\rightarrow\infty} \int_0^s\frac{\sin(t\lambda)}{\lambda} \overline{A}(\overline{A}+\lambda^2\mathrm{Id})^{-1}d\lambda,\,\, t>0,
\eeqs
where the limit exists in $\mathcal{L}_b(L^2(\RR^d))$. Observe that the equality also holds for $t=0$. For each $t\in[0,\infty)$, we denote $Q(t)=\mathrm{Id}-T(t)\in\mathcal{L}(L^2(\RR^d))$. By (\ref{equforequikno}), for each $\varphi\in\SSS^*(\RR^d)$ we have $\overline{A}(\overline{A}+\lambda^2\mathrm{Id})^{-1}\varphi= (b^{(1)}_{\lambda^2})^w\varphi+S^{(1)}_{\lambda^2}\varphi$ (we take $k=1$ since $z=1/2$). Propositions \ref{knoforcomp} and \ref{rkpropo117} imply that for each $\varphi\in\SSS^*(\RR^d)$, the mapping $\lambda\mapsto (1+\lambda^2)S^{(1)}_{\lambda^2}\varphi$, $\RR_+\rightarrow L^2(\RR^d)$, is continuous and bounded. Hence $\lambda\mapsto \lambda^{-1}\sin(t\lambda)S^{(1)}_{\lambda^2}\varphi$, $\RR_+\rightarrow L^2(\RR^d)$, is Bochner integrable. Similarly, Propositions \ref{continuity}, \ref{knoforcomp} and \ref{rkpropo117} imply that the mapping $\lambda\mapsto (1+\lambda^2)(b^{(1)}_{\lambda^2})^w\varphi$, $\RR_+\rightarrow L^2(\RR^d)$, is continuous and bounded and hence the mapping $\lambda\mapsto \lambda^{-1}\sin(t\lambda)(b^{(1)}_{\lambda^2})^w\varphi$, $\RR_+\rightarrow L^2(\RR^d)$, is Bochner integrable. Thus, for each $t\in[0,\infty)$ and $\varphi\in\SSS^*(\RR^d)$,
\beqs
Q(t)\varphi&=&\frac{2}{\pi} \int_0^{\infty}\frac{\sin(t\lambda)}{\lambda}(b^{(1)}_{\lambda^2})^w\varphi d\lambda+\frac{2}{\pi} \int_0^{\infty}\frac{\sin(t\lambda)}{\lambda}S^{(1)}_{\lambda^2}\varphi d\lambda\\
&=&Q_1(t)\varphi+Q_2(t)\varphi.
\eeqs
Let us prove that for each $t\in[0,\infty)$ and $\varphi\in\SSS^*(\RR^d)$, $Q_1(t)\varphi\in\SSS^*(\RR^d)$. This is trivial for $t=0$. For the moment, for $t>0$, denote
\beqs
\tilde{C}_t=\frac{2}{\pi}\int_0^{\infty} (\lambda(1+\lambda^2))^{-1}|\sin(t\lambda)|d\lambda>0.
\eeqs
Let $\varepsilon>0$. Since $\{(1+\lambda^2)(b^{(1)}_{\lambda^2})^w|\, \lambda>0\}$ is equicontinuous in $\mathcal{L}(\SSS^*(\RR^d),\SSS^*(\RR^d))$ (by Propositions \ref{continuity} and \ref{knoforcomp}), the set $\{(1+\lambda^2)\,\,{}^t((b^{(1)}_{\lambda^2})^w)|\, \lambda>0\}$ is equicontinuous in $\mathcal{L}(\SSS'^*(\RR^d),\SSS'^*(\RR^d))$. Hence there exists a neighbourhood of zero $W$ in $\SSS'^*(\RR^d)$ such that $|\langle (1+\lambda^2)\,\,{}^t((b^{(1)}_{\lambda^2})^w)f,\varphi\rangle|\leq \varepsilon/\tilde{C}_t$ for all $f\in W$. Hence, for $v\in W\cap L^2(\RR^d)$, by the properties of the Bochner integral, we have $|\langle v,Q_1(t)\varphi\rangle|\leq \varepsilon$. Thus $Q_1(t)\varphi$ is a continuous functional on $L^2(\RR^d)$ when the latter is equipped with the topology induced by $\SSS'^*(\RR^d)$. Hence $Q_1(t)\varphi$ is a continuous functional on $\SSS'^*(\RR^d)$, i.e. $Q_1(t)\varphi\in \SSS^*(\RR^d)$. Similarly, $Q_2(t)\varphi\in\SSS^*(\RR^d)$. We claim that for each $f\in\SSS'^*(\RR^d)$
\beq
\langle f,Q_1(t)\varphi\rangle&=&\frac{2}{\pi} \int_0^{\infty}\frac{\sin(t\lambda)}{\lambda}\langle f,(b^{(1)}_{\lambda^2})^w\varphi\rangle d\lambda\,\,\,\, \mbox{and} \label{stk1794}\\
\langle f,Q_2(t)\varphi\rangle&=&\frac{2}{\pi} \int_0^{\infty}\frac{\sin(t\lambda)}{\lambda}\langle f,S^{(1)}_{\lambda^2}\varphi\rangle d\lambda.
\eeq
These trivially hold for $f\in L^2(\RR^d)$ (by the properties of the Bochner integral). The general case follows from the sequential denseness of $L^2(\RR^d)$ in $\SSS'^*(\RR^d)$ and the fact that the sets $\{(1+\lambda^2)(b^{(1)}_{\lambda^2})^w|\, \lambda>0\}$ and $\{(1+\lambda^2)S^{(1)}_{\lambda^2}|\, \lambda>0\}$ are equicontinuous in $\mathcal{L}(\SSS^*(\RR^d),\SSS^*(\RR^d))$. Let $U$ be a neighbourhood of zero in $\SSS^*(\RR^d)$, which, by reflexivity, we can assume to be the polar $B'^{\circ}$ of a convex circled closed bounded subset $B'$ of $\SSS'^*(\RR^d)$. Since the set $\{(1+\lambda^2)\,\,{}^t((b^{(1)}_{\lambda^2})^w)|\, \lambda>0\}$ is equicontinuous in $\mathcal{L}(\SSS'^*(\RR^d),\SSS'^*(\RR^d))$, we conclude that $\bigcup_{\lambda>0}(1+\lambda^2)\,\,{}^t((b^{(1)}_{\lambda^2})^w)(B')$ is bounded in $\SSS'^*(\RR^d)$. Let $\tilde{B}$ be its closed convex circled hull. Then $V=(\tilde{C}_t\tilde{B})^{\circ}$ (where $\tilde{C}_t$ is the constant defined above) is a neighbourhood of zero in $\SSS^*(\RR^d)$. By (\ref{stk1794}), $Q_1(t)\varphi\in U$, for all $\varphi\in V$. Hence, for each $t\in\RR_+$, $\varphi\mapsto Q_1(t)\varphi$ is a continuous mapping from $\SSS^*(\RR^d)$ into itself. This trivially holds for $t=0$, since $Q_1(0)\varphi=0$, $\forall\varphi\in\SSS^*(\RR^d)$. Similarly, for each $t\in[0,\infty)$, $\varphi\mapsto Q_2(t)\varphi$ also belongs to $\mathcal{L}(\SSS^*(\RR^d),\SSS^*(\RR^d))$. We conclude that $T(t)\in\mathcal{L}(\SSS^*(\RR^d),\SSS^*(\RR^d))$, $t\geq0$.\\
\indent Let $t_0\in[0,\infty)$, $B$ a bounded subset of $\SSS^*(\RR^d)$ and $V=B'^{\circ}$ a neighbourhood of zero in $\SSS^*(\RR^d)$ (where $B'$ is convex circled closed bounded subset of $\SSS'^*(\RR^d)$). Consider the neighbourhood of zero $M(B,V)=\{J\in\mathcal{L}(\SSS^*(\RR^d),\SSS^*(\RR^d))|\, J(B)\subseteq V\}$ in $\mathcal{L}_b(\SSS^*(\RR^d),\SSS^*(\RR^d))$. There exists $C'>0$ such that $(1+\lambda^2)|\langle f,(b^{(1)}_{\lambda^2})^w\varphi\rangle|\leq C'$, for all $f\in B'$, $\varphi\in B$ (as $\{(1+\lambda^2)(b^{(1)}_{\lambda^2})^w|\, \lambda>0\}$ is equicontinuous in $\mathcal{L}(\SSS^*(\RR^d),\SSS^*(\RR^d))$). Thus, by (\ref{stk1794}),
\beqs
\sup_{\substack{f\in B'\\ \varphi\in B}}|\langle f,Q_1(t)\varphi-Q_1(t_0)\varphi\rangle|\leq \frac{2C'}{\pi}\int_0^{\infty}\frac{|\sin(t\lambda)-\sin(t_0\lambda)|}{\lambda(1+\lambda^2)}d\lambda.
\eeqs
Hence, there is $\varepsilon>0$ such that $Q_1(t)-Q_1(t_0)\in M(B,V)$, for all $t\in[0,\infty)\cap (t_0-\varepsilon,t_0+\varepsilon)$. As $t_0$ is arbitrary, the mapping $t\mapsto Q_1(t)$ is in $\mathcal{C}([0,\infty);\mathcal{L}_b(\SSS^*(\RR^d)\SSS^*(\RR^d)))$. Analogously, one can prove the same thing for the mapping $t\mapsto Q_2(t)$. Thus, the mapping $t\mapsto T(t)$ is in $\mathcal{C}([0,\infty);\mathcal{L}_b(\SSS^*(\RR^d),\SSS^*(\RR^d)))$. As a direct consequence, we deduce that, for each $s>0$, the set $\{T(t)|\, t\in[0,s]\}$ is bounded in $\mathcal{L}_b(\SSS^*(\RR^d),\SSS^*(\RR^d))$, consequently equicontinuous.\\
\indent Since $-\overline{A}^{1/2}$ is the infinitesimal generator of the analytic semigroup $T(t)$, we have
\beq\label{trh1753}
(d/dt)T(t)\varphi=-\overline{A}^{1/2}T(t)\varphi=-T(t)\overline{A}^{1/2}\varphi,\,\, t\geq 0,\,\, \varphi\in\SSS^*(\RR^d).
\eeq
Moreover, the mapping $t\mapsto (d/dt)T(t)\varphi=-T(t)\overline{A}^{1/2}\varphi$ is continuous from $[0,\infty)$ into $L^2(\RR^d)$ and
\beq\label{stk1745}
T(t)\varphi-T(t')\varphi=-\int_{t'}^tT(s)\overline{A}^{1/2}\varphi ds=-\int_{t'}^t\overline{A}^{1/2}T(s)\varphi ds.
\eeq
For each $t\in[0,\infty)$, we denote by $T_1(t)$ the mapping $\varphi\mapsto -T(t)\overline{A}^{1/2}\varphi$. As $-\overline{A}^{1/2}\in\mathcal{L}(\SSS^*(\RR^d),\SSS^*(\RR^d))$ (cf. Theorem \ref{maint}), $T_1(t)\in\mathcal{L}(\SSS^*(\RR^d),\SSS^*(\RR^d))$ and the mapping $t\mapsto T_1(t)$ is in $\mathcal{C}([0,\infty);\mathcal{L}_b(\SSS^*(\RR^d),\SSS^*(\RR^d)))$ (as $t\mapsto T(t)$ belongs to this space). Moreover, for each $f\in\SSS'^*(\RR^d)$, we have
\beq\label{rst1759}
\langle f,T(t)\varphi-T(t')\varphi\rangle=-\int_{t'}^t\langle f,T(s)\overline{A}^{1/2}\varphi\rangle ds.
\eeq
This trivially holds for $f\in L^2(\RR^d)$ by (\ref{stk1745}) and the properties of the Bochner integral. The general case follows since $L^2(\RR^d)$ is sequentially dense in $\SSS'^*(\RR^d)$ and, for each $t>0$, the set $\{T(s)\overline{A}^{1/2}|\, s\in[0,t]\}$ is equicontinuous in $\mathcal{L}(\SSS^*(\RR^d),\SSS^*(\RR^d))$. Since $\SSS^*(\RR^d)$ is the strong dual of $\SSS'^*(\RR^d)$, the equicontinuity of this set together with (\ref{rst1759}) implies that for each fixed $t_0\in[0,\infty)$ and $\varphi\in\SSS^*(\RR^d)$ the set
\beqs
H_{t_0,\varphi}=\{(t-t_0)^{-1}(T(t)\varphi-T(t_0)\varphi)|\, t\in[0,\infty)\cap([t_0-1,t_0+1]\backslash\{t_0\})\}
\eeqs
is bounded in $\SSS^*(\RR^d)$. Thus $H_{t_0,\varphi}$ is equicontinuous in $\mathcal{L}(\SSS'^*(\RR^d),\CC)$. Denoting by $\mathfrak{G}_1$ the family of all finite subsets of $L^2(\RR^d)$, its union is total in $\SSS'^*(\RR^d)$ and (\ref{trh1753}) implies that $(t-t_0)^{-1}(T(t)\varphi-T(t_0)\varphi)\rightarrow T_1(t_0)\varphi$ in $\mathcal{L}_{\mathfrak{G}_1}(\SSS'^*(\RR^d),\CC)$. Since $H_{t_0,\varphi}\cup\{T_1(t_0)\varphi\}$ is equicontinuous in $\mathcal{L}(\SSS'^*(\RR^d),\CC)$, the Banach-Steinhaus theorem \cite[Theorem 4.5, p. 85]{Sch} implies that the convergence also holds in $\mathcal{L}_p(\SSS'^*(\RR^d),\CC)$. As $\SSS'^*(\RR^d)$ is Montel, the convergence holds in $\mathcal{L}_b(\SSS'^*(\RR^d),\CC)=\SSS^*(\RR^d)$. Since $t_0\in[0,\infty)$ is arbitrary, we can conclude that, for each $\varphi\in\SSS^*(\RR^d)$, the mapping $t\mapsto T(t)\varphi$, $[0,\infty)\rightarrow \SSS^*(\RR^d)$, is differentiable and $T_1(t)\varphi$ is its derivative. Hence, we can apply Lemma \ref{ktn1793} $(ii)$ to deduce that the mapping $t\mapsto T(t)$ is differentiable and $T_1(t)$ is its derivative. As $t\mapsto T_1(t)\in\mathcal{C}([0,\infty);\mathcal{L}_b(\SSS^*(\RR^d),\SSS^*(\RR^d)))$ we conclude that $t\mapsto T(t)$ belongs to $\mathcal{C}^1([0,\infty);\mathcal{L}_b(\SSS^*(\RR^d),\SSS^*(\RR^d)))$. But, as $\partial_t T(t)=T_1(t)=-\overline{A}^{1/2}T(t)$, $t\mapsto T(t)$ is in fact a $\mathcal{C}^{\infty}$ mapping from $[0,\infty)$ to $\mathcal{L}_b(\SSS^*(\RR^d),\SSS^*(\RR^d))$.
\end{proof}

Now, as $T(t)$ solves (\ref{system3}) with $\tilde{\mathbf{K}}=0$, we have
\beq\label{rkv7591}
(\mathbf{u}(t))^w\varphi-T(t)\varphi=\int_0^t T(t-s)\tilde{\mathbf{K}}(s)\varphi ds,\,\, \varphi\in\SSS^*(\RR^d).
\eeq
By Theorem \ref{sol_heat_ker} and Lemma \ref{lmmsk11}, for each $t>0$, the mapping $s\mapsto T(t-s)\tilde{\mathbf{K}}(s)$ belongs to $\mathcal{C}^{\infty}([0,t];\mathcal{L}_b(\SSS'^*(\RR^d),\SSS^*(\RR^d)))$. For $t\in[0,\infty)$, $f\in\SSS'^*(\RR^d)$, define
\beqs
\mathbf{Q}(t)f=\int_0^t T(t-s)\tilde{\mathbf{K}}(s)f ds\in L^2(\RR^d).
\eeqs
One easily verifies (by analogous techniques as those employed in the proof of Lemma \ref{lmmsk11}) that $\mathbf{Q}(t)f\in\SSS^*(\RR^d)$ and
\beq\label{fhttr15}
\langle g,\mathbf{Q}(t)f\rangle=\int_0^t \langle g,T(t-s)\tilde{\mathbf{K}}(s)f\rangle ds,\,\,\, g\in\SSS'^*(\RR^d).
\eeq
Now, again by using analogous techniques as in the proof of Lemma \ref{lmmsk11}, one verifies that $f\mapsto \mathbf{Q}(t)f$ belongs to $\mathcal{L}(\SSS'^*(\RR^d),\SSS^*(\RR^d))$, for each $t\in[0,\infty)$. As a consequence of (\ref{fhttr15}) and the semigroup property of $T(t)$, one can derive the continuity of $t\mapsto \mathbf{Q}(t)$, $[0,\infty)\rightarrow \mathcal{L}_b(\SSS'^*(\RR^d),\SSS^*(\RR^d))$. This immediately implies that for each $t>0$ the set $\{\mathbf{Q}(s)|\, s\in[0,t]\}$ is equicontinuous in $\mathcal{L}(\SSS'^*(\RR^d),\SSS^*(\RR^d))$. Moreover, we have the following lemma.

\begin{lemma}\label{hks7951}
The mapping $t\mapsto \mathbf{Q}(t)$ belongs to $\mathcal{C}^{\infty}([0,\infty);\mathcal{L}_b(\SSS'^*(\RR^d),\SSS^*(\RR^d)))$.
\end{lemma}

\begin{proof} We prove that the derivative of $\mathbf{Q}(t)$ at $t_0\in[0,\infty)$ is the mapping
\beqs
f\mapsto \mathbf{Q}_1(t_0)f= \int_0^{t_0}(\partial_tT)(t_0-s)\tilde{\mathbf{K}}(s)fds+T(0)\tilde{\mathbf{K}}(t_0)f.
\eeqs
By analogous techniques as for the mapping $f\mapsto\mathbf{Q}(t)f$, one verifies that this is indeed a well defined element of $\mathcal{L}(\SSS'^*(\RR^d),\SSS^*(\RR^d))$ and the mapping $t\mapsto \mathbf{Q}_1(t)$, $[0,\infty)\rightarrow\mathcal{L}_b(\SSS'^*(\RR^d),\SSS^*(\RR^d))$, is continuous (cf. (\ref{trh1753})). Moreover, for each $g\in\SSS'^*(\RR^d)$,
\beq\label{tvk7593}
\langle g,\mathbf{Q}_1(t)f\rangle=\int_0^t\langle g,(\partial_tT)(t-s)\tilde{\mathbf{K}}(s)f\rangle ds+\langle g,T(0)\tilde{\mathbf{K}}(t)f\rangle.
\eeq
By employing (\ref{fhttr15}) and (\ref{tvk7593}) one easily verifies that, for fixed $f,g\in\SSS'^*(\RR^d)$,
\beqs
(t-t_0)^{-1}\langle g,\mathbf{Q}(t)f-\mathbf{Q}(t_0)f\rangle\rightarrow \langle g,\mathbf{Q}_1(t_0)f\rangle,\,\,\, \mbox{as}\,\, t\rightarrow t_0.
\eeqs
Thus, by Lemma \ref{ktn1793} $(ii)$, $t\mapsto\mathbf{Q}(t)$ is differentiable at $t_0$ with $\mathbf{Q}_1(t_0)$ being its derivative. Since $t_0$ is arbitrary and $t\mapsto\mathbf{Q}_1(t)$ is continuous, we infer that $t\mapsto \mathbf{Q}(t)$ is a $\mathcal{C}^1$ mapping. By induction, one verifies the claim in the lemma.
\end{proof}

Observe that all claims of Theorem \ref{res_sqr_rtt1} with the exception of (\ref{est1453}) follow from Theorem \ref{sol_heat_ker}, Lemmas \ref{lmmsk11} and \ref{hks7951} and the equality (\ref{rkv7591}). The estimate (\ref{est1453}) follows from the definition of $b$ together with (\ref{est_heat_par}), (\ref{est_semigroup}) for $\alpha=0$, (\ref{est_semigroup1}) and the first estimate in $(v)$ of Theorem \ref{maint} particularised for $\alpha=0$ and $N=1$. With this, the proof of Theorem \ref{res_sqr_rtt1} is complete.


\begin{thebibliography}{99}


\bibitem{ABP} M. F. Atiyah, R. Bott and V. K. Patodi, \emph{On the heat equation and the index theorem}, Invent. Math. 19 (1973), 279--330.

\bibitem{BN} E. Buzano and F. Nicola, \emph{Complex powers of hypoelliptic pseudodifferential operators}, J. Funct. Anal. 245 (2007), 353--378.

\bibitem{CPP} M. Cappiello, S. Pilipovi\'c and B. Prangoski, \emph{Parametrices and hypoellipticity for pseudodifferential operators on spaces of tempered ultradistributions}, J. Pseudo-Differ. Oper. Appl. 5(4) (2014), 491--506.

\bibitem{CPP1} M. Cappiello, S. Pilipovic and B. Prangoski, \emph{Semilinear pseudodifferential equations in spaces of tempered ultradistributions}, J. Math. Anal. Appl. 442(1) (2016), 317--338.

\bibitem{PilipovicK} R. Carmichael, A. Kami\'nski and S. Pilipovi\'c, \emph{Boundary Values and Convolution in Ultradistribution Spaces}, World Scientific Publishing Co. Pte. Ltd., 2007.

\bibitem{Fractionalpowersbook} C. M. Carracedo and M. S. Alix, \emph{The theory of fractional powers of operators}, Amsterdam, Elsevier, 2001.

\bibitem{faadib} G. M. Constantine and T. H. Savits, \emph{A multivariate Fa\`a di Bruno formula with applications}, Trans. Amer. Math. Soc. 348(2) (1996), 503--520.

\bibitem{CSS} S. Coriasco, E. Schrohe and J. Seiler, \emph{Bounded imaginary powers of differential operators on manifolds with conical singularities}, Math. Z. 244 (2003), 235--269.

\bibitem{DPPV} P. Dimovski, S. Pilipovi\'{c}, B. Prangoski and J. Vindas, \emph{Convolution of ultradistributions and ultradistribution spaces associated to translation-invariant Banach spaces}, Kyoto J. Math. 56(2) (2016), 401--440.

\bibitem{DV} G. Dore and A. Venni, \emph{On the closedness of the sum of two closed operators}, Math. Z. 196 (1987), 189--201.

\bibitem{Dustermaat} J. J. Duistermaat and V. W. Guillemin, \emph{The spectrum of positive elliptic operators and periodic bicharacteristics}, Invent. Math. 29 (1975), 39--79.

\bibitem{Helffer} B. Helffer, \emph{Th\' eorie spectrale pour des op\' erateurs globalement elliptiques}, Ast\' erisque 112, Soc. Math. de France, 1984.

\bibitem{Hor} L. H\" ormander, \emph{The Analysis of Linear Differential Operators}, I--IV, Berlin, Springer, 1983--1985.

\bibitem{KNonnegative} H. Komatsu, \emph{Fractional powers of operators, III Negative powers}, J. Math. Soc. Japan 21(2) (1969), 205--220.

\bibitem{Komatsu1} H. Komatsu, \emph{Ultradistributions, I: Structure theorems and a characterization}, J. Fac. Sci. Univ. Tokyo, Sect. IA Math. 20(1) (1973), 25--105.

\bibitem{Komatsu2} H. Komatsu, \emph{Ultradistributions, II: The kernel theorem and ultradistributions with support in submanifold}, J. Fac. Sci. Univ. Tokyo, Sect. IA Math. 24(3) (1977), 607--628.

\bibitem{Komatsu3} H. Komatsu, \emph{Ultradistributions, III: Vector valued ultradistributions and the theory of kernels}, J. Fac. Sci. Univ. Tokyo, Sect. IA Math. 29(3) (1982), 653--717.

\bibitem{Kumanogo} H. Kumano-go and C. Tsutsumi, \emph{Complex powers of hypoelliptic pseudo-differential operators with applications}, Osaka J. Math. 10 (1973), 147--174.

\bibitem{Loya1} P. Loya, \emph{The structure of the resolvent of elliptic pseudo-differential operators}, J. Funct. Anal. 184 (2001), 77--135.

\bibitem{Loya2} P. Loya, \emph{Complex powers of differential operators on manifolds with conical singularities}, J. Anal. Math. 89 (2003), 31--56.

\bibitem{Melrose} R. B. Melrose, \emph{The Atiyah-Patodi-Singer index theorem}, Research Notes in Mathematics 4, A K Peters Ltd., Wellesley, MA, 1993.

\bibitem{NR} F. Nicola and L. Rodino, \emph{Global Psedo-Differential Calculus on Euclidean Spaces}, Vol. 4. Birkh\" auser Basel, 2010.

\bibitem{PilipovicU} S. Pilipovi\'c, \emph{Tempered ultradistributions}, Boll. Un. Mat. Ital. 2(2) (1988), 235--251.

\bibitem{ppv} S. Pilipovi\'c, B. Prangoski and J. Vindas, \emph{Spectral asymptotics for infinite order pseudo-differential operators}, arXiv:1701.07907

\bibitem{ppv1} S. Pilipovi\'c, B. Prangoski and J. Vindas, \emph{Power series of Shubin type differential operators}, arXiv:1711.05628.

\bibitem{BojanP} B. Prangoski, \emph{Pseudodifferential operators of infinite order in spaces of tempered ultradistributions}, J. Pseudo-Differ. Oper. Appl. 4(4) (2013), 495--549.

\bibitem{Robert} D. Robert, \emph{Autour de l'approximmation semi-classique}, Birkh\" auser, Boston, MA, 1987.

\bibitem{Sch} H. H. Schaefer, \emph{Topological vector spaces}, Springer-Verlag, New York Heidelberg Berlin, 1970.

\bibitem{Schrohe1} E. Schrohe, \emph{Complex powers of elliptic pseudo-differential operators}, Integral Equations and Operator Theory 9 (1986), 337--354.

\bibitem{Schrohe1} E. Schrohe, \emph{Complex powers on non-compact manifolds and manifolds with singularities}, Math. Ann. 281 (1988), 393--409.

\bibitem{Seeley1} R. T. Seeley, \emph{Complex powers of an elliptic operator}, in Singular Integrals: Proc. Sympos. Pure Math. 10, Amer. Math. Soc., (1967), 288--307.

\bibitem{Shubin} M. A. Shubin, \emph{Pseudodifferential operators and spectral theory}, Springer- Verlag, Berlin, 1987.


\end{thebibliography}
\end{document}